\documentclass[reqno,11pt]{amsart} %amsproc/amsbook
\usepackage{esint,mathrsfs,amsmath,amssymb,amsfonts}%,,graphics,,bbm,amsthm,, ,
\usepackage{verbatim}
\usepackage{color}
\usepackage[hyperfootnotes=false]{hyperref}
\hypersetup{
  colorlinks,
  citecolor=blue,
  linkcolor=red,
  urlcolor=blue}
  
\usepackage[left=1.2in, right=1.2in, bottom=1.5in]{geometry}
\newtheorem{theorem}{Theorem}[section]
\newtheorem{lemma}[theorem]{Lemma}
\newtheorem{proposition}[theorem]{Proposition}

\theoremstyle{definition}

\theoremstyle{remark}
\newtheorem{remark}[theorem]{Remark}

\numberwithin{equation}{section}

\newcommand{\Ag}[1]{\langle#1\rangle}
\newcommand{\cL}{\mathcal{L}}

\newcommand{\Z}{\mathbb{Z}}

\newcommand{\R}{\mathbb{R}}

\newcommand{\e}{\varepsilon}

\newcommand{\dist}{\text{dist}}

\begin{document}

\title[Eigenvalue problems in perforated domains]{Convergence rates of eigenvalue problems in perforated domains: the case of small volume}

%    Only \author and \address are required; other information is
%    optional.  Remove any unused author tags.
%    author one information

\author{Zhongwei Shen}
\address{Zhongwei Shen: Department of Mathematics, University of Kentucky, Lexington, Kentucky 40506, USA.}
\email{zshen2@uky.edu}

\author{Jinping Zhuge}
\address{Jinping Zhuge: Morningside Center of Mathematics, Academy of Mathematics and systems science,
Chinese Academy of Sciences, Beijing, China.}
\email{jpzhuge@amss.ac.cn}

%\address{Department of Math, University of Kentucky, Lexington, KY, 40506, USA.}
%\curraddr{}
%\email{}
%\thanks{The author is supported in part by National Science Foundation grant DMS-XXXXXX.}
%\date{\today}

%    The 2020 edition of the Mathematics Subject Classification is
%    now available.  If you are citing a classification from the
%    new scheme, use the following input coding instead.
%\subjclass[2020]{Primary }

%    Text of article.
%    Input .tex files

%\input{s0.tex}
%\input{s1 - copy.tex}

\begin{abstract}
    
This paper is concerned with the Dirichlet eigenvalue problem for Laplace operator in a bounded domain with periodic perforation in the case of small volume.
We obtain the optimal quantitative error estimates independent of the spectral gaps for an asymptotic expansion, with two leading terms, of Dirichlet eigenvalues.
We also establish the convergence rates for the corresponding eigenfunctions.
Our approach uses a known reduction to a degenerate elliptic eigenvalue problem
for which a quantitative analysis is carried out.

\vspace{5pt}
\end{abstract}

\maketitle

\centerline{Dedicated to our teacher Professor Robert Fefferman on the occasion of his $70^{\rm th}$ birthday}

\tableofcontents

\section{Introduction}
\subsection{Motivations and main results}
The homogenization theory of partial differential equations in perforated domains has been extensively studied in the past four decades due to its theoretical and numerical importance. In this paper, we are concerned with the Dirichlet eigenvalue problem of Laplace operator in
periodically perforated domains: find $(\psi_{\eta,\e}, \lambda_{\eta,\e}) \in H^1_0(\Omega_{\eta,\e}) \times \R$ such that
%\footnote{In terms of scaling, one can alternatively consider the equation $-\e^2 \Delta \psi_\e = \lambda_\e \psi_\e$. In this equivalent (and more physically relevant) setting, each eigenvalue differs by a factor of $\e^{-2}$ and will remain bounded as $\e \to 0$.}
\begin{equation}\label{main eq.eigen}
    -\Delta \psi_{\eta,\e}  =\lambda_{\eta,\e} \psi_{\eta,\e} \quad \text{in } \Omega_{\eta,\e},
\end{equation}
where $\Omega_{\eta,\e}$ is a bounded and periodically perforated domain in $\mathbb{R}^d$, $d\ge 3$. This problem was introduced by J. Rauch \cite{R74} in a crashed ice problem (also see \cite{RT75}), wherein the first eigenvalue of \eqref{main eq.eigen} determines the rate of cooling for the evenly distributed crashed ice. Since $-\Delta$ is symmetric with compact resolvent in $L^2(\Omega_{\eta,\e})$, by the classical spectral theorem, the Dirichlet eigenvalues are positive, countable and can be listed in a nondecreasing order $\{ \lambda_{\eta,\e}^k: k =1,2,\dots \}$ (counted with multiplicity) such that $\lambda_{\eta,\e}^k \to \infty$ as $k\to \infty$. We denote by $\psi^k_{\eta,\e}$ the normalized eigenfunction corresponding to $\lambda^k_{\eta,\e}$. We are interested in the asymptotic behaviors of $\lambda_{\eta,\e}^k$ and $\psi_{\eta,\e}^k$ as $\e\to 0$.
%For related work on eigenvalue problems for elliptic equations with oscillating coefficients in a fixed domain without holes,we refer the reader to \cite{K79-1, K79-2, MV97, AP02, KLS13, Z20}.
%We mention that 
%there is a considerable amount of work focused on the eigenvalue problems for elliptic equations with oscillating coefficients in a fixed domain without holes (see, e.g., \cite{K79-1, K79-2, MV97, AP02, KLS13, Z20}).
%or on the Poisson equation $-\Delta u_{\eta,\e} = f$ in a perforated domain $\Omega_{\eta,\e}$ with a fixed function $f$ (see, e.g., \cite{CP79,L80,Jing20,S22,BW22}).

The perforated domain $\Omega_{\eta,\e}$ is defined as follows. Let $0<\eta, \e \le 1$ be two small parameters. Let $\tau^i$ (with $i=1,2,\cdots,m$) be $C^{1,1}$ connected domains with diameters comparable to 1. Define $Y = [0,1]^d$. Let $\tau^i_\eta \subset Y$ be a translated copy of $\eta \tau^i$ satisfying $\dist(\tau^i_\eta, \tau^j_\eta) > c_0$ for $i\neq j$ and $\dist(\tau^i_\eta, \partial Y)>c_0$. Let $T_\eta = \cup_{i=1}^m \overline{\tau^i_\eta} \subset Y$. In other words, $T_\eta$ is the union of $m$ holes contained in $Y$. Define
\begin{equation}
    T_{\eta,\e} = \bigcup_{z\in \Z^d} \e (z + T_\eta), \qquad \Omega_{\eta,\e} = \Omega \setminus T_{\eta,\e}.
\end{equation}
Note that the diameters of the holes in $\Omega_{\eta,\e}$ are roughly $\eta \e$ and thus the holes only occupy a small portion (roughly $O(\eta^d)$) of the volume of $\Omega$. 

We will impose a geometric assumption $\mathbf{A}$: for each hole $\e(z + \tau^i_\eta)$ of $T_{\eta,\e}$, either the hole has at least a distance $c_0 \e \eta$ away from $\partial \Omega$, or $\partial \Omega$ intersects with the hole such that both $\Omega \cap \e(z + \tau^i_\eta)$ and $\Omega \setminus \e(z + \tau^i_\eta)$ (in a neighborhood of $\e(z + \tau^i_\eta)$) are Lipschitz domains with a uniform Lipschitz character after rescaling. Let $\Gamma_{\eta,\e} = \partial \Omega \cap \partial \Omega_{\eta,\e}$ and $\Sigma_{\eta,\e} = \Omega \cap \partial T_{\eta,\e}$.

Under the above assumptions with $\eta = 1$, the asymptotic behavior of the Dirichlet eigenvalues was first studied by Vanninathan \cite{V81} and then the convergence rates were obtained by Ole\u{\i}nik, Shamaev and Yosifian \cite{OSY92}. We first describe the procedure of reduction for the case $\eta \ll 1$ which is the same as $\eta = 1$ given in \cite{V81}. Let $Y_{\eta} = Y\setminus T_\eta$, viewed as a flat torus with holes. Let $( \phi_{\eta},\overline{\lambda}_\eta) \in H^1_{0,{\rm per}}(Y_{\eta}) \times \R$ denote the first eigenvalue and eigenfunction of the cell problem,
\begin{equation}\label{eq.Phieta}
    -\Delta \phi_{\eta}  =\overline{\lambda}_\eta \phi_{\eta}  \quad  \text{in}\ \  Y_{\eta},
\end{equation}
with $\|\phi_{\eta} \|_{L^2(Y_{\eta})} =1$ and $\phi_\eta = 0$ in $\partial T_\eta$. The principal eigenfunction $\phi_\eta$ is nonnegative and satisfies $\phi_{\eta}(x) \approx \min\{ \eta^{-1} \text{dist}(x,T_\eta), 1\}$, which means that $\phi_{\eta}$ degenerates proportional to a distance function near the holes.
%We may extend $\phi_{\eta}$ to the entire $Y$ by zero.
%Note that if $\eta=1$, $(\phi_{\eta},\overline{\lambda}_\eta )=( \phi,\overline{\lambda})$.
Let
$\phi_{\eta,\e} (x) = \phi_{\eta} (x/\e)$.
As in the case $\eta=1$, we have
\begin{equation}\label{eq.EV2scale}
\lambda_{\eta,\e}^k =\e^{-2} \overline{\lambda}_\eta + \mu_{\eta, \e}^k,
\end{equation}
where $\{ \mu_{\eta, \e}^k\}$ denote the Dirichlet eigenvalues of the degenerate problem
\begin{equation}\label{eq.DEVP}
-\nabla \cdot (
\phi_{\eta, \e}^2 \nabla \rho_{\eta,\e}^k ) = \mu_{\eta,\e}^k \phi_{\eta, \e}^2 \rho_{\eta,\e}^k, \quad \text{in } \Omega_{\eta, \e},
\end{equation}
with corresponding eigenfunctions $\{ \rho_{\eta,\e}^k \} \subset H^1_{\phi_{\eta,\e}, 0}(\Omega_{\eta,\e})$. 
The eigenvalues $\mu_{\eta,\e}^k$ can be given by the minimax principle, 
\begin{equation}
\mu_{\eta, \e}^k =\min_{\substack{S\subset H^1_{\phi_{\eta,\e}, 0}(\Omega_{\eta,\e})\\  \text{dim} S = k} } \max_{u \in S} \frac{\int_{\Omega_{\eta, \e}} \phi_{\eta, \e}^2 |\nabla u|^2 }{\int_{\Omega_{\eta, \e} } \phi_{\eta, \e}^2 |u|^2 }.
\end{equation}
Thus, to obtain the asymptotic behaviors of $\lambda_{\eta,\e}^k$, it suffices to study those of $\mu_{\eta,\e}^k$. More details can be found in Section \ref{sec.2}.

For the case $\eta = 1$, it was shown in \cite{OSY92} that $|\mu_{1,\e}^k - \mu_{1}^k| \lesssim \e$, where $\mu_{1}^k$ are the eigenvalues of
\begin{equation}\label{eq.homo.eta=1}
    -\nabla\cdot (\overline{A}_1 \nabla \rho_{1}^k) = \mu_1^k \rho_1^k \quad \text{ in } \Omega,
\end{equation}
and $\overline{A}_1$ is a constant coefficient matrix satisfying the ellipticity condition. In view of \eqref{eq.EV2scale}, we have
\begin{equation}
    |\lambda_{1,\e}^k - \e^{-2} \overline{\lambda}_1 - \mu_1^k| \lesssim \e.
\end{equation}
Note that the homogenized eigenvalue problem \eqref{eq.homo.eta=1} is posed in the domain $\Omega$ without holes, which is well-understood. The convergence rates of $O(\e)$ for the eigenfunctions were also obtained in \cite{OSY92}. We point out that the convergence rate of order $O(\e)$ obtained in \cite{OSY92} is optimal in the case of $\eta = 1$. However, due to several technical reasons, the implicit constants obtained via their approach is not quantitative (uncomputable) and depends essentially on the spectral gaps (appearing in the denominator).
%, which are almost impossible to quantify in general; see analogs in recent references \cite{Z20,D22,AV22}.

For the small holes with $\eta \ll 1$, a special case with $d=3$ and $\eta = \e^2$ was also studied in \cite{OSY92}. It was shown that as $\e \to 0$, the eigenvalues of \eqref{eq.DEVP} converge to the Dirichlet eigenvalues of $-\Delta$ in $\Omega$. More precisely, for any $k\ge 1$,
\begin{equation}\label{est.OSY.small}
    | \lambda_{\eta,\e}^k -\overline{\xi} - \xi_0^k | \lesssim \e,
\end{equation}
where $\xi_0^k$ is the $k$th Dirichlet eigenvalue of $-\Delta$ in $\Omega$, or equivalently, $\bar{\xi} + \xi_0^k$ is the eigenvalue of $-\Delta + \bar{\xi}$, and $\bar{\xi}>0$ is a fixed number independent of $\e$.

In this paper, we revisit \eqref{main eq.eigen} with holes of arbitrary size, particularly including the case of small volume. We will show the optimal quantitative error estimate of eigenvalues under a refined asymptotic expansion.
%The main objective of this paper is to establish the optimal quantitative convergence rates of the eigenvalues, improving the previous results in \cite{OSY92}.
The following is the main result of this paper.
\begin{theorem}\label{thm.main-EV}
    Let $\Omega$ be a smooth ($C^3$) domain. Assume that $\Omega_{\eta,\e}$ satisfies the geometric assumption \textbf{A}.
    Then
    \begin{equation}\label{est.MainEV}
        |\lambda_{\eta,\e}^k -\e^{-2} \overline{\lambda}_\eta - \mu_\eta^k| \le C_k \e \eta^{\frac{d-2}{2}}, 
    \end{equation}
    where $\mu_\eta^k$ is the $k$th Dirichlet eigenvalue (counted with multiplicity) for $-\nabla\cdot (\overline{A}_\eta \nabla)$ in $\Omega$, i.e., the eigenpair $(\rho^k_\eta, \mu^k_\eta) \in H^1_0(\Omega) \times \R$ satisfies
    \begin{equation}\label{eq.MainHomo-1}
        -\nabla\cdot (\overline{A}_\eta \nabla \rho^k_{\eta}) = \mu^k_\eta \rho^k_\eta, \quad \text{in } \Omega.
    \end{equation}
    Moreover, $\overline{A}_\eta$ is a constant coefficient matrix satisfying a uniformly elliptic conidtion independent of $\eta$, and $C_k$ depends only on $k,d, c_0$ and the geometric characters of the domain (including $\Omega$ and $\tau^i$).
\end{theorem}

Compared with the results in \cite{OSY92}, our result in Theorem \ref{thm.main-EV} has two major improvements. 
\begin{itemize}
    \item[(1)] In the case $\eta\ll 1$, our result shows that, compared to $-\Delta$, the refined operator $-\nabla\cdot (\overline{A}_{\eta} \nabla)$ is a more accurate effective operator. In particular, we may compare the result of the special case $d=3$ and $\eta=\e^2$ with that of \cite{OSY92} mentioned before.
    In this case, \eqref{est.MainEV} is reduced to
    \begin{equation}\label{est.3D.opt}
    | \lambda_{\e, \eta}^k -\e^{-2} \overline{\lambda}_\eta - \mu_{\eta}^k | \le C_k \e^2,
\end{equation}
where $\overline{\lambda}_\eta \approx \eta^{d-2} = \e^2$. Definitely, the error in \eqref{est.3D.opt} is much sharper than \eqref{est.OSY.small}.
Moreover, it can be shown that $|\overline{A}_{\eta} - I| \le C\eta^\frac{d-2}{2} = C\e$ and therefore $|\mu_{\eta}^k - \xi_0^k| \le C_k \eta^\frac{d-2}{2} =C_k\e$, which explains why the improvement is reasonable by using the refined operator $-\nabla\cdot (\overline{A}_{\eta} \nabla)$ instead of $-\Delta$. The refined operator takes advantage of both the homogenization process at large scales, leading to the factor $\e$, and the perturbation due to the perforation of small $\eta$ at small scales, leading to the factor $\eta^\frac{d-2}{2}$ in \eqref{est.MainEV}. We should point out that, from the point of view of numerical computation, the calculation for $\overline{A}_\eta$ for a fixed $\eta$ is a one-time cost for solving a cell problem similar to $\eta = 1$ which makes it practical in applications.

\item[(2)] The constant $C_k$ in Theorem \ref{thm.main-EV} and Theorem \ref{thm.main-EF} below is quantitative in the sense that its dependence on $k$ and the geometric characters can be determined. In particular, $C_k$ is bounded by $C_0 k^a$ for some computable $a = a(d)>0$, where $C_0$ is independent of $k$. However, the constants given by \cite{OSY92} are not fully quantitative and may rely on the spectral gaps.
\end{itemize}

Our second result is concerned with  the convergence rates of eigenfunctions, which are typically more difficult due to the multiplicity of eigenvalues or the small spectral gaps between them; see \cite{Z20,AV22,D22} for related problems. 
% For example, a small change of an eigenvalue will lead to a completely orthogonal eigenfunction,
% and as a result,  a convergence rate of eigenvalues does not imply a convergence rate for their corresponding eigenfunctions (unless the dimension of the corresponding eigenspace is one). Even if the eigenvalue is simple, the convergence rate of eigenfunctions depends critically on the spectral gap, which is almost impossible to quantify in general; see analogs in recent references \cite{Z20,D22,AV22}. 
As such, it seems natural to consider the asymptotic behaviors of an eigenspace corresponding to a collection of eigenvalues in a narrow spectral band.

To describe the result, recall that $\{ \rho_{\eta}^k \}$ are the normalized eigenfunctions of \eqref{eq.MainHomo} corresponding to $\{ \mu_{\eta}^k \}$. For $\theta>0, t>0$, define
\begin{equation*}
    S_\eta(\theta;t) = \text{span} \{ \rho_{\eta}^k: |\mu_{\eta}^k -\theta| \le t \}.
\end{equation*}
In other words, $S_\eta(\theta;t)$ is the eigenspace of the homogenized eigenvalue problem \eqref{eq.MainHomo} corresponding to the spectrum in the band $[\theta-t, \theta+t]$.

\begin{theorem}\label{thm.main-EF}
    Under the same conditions of Theorem \ref{thm.main-EV}, for any $k\ge 1$ and $t>0$,
    \begin{equation}\label{est.main-EF}
        \inf_{v \in S_\eta(\mu_{\eta}^k;t) } \| \psi_{\eta,\e}^k - \phi_{\eta,\e} v \|_{L^2(\Omega_{\eta,\e})} \le C_k (\e^\frac12 \eta^{\frac{d-2}{2}} \wedge \e + t^{-1} \e \eta^{\frac{d-2}{2}}),
    \end{equation}
    where $a\wedge b = \min\{a,b \}$.
\end{theorem}

Roughly speaking, the above theorem states that an eigenfunction of \eqref{main eq.eigen} corresponding to the eigenvalue $\lambda_\e^k = \e^{-2} \overline{\lambda} + \mu_{\eta,\e}^k$ can be well approximated by $\phi_{\eta,\e} u_{\eta}$, where $\phi_{\eta,\e} (x)=\phi_\eta(x/\e)$ and
$u_{\eta}$ is a finite linear combination of the eigenfunctions of the homogenized operator corresponding to the eigenvalues in a narrow band (with width $2t$) around $\mu_{\eta}^k$. 
{The main novelty here is that one is  free to choose the width of the band and the estimate is independent of the spectral gap near $\mu_\eta^k$.} In particular, if $\mu_\eta^k$ is simple and has large spectral gaps from $\mu_\eta^{k-1}$ and $\mu_\eta^{k+1}$ (e.g., the principal eigenvalue $\mu_\eta^1$), then $\psi_\e^k$ converges to $\phi_{\eta,\e} \rho_\eta^k$ in $L^2(\Omega_{\eta,\e})$ with an error of order $O(\e^\frac12 \eta^{\frac{d-2}{2}} \wedge \e)$. This rate may not be optimal in general which is caused by the suboptimal estimates of boundary layers (two different approaches will be used). The optimal boundary layer estimate in a perforated domain is an interesting and challenging problem that is beyond the scope of this paper.
%If the width of the band is chosen too small, say $t< \min\{ \e, \e^\frac12 \eta^\frac{d-2}{2} \}$, then the approximation is invalid. This is consistent with the convergence rates of the eigenvalues, as $\mu_{\eta,\e}^k$ may converge to a limit that does not belong to the $\e \eta^\frac{d-2}{2}$-neighborhood of $\mu_{\eta,\e}^k$. 
On the other hand, in case of small $t$,
the negative power $t^{-1}$ in \eqref{est.main-EF} cannot be avoided generally and the dominating term $t^{-1} \e \eta^{\frac{d-2}{2}}$ (say $t<\e^\frac12$) should be optimal.

%. However, the error bound $\min\{ \e, \e^\frac12 \eta^\frac{d-2}{2} \}$ is not optimal due to some technical reasons. The optimal convergence rate of order $O(\e \eta^\frac{d-2}{2})$ can be shown under a weaker norm or if $\Omega$ satisfies a more restrictive assumption; see Remark ??.

%converges to zero with a rate $O(\e \eta^\frac{d-2}{2}) \text{Gap}(\mu_\eta^k)^{-1}$.
%In view of the convergence rate $|\mu_{\eta,\e}^k - \mu_\eta^k| \le O(\e \eta^\frac{d-2}{2})$, one may replace $\mu_{\eta,\e}^k$ in \eqref{est.main-EF} by $\mu_\eta^k$.

% Let $0< \eta\le 1$.
% In this section we consider the case where $T$ is replaced by $T_\eta$ and study the asymptotic behavior of Dirichlet eigenvalues and eigenfunctions 
% for $-\Delta$ in $\Omega_{\eta, \e}$,
% \begin{equation}
%         \left\{
%         \aligned
%         -\Delta u  & =\lambda u  & \quad & \text{ in } \Omega_{\eta, \e},\\
%         u& = 0 & \quad & \text{ on } \partial \Omega_{\eta, \e}, 
%         \endaligned
%         \right.
%     \end{equation}
% as $\e \to 0$ and $\eta\to 0$, where
% \begin{equation}\label{O-eta}
%     \Omega_{\eta, \e} =\Omega \setminus \bigcup_z \e (z+ T_\eta).
% \end{equation}
\subsection{Strategy of the proof}

First of all, we point out that the optimal upper bounds of the eigenvalues actually can be proved directly by the minimax principle. However, the optimal lower bounds of the eigenvalues are much harder and we develop a new approach to resolve the problem.

The proofs of our main results (as well as the proof in \cite{OSY92} for $\eta = 1$) can be roughly divided into two stages. In the first stage, we need to obtain an a priori quantitative convergence rate (not optimal) of eigenvalues  in dependent of the spectral gaps. By this and the classical Weyl's law for the homogenized eigenvalue problem, we are able to locate some common large spectral gaps between eigenvalues (see Lemma \ref{lem.N1gap}), which will play a key role in the second stage.
We mention that in \cite{OSY92} only a qualitative convergence was proved as their first stage and hence their results are not fully quantitative. Unlike the other homogenization problems in domains without holes (see, e.g., \cite{K79-1,KLS13}), our proof for the degenerate problems in perforated domains involves an intermediate problems as a bridge. 

In fact, in order to study the eigenvalue problem \eqref{eq.DEVP}, we need to consider the boundary value problem,
\begin{equation}\label{eq.MainDegeneate}
    -\nabla\cdot (\phi_{\eta,\e}^2 \nabla u_{\eta,\e}) = \phi_{\eta,\e}^2 f \qquad \text{in } \Omega_{\eta,\e}
    \quad \text{ and } \quad u_{\eta,\e} =0 \quad \text{ on } \partial \Omega,
\end{equation}
and the corresponding homogenized problem is given by
\begin{equation}\label{eq.MainHomo}
    -\nabla\cdot (\overline{A}_{\eta} \nabla u_{\eta}) =  f \qquad \text{in } \Omega
    \quad \text{ and } \quad u_{\eta} =0 \quad \text{ on }  \partial \Omega.
\end{equation}
Unfortunately, the convergence from $\phi_{\eta,\e}^2 f$ to $f$ is very weak if $f\in L^p(\Omega)$ and therefore it is impossible to establish a quantitative convergence rate from $u_{\eta,\e}$ to $u_{\eta}$ under the assumption 
$f\in L^p(\Omega)$ for some $p\le \infty$, which is a necessary regularity assumption in the application of minimax principle to \eqref{eq.DEVP}. To handle this difficulty, we introduce an intermediate eigenvalue problem,
\begin{equation}\label{main.intermediate}
    -\nabla\cdot (\overline{A}_\eta \nabla \tilde{\rho}_{\eta,\e}^k) = \tilde{\mu}_{\eta,\e}^k {\phi}_{\eta,\e}^2 \tilde{\rho}_{\eta,\e}^k  \quad \text{ in } \Omega
    \quad \text{ and } \quad \tilde{\rho}_{\eta,\e}^k=0 \quad \text{ on } \partial \Omega,
\end{equation}
and the corresponding equation
\begin{equation}\label{eq.MainInter}
    -\nabla\cdot (\overline{A}_\eta \nabla \tilde{u}_{\eta,\e}) = \phi_{\eta,\e}^2 f \qquad \text{in } \Omega
    \quad \text{ and } \quad \tilde{u}_{\eta,\e} =0 \quad \text{ on }  \partial \Omega.
\end{equation}
Observe that \eqref{eq.MainInter} only has an oscillating factor $\phi_{\eta,\e}^2$ on the right-hand side. Thus, the solution $\tilde{u}_{\eta,\e}$, as well as the eigenfunctions in \eqref{main.intermediate}, admits better regularity (indeed, $\tilde{u}_{\eta,\e} \in W^{2,p}(\Omega)$, provided $f\in L^p(\Omega_{\eta,\e})$).
Now, we consider two steps of convergence via the intermediate problem: the convergence rate of the solutions from \eqref{eq.MainDegeneate} to \eqref{eq.MainInter} can be established under the condition $f\in L^p(\Omega_{\eta,\e})$ for some $p>d$ (see Theorem \ref{thm.L2rate}); while the convergence rate from \eqref{eq.MainInter} to \eqref{eq.MainHomo} can only be shown under the stronger condition $f\in W^{1,p}(\Omega)$ for some $p>d$ (see Theorem \ref{thm.tT0-T0}). The regularity condition on $f$ in either step is compatible with the a priori regularity of the corresponding eigenfunctions and thus allows us to use the minimax/maximin principle to derive the suboptimal convergence rates of eigenvalues (see Proposition \ref{prop.mue<mu0+} and Proposition \ref{prop.mue>mu0-}). We emphasize that in order to extract the small factor $\eta^\frac{d-2}{2}$ in the convergence rates, we have to estimate the weight $\phi_\eta$ and correctors $\chi_\eta$ optimally (see Section \ref{sec.2}) and carry out a careful analysis throughout the entire proof (see Section \ref{sec.4}).

%Step 1: Establish a \textbf{suboptimal} convergence rate for the eigenvalues. The constant should not depend on the spectral gaps and but depends quantitatively on $k$. We need this only in order to get the distribution of eigenvalues so that we can find the common large spectral gaps by the Weyl's law. Tools needed includes harmonic extension and minimax principle. 

In the second stage, we prove the optimal convergence rates of eigenvalues (and the convergence rates of eigenfunctions as a byproduct). To this end, we first prove an improved convergence rate (still not optimal) directly from \eqref{eq.MainDegeneate} to \eqref{eq.MainHomo} without relying on the intermediate problem \eqref{eq.MainInter} under the stronger assumption $f\in W^{1,p}(\Omega)$ for some $p>d$, i.e.,
\begin{equation}
    \| \mathscr{T}_{\eta,\e} f - \mathscr{T}_{\eta} f \|_{L^2_{\phi_{\eta,\e}}(\Omega_{\eta,\e})} \le C\e^\frac12 \eta^{\frac{d-2}{2}} \| f \|_{W^{1,p}(\Omega)},
\end{equation}
where $\mathscr{T}_{\eta,\e}: f \to u_{\eta,\e}$ and $\mathscr{T}_{\eta}: f \to u_\eta$ are the resolvents given by \eqref{eq.MainDegeneate} and \eqref{eq.MainHomo}, respectively.
Then, we apply a duality argument to show the optimal convergence rate in the weak sense, i.e., for any $g\in W^{1,p}(\Omega)$,
\begin{equation}
    \bigg| \int_{\Omega_{\eta,\e}} \phi_{\eta,\e}^2 (\mathscr{T}_{\eta,\e} f - \mathscr{T}_{\eta} f) g \bigg| \le C\e \eta^{\frac{d-2}{2}} \| f \|_{W^{1,p}(\Omega)} \|g \|_{W^{1,p}(\Omega)}.
\end{equation}
Combining the above two estimates with the large spectral gaps found in the first stage and the almost orthogonality of eigenfunctions (see Lemma \ref{lem.OptAlmostOrtho}), we are able to analyze precisely the dimension of the eigenspace of $\mathscr{T}_{\eta,\e}$ corresponding to the eigenvalues less than $\mu^k_\eta + C_k \e \eta^\frac{d-2}{2}$ for each $k\ge 1$, which implies the optimal lower bounds of the eigenvalues of $\mathscr{T}_{\eta,\e}$. This argument relies essentially on the orthogonal structure of linear eigenspaces.

\textbf{Acknowledgements.} Z.S. is partially supported by the NSF grant DMS-2153585.  J. Z. is partial supported by the NSFC (No. 12288201) and a grant for Excellent Youth from the NSFC.

\section{Preliminaries}\label{sec.2}
Throughout this paper, the constants $C$ and $c$ will never depend on $\e$ and $\eta$. But they may depend on the geometric characters of $\Omega$ and $T$,  the number of holes $m$ in $Y$, etc. We write $A \approx B$ if there exist $C \ge c > 0$ such that $c B \le A \le CB$. We write $A \le B + O(\delta)$ if $A\le B + C\delta$ for some constant $C$ independent of $\delta$; we write $A = B + O(\delta)$ if $|A-B| \le C\delta$. Given an arbitrary $Y$-periodic function $f$, we will apply the notation $f_\e(x) = (f)_\e(x) = f(x/\e)$.

\subsection{The cell eigenvalue problem}
This subsection is devoted to the estimates of the eigenpair $(\phi_{\eta}, \overline{\lambda}_\eta)$.
%Unless stated otherwise, all bounding constants $C$ and $c$ in this section are independent of $\e$ and $\eta$. 

%\jz{The upper bounded can be improved to $\e \eta^{\frac{d-2}{2}}$. The lower bound can be improved if we replace $\Delta$ by a homogenized operator depending on $\eta$, i.e, $\overline{A}_\eta = I + O(\eta^\frac{d-2}{2})$. }
Recall that $Y_{\eta}=Y\setminus T_\eta$.
Define
\begin{equation*}
    H^1_{\rm 0, per}(Y_{\eta}) = \{ v\in H^1(Y_{\eta}): \text{$v$ is $Y$-periodic and $v=0$ on $\partial T_\eta$} \}.
\end{equation*}
Here and after, a function defined in $Y_{\eta}$ is said to be $Y$-periodic if it has the same trace on the opposite faces of $Y$. Thus the functions in $H^1_{\rm 0, per}(Y_{\eta})$ can be identified as $Y$-periodic functions in $\R^d$ with no jumps across the cell boundaries.

Consider the eigenvalue problem
\begin{equation}\label{eq.cell}
    -\Delta u = \lambda u \quad \text{in } Y_{\eta},
\end{equation}
with $(u,\lambda)\in H^1_{\rm 0, per}(Y_{\eta}) \times \R$, where $Y_{\eta}$ should be understood as a flat torus with holes (i.e., the equation continues to hold across $\partial Y$). It is well-known that the principal eigenvalue of the above cell problem, denoted by $\overline{\lambda}$, is positive and simple, and the corresponding eigenfunction, denoted by $\phi_{\eta}$, does not change sign in $Y_{\eta}$. Without loss of generality, we assume
\begin{equation}\label{cond.phi}
    \| \phi_{\eta} \|_{L^2(Y_{\eta})} = 1 \quad \text{ and } \quad \phi_{\eta} > 0 \quad \text{in } Y_{\eta}.
\end{equation}
% If the holes in $T_\eta$ also have smooth boundaries (the uniform interior ball condition and $C^{1,\alpha}$ will suffice), we have 
% \begin{equation}\label{est.nondegeneracy}
%     \phi_{\eta}(x) \approx \dist(x,T_\eta) \quad \text{for all } x\in Y_{\eta}.
% \end{equation}
% Moreover, $|\nabla \phi_{\eta}| \le C$.
For convenience, we extend $\phi_{\eta}$ by zero in the interior of 
$T_\eta$ and denote it still by $\phi_{\eta}$.

Since $\phi_{\eta}$ is $Y$-periodic, by rescaling, we see that $\phi_{\eta,\e}(x) := \phi_{\eta}(x/\e)$ satisfies
\begin{equation}\label{eq.principle eigen}
\left\{
    \begin{aligned}
    -\Delta \phi_{\eta,\e} & = \e^{-2} \overline{\lambda}_\eta \phi_{\eta,\e}  &\quad & \text{in } \R^d \setminus T_{\eta,\e}, \\
    \phi_{\eta,\e} & = 0 & \quad &  \text{on } \partial T_{\eta,\e}. 
    \end{aligned}
    \right.
\end{equation}
%Note that $\phi_{\eta,\e} \ge c>0$ on $\partial \Omega$ due to the geometric assumption $H$.

\begin{lemma}\label{lem.lambda.bar}
    There exist positive constants $c_1$ and $C_1$ such that
    \begin{equation}\label{7.1-0}
        c_1 \eta^{d-2} \le \overline{\lambda}_\eta \le C_1 \eta^{d-2} 
    \end{equation}
    for $\eta\in (0, 1)$.
\end{lemma} 

\begin{proof}
    To show $\overline{\lambda}_\eta \ge c_1 \eta^{d-2}$, we use the Poincar\'{e} inequality,
    \begin{equation}\label{7.1-1}
        c_1 \eta^{d-2} \int_{Y_{\eta}} |u|^2\le \int_{Y_{\eta}} |\nabla u|^2,
    \end{equation}
    for functions $u\in H^1_{0,{\rm per}}(Y_{\eta})$, where $c_1>0$ depends only on $T_1$ {(see \cite{Allaire-1990} for a proof)}.
    This gives
    $$
    \overline{\lambda}_\eta \int_{Y_{\eta}} | \phi_{\eta}| ^2 =\int_{Y_{\eta}} |\nabla \phi_{\eta}|^2 \ge c_0 \eta^{d-2} \int_{Y_{\eta}} |\phi_{\eta}|^2,
    $$
    which yields $\overline{\lambda}_\eta \ge c_1 \eta^{d-2}$.

    To see $\overline{\lambda}_\eta\le C_1 \eta^{d-2}$, we use the fact
    $$
    \overline{\lambda}_\eta =\inf_{u\in H^1_{0,{\rm per}}(Y_{\eta})} \frac{\int_{Y_{\eta}} |\nabla u|^2}{ \int_{Y_{\eta}} |u|^2}.
    $$
%where the inf is taken over all 1-periodic functions $u\in H^1(Y)$ such that $ u=0$ in $T_\eta$.
%Assume $\eta$ is sufficiently small so that $T_\eta \subset B(0, C\eta)\subset B(0, 2C \eta)\subset B(0, 1/4)$.
Let $T_{\eta}^t = \{ x\in Y: \dist(x,T_\eta) \le t \}.$ Let $c = \frac14 c_0$.
It is not hard to construct a function $u\in H^1_{0,{\rm per}}(Y_{\eta})$ such that $0\le u \le 1$, $u=1$ in $Y\setminus T_{\eta}^{c\eta}$, and $|\nabla u| \le C\eta^{-1}$ in $T_{\eta}^{c\eta} \setminus T_\eta$. It follows that
$$
\int_{Y_{\eta}} |\nabla u|^2  \le C\eta^{-2} |T_{\eta}^{c\eta} \setminus T_\eta| \le C\eta^{d-2}.
$$
Also, $\| u\|_{L^2(Y_{\eta})} \approx 1$. This implies that $\overline{\lambda}_\eta\le C_1 \eta^{d-2}$.
\end{proof}

% It follows from Lemma \ref{lemma-7.1} that 
% \begin{equation}\label{g-estimate}
% c\eta^{\frac{d-2}{2}} \le \|\nabla \phi_{\eta,\e}\|_{L^2(Y)} \le C \eta^{\frac{d-2}{2}}.
% \end{equation}

\begin{lemma}\label{lemma-7.2} 

There exist positive constants $c_1$ and $C_1$ such that the following estimates hold.
\begin{itemize}
    \item[(i)] Norm estimate:
    \begin{equation}\label{7.2-2}
    \| \nabla \phi_{\eta} \|_{L^2(Y)} + \| \phi_{\eta} -1 \|_{L^{2^*}(Y)} \le C_1 \eta^{\frac{d-2}{2}}.
\end{equation}
where $2^*=\frac{2d}{d-2}$.

    \item[(ii)] Gradient estimate: For any $x\in Y_{\eta}$,
    \begin{equation}
        |\nabla \phi_{\eta}(x)| \le C_1 \eta^{-1}  \Big\{ \Big( \frac{\eta}{{\rm dist}(x, T_\eta)} \Big)^{\frac{d}{2}} \wedge 1 \Big\}.
    \end{equation}

    \item[(iii)] Degeneracy estimate: For any $x\in Y_{\eta}$,
    \begin{equation}\label{7.2-1}
    c_1 \Big\{ 1 \wedge \frac{{\rm dist}(x, T_\eta)}{\eta} \Big\} \le \phi_{\eta}(x) \le C_1 \Big\{ 1 \wedge \frac{{\rm dist}(x, T_\eta)}{\eta} \Big\}.
\end{equation}

    \item[(iv)] Interior estimate: For any $x\in Y_{\eta}$,
    \begin{equation}
        |\phi_{\eta}(x) - 1| \le C_1 \Big\{ \Big( \frac{\eta
        }{{\rm dist}(x, T_\eta)} \Big)^{\frac{d-2}{2}} \wedge 1  \Big\}.
    \end{equation}
\end{itemize}

% \begin{equation}\label{7.2-0}
%     c_1 \le \phi_{\eta,\e} (x)  \le C_1 \quad \text{ if } x \in Y \text{ and } \text{dist} (x, T_\eta) \ge \eta,
% \end{equation}
% and
% \begin{equation}\label{7.2-1}
%     c_1 \eta^{-1} \text{dist}(x, T_\eta)
%     \le \phi_{\eta,\e} (x) \le C_1 \eta^{-1} \text{dist} (x, T_\eta) \quad  
% \end{equation}
% if $ x\in Y \text{ and } 0< \text{dist}(x, T_\eta) \le \eta$.

\end{lemma}

\begin{proof}
    First of all, the eigenvalue equation \eqref{eq.Phieta} gives $\| \nabla \phi_{\eta}\|_{L^2(Y)}^2 = \overline{\lambda}_\eta \| \phi_{\eta} \|_{L^2(Y)}^2 \approx \eta^{d-2}$, by Lemma \ref{lem.lambda.bar}. This proves the first part of (i). The remaining estimates will be divided into several steps.

    Step 1: Establish (ii) and the upper bounds in (iii). To do this, we use the observation that if
    $-\Delta u=\lambda u$ in $B_{2r}(x_0)$ and $|\lambda| r^2\le 1$, then
    \begin{equation}\label{7.2-3}
        \max_{B_r(x_0)} |u| \le C \left(\fint_{B_{2r}(x_0)} |u|^2 \right)^{1/2},
    \end{equation}
    and
    \begin{equation}\label{7.2-4}
    \max_{B_{r}(x_0)} |\nabla u|
    \le C \left(\fint_{B_{2r}(x_0)} |\nabla u|^2 \right)^{1/2}
    + C |\lambda | r \left(\fint_{B_{2r}(x_0)} |u|^2 \right)^{1/2}.
\end{equation}
The estimates above follow from the standard elliptic estimates for $-\Delta u=F$  by an iteration argument.
We note that the estimates above continue to hold if we replace $B_{r}(x_0) $ by
$B_{r}(x_0) \cap Y_{\eta}$, where $x_0 \in \partial T_\eta$, $0< r\le c\eta$, and
$-\Delta u=\lambda u$ in $Y_{\eta}$,  $u=0$ on $\partial T_\eta$.
As a result, we see that if $x\in Y_{\eta}$ and dist$(x, T_\eta) \le \eta$, then applying \eqref{7.2-4} to $\phi_{\eta}$ gives
\begin{equation}\label{est.Dphi-b}
|\nabla \phi_{\eta} (x)|
  \le C \eta^{-\frac{d}{2}} \|\nabla \phi_{\eta} \|_{L^2(Y)} + C \eta^{d-1} \eta^{-\frac{d}{2} } \| \phi_{\eta} \|_{L^2(Y)}
  \le C \eta^{-1},
\end{equation}
where we have used the fact $\|\nabla \phi_{\eta} \|^2_{L^2(Y)} =\overline{\lambda}_\eta \approx \eta^{d-2}$.
Since $\phi_{\eta,\e} =0$ on $\partial T_\eta$, this implies that
\begin{equation}\label{est.Phi.upp-b}
\phi_{\eta} (x) \le C \eta^{-1} \text{dist}(x, T_\eta),
\end{equation}
if $x\in Y_{\eta}$ and $\text{dist}(x, T_\eta) \le \eta$.
A similar argument shows that if $x\in Y$ and $ r=\dist(x, T_\eta) \ge \eta$, then
\begin{equation}\label{est.Dphi.i}
|\nabla \phi_{\eta} (x) |
\le C r^{-\frac{d}{2}} \eta^{\frac{d}{2} -1} + C \eta^{d-2} r^{1-\frac{d}{2}} \le C_1 \eta^{-1}  (\eta/r)^{\frac{d}{2}}.
\end{equation}
Note that \eqref{est.Dphi-b} and \eqref{est.Dphi.i} together leads to (ii). Moreover,
it follows by integration that $ \phi_{\eta} (x) \le C$ for $\dist(x,T_\eta) \ge \eta$. This together with \eqref{est.Phi.upp-b} gives the upper bound in (iii).

 Step 2: Prove the second part of (i). Let 
    $
    L_\eta =\fint_Y \phi_{\eta}.
    $
    Then
    $$
    \int_Y |\phi_{\eta} -L_\eta|^2 \le C \int_Y |\nabla \phi_{\eta}|^2 \le C \eta^{d-2}.
    $$
    It follows that
    $$
    \int_Y | (\phi_{\eta})^2 -L_\eta^2|
    =\int_Y |\phi_{\eta} +L_\eta| |\phi_{\eta} -L_\eta| \le C \int_Y |\phi_{\eta} -L_\eta| \le C \eta^{\frac{d-2}{2}},
    $$
    where we have used the upper bound $\phi_{\eta} \le C$ in Step 1.
    Since $\int_Y (\phi_{\eta})^2 =1$, we see that 
    $$
    |L_\eta^2 -1| \le C \eta^{\frac{d-2}{2}}.
    $$
    This implies that $|L_\eta -1|\le C \eta^{\frac{d-2}{2}}$, by the positivity of $L_\eta$.
    Thus,
    $$
    \aligned
    \|\phi_{\eta} -1\|_{L^p(Y)}
     & \le \| \phi_{\eta} -L_\eta\|_{L^p(Y)} + |L_\eta-1|\\
     & \le C \|\nabla \phi_{\eta} \|_{L^2(Y)} + |L_\eta -1| 
    \le C \eta^{\frac{d-2}{2}},
    \endaligned
    $$
where $p=\frac{2d}{d-2}$ (for $d>2$).

    Step 3: Prove the lower bounds in (iii). Since $T_\eta$ is a union of finite holes $\tau^i_\eta$, it suffices to consider the behavior of $\phi_{\eta}$ near an arbitrary hole. Let $x_1 \in Y$ (view $Y$ as a flat torus). We first claim that there exists $C_* >0$ such that
    \begin{equation}\label{est.lowbd.phieta}
        \sup_{x\in B_{C_* \eta}(x_1)} \phi_{\eta}(x) \ge \frac12.
    \end{equation}
    Actually, if the above inequality is not true, then
    \begin{equation}
        \| \phi_{\eta} - 1 \|_{L^{2^*}(Y)} \ge \frac12 |B_{C_* \eta}(x_1)|^{\frac{1}{2^*}} \ge \frac{1}{2}|B_1| (C_* \eta)^\frac{d-2}{2}.
    \end{equation}
    This contradicts to \eqref{7.2-2} if we choose $C_*$ such that $\frac12 |B_1| C_*^{\frac{d-2}{2}} > C_1$. Hence the claim is proved.

    %The claim shows that there exists at least one point $x_1 \in B_{C_* \eta}(x_1)$ such that $\phi_{\eta,\e}(x_2) \ge \frac12$. On the other hand, the upper bound in Step 1 indicates that if $\dist(x_2,T_\eta) < \frac{1}{2C_1} \eta$, then $\phi_{\eta,\e}(x_2) < \frac12$. Consequently, for arbitrary $x_1 \in Y$, we must have $x_2 \in B_{C_* \eta}(x_1) $ and $\dist(x_2,T_\eta) \ge \frac{1}{2C_1} \eta$ such that $\phi_{\eta,\e}(x_2) \ge \frac12$.

    We next employ a standard lifting technique to reduce the nonnegative eigenfunction $\phi_{\eta}$ to a nonnegative harmonic function so that the Harnack inequality applies. Define $h(x,s) = \exp(\overline{\lambda}_\eta^\frac12 s) \phi_{\eta}(x)$. Then $(\Delta + \partial_s^2) h =0$ in $Y_{\eta} \times (-1,1)$ and $h = 0$ on $\partial T_\eta \times (-1,1)$. Now, if $B_{2C_* \eta}(x) \cap T_\eta = \varnothing$, by the Harnack inequality for the nonnegative harmonic function $h$ in $Y_{\eta} \times (-1,1)$ and the previous result,
    \begin{equation}
        \inf_{B_{ C_* \eta}(x)} \phi_{\eta} \ge \exp(-\overline{\lambda}_\eta^\frac12 )\inf_{B_{ C_* \eta}(x,0)} h \ge c \sup_{B_{ C_* \eta}(x,0)} h \ge c \sup_{B_{ C_* \eta}(x)} \phi_{\eta} \ge c_1,
    \end{equation}
    where we have used the fact $h(x,s) \approx \phi_{\eta}(x)$ for any $(x,s) \in Y_{\eta} \times (-1,1)$ and \eqref{est.lowbd.phieta}. This shows that $\phi_{\eta}(x) \ge c_1$ for any $x$ with $\dist(x,T_\eta) \ge 2C_* \eta$.

    Now if $x\in Y_\eta$ with $\dist(x,T_\eta) < 2C_* \eta$, then by a standard barrier function argument and the lower bound for $\dist(x,T_\eta) \ge 2C_* \eta$, we have
    \begin{equation}
        h(x,0) \approx \phi_{\eta}(x) \ge c_1 \eta^{-1} \dist(x, T_\eta).
    \end{equation}
    This can be shown by considering each hole $\tau_{\eta}^i$ and the boundary behavior of $h$ near the $C^{1,1}$ boundary $\partial \tau_{\eta}^i \times (-1,1)$. The lower bound in \eqref{7.2-1} then follows.

    Step 4: Prove (iv). We view $\phi_{\eta} - 1$ as a solution of
    \begin{equation}
        -\Delta (\phi_{\eta} - 1) = \overline{\lambda}_\eta \phi_{\eta}.
    \end{equation}
    Then for $r = \dist(x,T_\eta) > \eta$, we apply the interior estimate to get
    \begin{equation}
        |\phi_{\eta}(x) - 1| \le C\bigg( \fint_{B_r(x)} |\phi_{\eta} - 1|^{2^*} \bigg)^{1/2^*} + C\overline{\lambda}_\eta r \| \phi_{\eta} \|_{L^\infty(B_r(x))} \le C\Big( \frac{\eta}{r} \Big)^{\frac{d-2}{2}},
    \end{equation}
    where we have used \eqref{7.2-2} and \eqref{7.2-1} in the last inequality.
    Finally for $\dist(x,T_\eta) \le \eta$, \eqref{7.2-1} gives the desired bounded.
\end{proof}

\subsection{Weighted Sobolev spaces}
Let 
$$
L^2_{\phi_\eta}(Y_{\eta}) = \left\{ v\in L^1_{\rm loc}(Y_{\eta}): v\phi_\eta \in L^2(Y_{\eta}) \right\},
$$
$$
L^2_{\phi_{\eta,\e}}(\Omega_{\eta,\e}) = \left\{ v\in L^1_{\rm loc}(\Omega_{\eta,\e}): v\phi_{\eta,\e} \in L^2(\Omega_{\eta,\e})\right \}.
$$
Let $H^1_{\phi_\eta,{\rm per}}(Y_{\eta})$ be the periodic $\phi_\eta$-weighted Sobolev space defined by
\begin{equation*}
    H^1_{\phi_\eta,{\rm per}}(Y_{\eta}) = \big\{ v\in L^1_{\rm loc}(Y_{\eta}): v \in L^2_{\phi_\eta}(Y_{\eta}), \nabla v \in L^2_{\phi_\eta}(Y_{\eta})^d, \text{ and $v$ is $Y$-periodic} \big\}.
\end{equation*}
Similarly, define the $\phi_{\eta,\e}$-weighted Sobolev space in $\Omega_{\eta,\e}$ by
\begin{equation*}
    H^1_{\phi_{\eta,\e},0}(\Omega_{\eta,\e}) = \big\{ v\in L^1_{\rm loc}(\Omega_{\eta,\e}):  v \in L^2_{\phi_{\eta,\e}}(\Omega_{\eta,\e}), \nabla v \in L^2_{\phi_{\eta,\e}}(\Omega_{\eta,\e})^d, \text{ and $v = 0$ on $\Gamma_{\eta,\e}$} \big\}.
\end{equation*}
For simplicity, from now on, we will denote $H^1_{\phi_{\eta,\e},0}(\Omega_{\eta,\e})$ by $V_{\eta,\e}$.

The corresponding norms of the above spaces are given by
\begin{equation*}
    \| v\|_{H^1_{\phi_\eta,{\rm per}}(Y_{\eta})} := \| \phi_\eta v \|_{L^2(Y_{\eta})} + \| \phi_\eta \nabla v \|_{L^2(Y_{\eta})},
\end{equation*}
and
\begin{equation*}
    \| v \|_{V_{\eta,\e}} =  \| v\|_{H^1_{\phi_{\eta,\e},0}(\Omega_{\eta,\e})} := \| \phi_{\eta,\e} v \|_{L^2(\Omega_{\eta,\e})} + \| \phi_{\eta,\e} \nabla v \|_{L^2(\Omega_{\eta,\e})}.
\end{equation*}

%\jz{Alternative choices for notations: $H^1_{\rm per}(Y_{\eta}; \phi)$ and $H^1_{0}(\Omega_{\eta,\e}; \phi_{\eta,\e})$}

%The analysis in the weighted Sobolev space $H^1_{\phi_{\eta,\e},0}(\Omega_{\eta,\e})$ is the key of the present paper. 
Let us first list some properties of the weighted Sobolev spaces $V_{\eta,\e}$ and $H^1_{\phi,{\rm per}}(Y_{\eta})$. These properties guarantee the solvability of the degenerate equation associated with the operator $\cL_{\eta,\e}$ and the validity of the classical spectral theorem for compact self-adjoint operators in separable Hilbert spaces.

\begin{proposition}\label{prop.H1e}
    The following properties regarding the space $V_{\eta,\e}$ hold:
    \begin{itemize}
        \item[(i)] The space $C_0^\infty(\Omega_{\eta,\e})$ is dense in $V_{\eta,\e}$.

        \item[(ii)] The inclusion $V_{\eta,\e} \to L^2(\Omega_{\eta,\e})$ is continuous uniformly in $\e$ and $\eta$. That is, there exists a constant $C>0$ such that
        \begin{equation*}
            \|v \|_{L^2(\Omega_{\eta,\e})} \le C\| v\|_{V_{\eta,\e}}, \quad \text{for all } v\in V_{\eta,\e}.
        \end{equation*}

        \item[(iii)] The map $v\to \phi_{\eta,\e} v$ is an isomorphism from $V_{\eta,\e}$ to $H^1_0(\Omega_{\eta,\e})$.

        \item[(iv)] The weighted Poincar\'{e} inequality holds: there exists a constant $C$ independent of $\e$ and $\eta$ such that
        \begin{equation*}
            \| \phi_{\eta,\e} v \|_{L^2(\Omega_{\eta,\e})} \le C\| \phi_{\eta,\e} \nabla v \|_{L^2(\Omega_{\eta,\e})}, \quad \text{for all } v\in V_{\eta,\e}.
        \end{equation*}
        \item[(v)] The inclusion $V_{\eta,\e} \to L^2_{\phi_{\eta,\e}}(\Omega_{\eta,\e})$ is compact.
    \end{itemize}
\end{proposition}

The proof of the qualitative properties (i), (iii) and (iv) in Proposition \ref{prop.H1e} can be found in \cite[Proposition 5.1]{V81} and the references therein. The quantitative estimates independent of $\e$ and $\eta$ are special cases of Theorem \ref{thm.unweigthed.eta} and Theorem \ref{thm.WSPI.eta} proved in Section \ref{Sec.3}.

% \begin{remark}
%     The constant $C$ in the previous proposition may depend on the domain $\Omega$ and the geometry of the holes in $T \subset Y$. The point is that it is independent of $\e$.
% \end{remark}

%Later to define the correctors, we also need a few properties for the perodic weighted space $H^1_{\phi,{\rm per}}(Y_{\eta})$.

\begin{proposition}\label{Prop.Hper}
    The following properties regarding the space $H^1_{\phi_\eta,{\rm per}}(Y_{\eta})$ hold:
    \begin{itemize}
        \item[(i)]  The space $C^\infty_{0,{\rm per}}(Y_{\eta})$, consisting of $C^\infty$ 
        periodic functions that are vanishing in a neighborhood of $T_\eta$, is dense in $H^1_{\phi_\eta,{\rm per}}(Y_{\eta})$.

        \item[(ii)] The inclusion $H^1_{\phi_\eta,{\rm per}}(Y_{\eta}) \to L^2(Y_{\eta})$ is continuous. That is, there exists a constant $C>0$ independent of $\eta$ such that
        \begin{equation*}
            \| v\|_{L^2(Y_{\eta})} \le C\| v\|_{H^1_{\phi_\eta,{\rm per}}(Y_{\eta})}, \quad \text{for all } v\in H^1_{\phi_\eta,{\rm per}}(Y_{\eta}).
        \end{equation*}

        \item[(iii)] The map $v\to \phi_\eta v$ is an isomorphism from $H^1_{\phi_\eta,{\rm per}}(Y_{\eta})$ to $H^1_{0,{\rm per}}(Y_{\eta})$.
        \item[(iv)] The inclusion $H^1_{\phi_\eta,{\rm per}}(Y_{\eta}) \to L^2_{\phi_\eta}(Y_{\eta})$ is compact.
    \end{itemize}
\end{proposition}

The proof of Proposition \ref{Prop.Hper} can be found in \cite[Proposition 5.2]{V81} and the references therein. The independence of $\eta$ of the constant $C$ in (ii) can be seen by the Hardy's inequality applied in a cell as in \eqref{est.Hardy.eta}.

\subsection{Two-scale expansion}\label{sec.2.3}
One of the main tools to study the asymptotic behavior of the eigenvalues is the minimax principle (or the maximin principle). The minimax principle for the eigenvalue problem \eqref{main eq.eigen} states that
\begin{equation}\label{eq.minimax.Re}
    \lambda_{\eta,\e}^k = \min_{\substack{S\subset H^1_0(\Omega_{\eta,\e})\\  \text{dim} S = k} } \max_{u \in S} \frac{\int_{\Omega_{\eta,\e}} |\nabla u|^2 }{\int_{\Omega_{\eta,\e}} u^2}.
\end{equation}

The following lemma, first discovered in \cite{V81}, provides the two-scale relationship between the eigenvalues of \eqref{main eq.eigen} and the principal eigenvalue of \eqref{eq.cell}.
\begin{lemma}
    For all $v\in H^1_{\phi_{\eta,\e},0}(\Omega_{\eta,\e})$, we have
    \begin{equation}\label{eq.int.expansion}
        \int_{\Omega_{\eta,\e}} \nabla(\phi_{\eta,\e} v)\cdot \nabla(\phi_{\eta,\e} v) = \e^{-2} \overline{\lambda}_\eta \int_{\Omega_{\eta,\e}} \phi_{\eta,\e}^2 v^2  + \int_{\Omega_{\eta,\e}} \phi_{\eta,\e}^2 |\nabla v|^2.
    \end{equation}
\end{lemma}
See \cite[Lemma 6.1]{V81} for a proof of the above lemma.

Dividing $\int_{\Omega_{\eta,\e}} \phi_{\eta,\e}^2 v^2$ on both sides of \eqref{eq.int.expansion}, we get
\begin{equation*}
    \frac{\int_{\Omega_{\eta,\e}} \nabla(\phi_{\eta,\e} v)\cdot \nabla(\phi_{\eta,\e} v)}{\int_{\Omega_{\eta,\e}} \phi_{\eta,\e}^2 v^2} = \e^{-2} \overline{\lambda}_\eta + \frac{ \int_{\Omega_{\eta,\e}} \phi_{\eta,\e}^2 |\nabla v|^2}{\int_{\Omega_{\eta,\e}} \phi_{\eta,\e}^2 v^2}.
\end{equation*}
By Proposition \ref{prop.H1e} (iii) and the minimax principle \eqref{eq.minimax.Re}, we see
\begin{equation}\label{eq.2scale relation}
\begin{aligned}
    \lambda_{\eta,\e}^k & = \min_{\substack{S\subset H^1_{\phi_{\eta,\e}, 0}(\Omega_{\eta,\e})\\  \text{dim} S = k} } \max_{v \in S} \frac{\int_{\Omega_{\eta,\e}} \nabla(\phi_{\eta,\e} v)\cdot \nabla(\phi_{\eta,\e} v)}{\int_{\Omega_{\eta,\e}} \phi_{\eta,\e}^2 v^2} \\
    & = \e^{-2} \overline{\lambda}_{\eta} + \min_{\substack{S\subset H^1_{\phi_{\eta,\e},0}(\Omega_{\eta,\e})\\  \text{dim} S = k} } \max_{v \in S} \frac{ \int_{\Omega_{\eta,\e}} \phi_{\eta,\e}^2 |\nabla v|^2}{\int_{\Omega_{\eta,\e}} \phi_{\eta,\e}^2 v^2}.
\end{aligned}
\end{equation}
Observe that the last term on the right-hand side is the Rayleigh quotient in the weighted Sobolev space $V_{\eta,\e}$, which leads to the following eigenvalue problem
\begin{equation}\label{eq.eigen.weighted}
    -\nabla\cdot (\phi^2_{\eta,\e} \nabla \rho_{\eta,\e}^k) = \mu_{\eta,\e}^k \phi_{\eta,\e}^2 \rho_{\eta,\e}^k \quad \text{ in } \Omega_{\eta,\e}.
\end{equation}

To see that this eigenvalue problem is well-posed, we first verify that the problem
\begin{equation}\label{eq.weighted elliptic}
    \cL_{\eta,\e}(u_{\eta,\e}) := -\nabla\cdot (\phi^2_{\eta,\e} \nabla u_{\eta,\e}) = \phi_{\eta,\e}^2 f \quad \text{in } \Omega_{\eta,\e}
\end{equation}
has a unique weak solution $u_{\eta,\e} \in V_{\eta,\e}$ for any $f\in L^2_{\phi_{\eta,\e}}(\Omega_{\eta,\e})$, in the sense that for each $v\in H^1_{\phi_{\eta,\e},0}(\Omega_{\eta,\e})$, we have
\begin{equation}\label{eq.weaksol}
    \int_{\Omega_{\eta,\e}} \phi_{\eta,\e}^2 \nabla u_{\eta,\e} \cdot \nabla v = \int_{\Omega_{\eta,\e}} \phi_{\eta,\e}^2 f v.
\end{equation}
This follows easily from the weighted Poincar\'{e} inequality in Proposition \ref{prop.H1e} (iv) and the Lax-Milgram theorem in the Hilbert space $V_{\eta,\e}$. Note that if the right-hand side of \eqref{eq.weighted elliptic} is replaced by a divergence form $\nabla\cdot (\phi_{\eta,\e}^2 g)$ with $g\in L^2_{\phi_{\eta,\e}}(\Omega_{\eta,\e})^d$, then the equation is still solvable by the Lax-Milgram theorem, since $\nabla\cdot (\phi_{\eta,\e}^2 g)$ is in the dual space of $V_{\eta,\e}$.

Also by Proposition \ref{prop.H1e} (v), we see that the linear operator $\mathscr{T}_{\eta,\e}: f \to u_{\eta,\e}$ is compact and self-ajoint in $L^2_{\phi_{\eta,\e}}(\Omega_{\eta,\e})$. Thus the classical spectral theory applies and there exists a sequence of positive eigenvalues $\{\mu_{\eta,\e}^k: k = 1,2,\cdots\}$ of the problem \eqref{eq.eigen.weighted} such that $\mu_{\eta,\e}^k\to \infty$ as $k\to \infty$. Moreover, the minimax principle gives
\begin{equation}\label{eq.minimax1}
    \mu_{\eta,\e}^k = \min_{\substack{S\subset V_{\eta,\e}\\  \text{dim} S = k} } \max_{v \in S} \frac{ \int_{\Omega_{\eta,\e}} \phi_{\eta,\e}^2 |\nabla v|^2}{\int_{\Omega_{\eta,\e}} \phi_{\eta,\e}^2 v^2}.
\end{equation}
Alternatively, the maximin principle  for the compact operator $\mathscr{T}_{\eta,\e}$ implies
\begin{equation}\label{eq.minimax2}
    \frac{1}{\mu^k_{\eta,\e}} = \max_{\substack{S\subset L^2_{\phi_{\eta,\e}}(\Omega_{\eta,\e}) \\ \text{dim} S = k }} \min_{v\in S} \frac{\int_{\Omega_{\eta,\e}}\phi_{\eta,\e}^2  \mathscr{T}_{\eta,\e}(v) \cdot v }{\int_{\Omega_{\eta,\e}} \phi_{\eta,\e}^2v^2 }.
\end{equation}

Now, we rigorously show the two-scale relationship between the eigenfunctions of \eqref{main eq.eigen} and \eqref{eq.eigen.weighted}.

\begin{proposition}\label{prop.2scale EF}
    $\psi_{\eta,\e}^k$ is an eigenfunction of \eqref{main eq.eigen} corresponding to the eigenvalue $\lambda_{\eta,\e}^k$ if and only if there exists an eigenfunction ${\rho}_\e^k$ of 
    \eqref{eq.eigen.weighted} corresponding to $\mu_{\eta,\e}^k$ such that $\psi_{\eta,\e}^k = \phi_{\eta,\e} \rho_{\eta,\e}^k$ and $\lambda_{\eta,\e}^k = \e^{-2} \overline{\lambda}_\eta + \mu_{\eta,\e}^k$.
\end{proposition}
\begin{proof}
    Let $\psi_{\eta,\e}^k$ be an eigenfunction of \eqref{main eq.eigen} corresponding to the eigenvalue $\lambda_{\eta,\e}^k$. Since $\phi_{\eta,\e} > 0$ in $\Omega_{\eta,\e}$, we may write $\psi_{\eta,\e}^k = \phi_{\eta,\e} u_{\eta,\e}^k$. Since $\psi_{\eta,\e}^k \in H_0^1(\Omega_{\eta,\e})$, then it is easy to verify that $u_{\eta,\e}^k \in H^1_{\phi_{\eta,\e},0}(\Omega_{\eta,\e})$ (also see Proposition \ref{prop.H1e} (iii)). Now we directly compute
    \begin{equation*}
        -\Delta \psi_{\eta,\e}^k = -\Delta \phi_{\eta,\e} u_{\eta,\e}^k - 2\nabla \phi_{\eta,\e}\cdot \nabla u_{\eta,\e}^k - \phi_{\eta,\e} \Delta u_{\eta,\e}^k = \e^{-2} \overline{\lambda}_{\eta} \phi_{\eta,\e} u_{\eta,\e}^k - 2\nabla \phi_{\eta,\e}\cdot \nabla u_{\eta,\e}^k - \phi_{\eta,\e} \Delta u_{\eta,\e}^k,
    \end{equation*}
    where we have used the equation \eqref{eq.principle eigen}. By $-\Delta \psi_{\eta,\e}^k = \lambda_{\eta,\e}^k \psi_{\eta,\e}^k$, we have
    \begin{equation*}
        - 2\nabla \phi_{\eta,\e}\cdot \nabla u_{\eta,\e}^k - \phi_{\eta,\e} \Delta u_{\eta,\e}^k = (\lambda_{\eta,\e}^k - \e^{-2} \overline{\lambda}_{\eta}) \phi_{\eta,\e} u_{\eta,\e}^k.
    \end{equation*}
    Multiplying the above equation by $\phi_{\eta,\e}$ and observing $\nabla\cdot (\phi_{\eta,\e}^2 \nabla u_{\eta,\e}^k) = 2\phi_{\eta,\e} \nabla \phi_{\eta,\e}\cdot \nabla u_{\eta,\e}^k + \phi_{\eta,\e}^2 \Delta u_{\eta,\e}^k$, we obtain 
    \begin{equation*}
        -\nabla\cdot (\phi_{\eta,\e}^2 \nabla u_{\eta,\e}^k) = (\lambda_{\eta,\e}^k - \e^{-2} \overline{\lambda}_{\eta}) \phi_{\eta,\e}^2 u_{\eta,\e}^k.
    \end{equation*}
    This equation shows that $\rho_{\eta,\e}^k = u_{\eta,\e}^k$ is an eigenfunction of \eqref{eq.eigen.weighted} corresponding to the eigenvalue $\mu_{\eta,\e}^k = \lambda_{\eta,\e}^k - \e^{-2} \overline{\lambda}_{\eta}$. The proof for the other direction is similar and straightforward.
\end{proof}

%where $\perp$ stands for the orthogonality in $L^2_{\phi_{\eta,\e}}(\Omega_{\eta,\e})$.

% As a corollary of \eqref{eq.2scale relation} and \eqref{eq.minimax1}, we have
% \begin{corollary}\label{coro.1stOrder}
%     Let $\{ \lambda_{\eta,\e}^k \}$ and $\{ \mu_{\eta,\e}^k \}$ be the nondecreasing sequences of eigenvalues of \eqref{main eq.eigen} and \eqref{eq.eigen.weighted}, respectively. Then for any $k\ge 1$, we have $\lambda_{\eta,\e}^k = \e^{-2} \overline{\lambda} + \mu^k_{\eta,\e}$. Moreover, $\mu^k_{\eta,\e}$ is bounded  uniformly in $\e$ for each $k\ge 1$.
% \end{corollary}

% The proof of the uniform boundedness of $\mu_{\eta,\e}^k$ can be found in \cite[Proposition 6.1]{V81}.

\subsection{Correctors and homogenized operators}

By virtue of Proposition \ref{prop.2scale EF}, it suffices to study the eigenvalue problem \eqref{eq.eigen.weighted} and the weighted or degenerated elliptic equation \eqref{eq.weighted elliptic}. For each fixed $\eta \in (0,1]$, it can be shown as in \cite{V81} that as $\e \to 0$, the eigenvalue $\mu_{\eta,\e}^k$ converges to some $\mu_{\eta}^k$, which is the $k$th eigenvalue of
\begin{equation}
    \cL_\eta (\rho_{\eta}^k) := -\nabla\cdot (\overline{A}_\eta \nabla \rho_{\eta}^k) = \mu_{\eta}^k \rho_{\eta}^k, \quad \text{ in } \Omega,
\end{equation}
where $\rho_{\eta}^k \in H_0^1(\Omega)$ and $\overline{A}_\eta$ is a constant matrix satisfying the ellipticity condition uniformly in $\eta$. 

% Although the formula,
% $\lambda^k_\e =\e^{-2} \overline{\lambda} + \mu^k_\eta + O(\e)$,
% was stated in the Abstract of \cite{V81},
%  no proof was given for the quantitative convergence rate. We emphasize that the homogenized eigenvalue problem is posed in the entire domain $\Omega$ without holes, which is a well-understood classical Dirichlet eigenvalue problem.

The homogenized matrix $\overline{A}_\eta$ is defined in a usual way as in the classical homogenization theory. The correctors $\{\chi^j_\eta: j=1,2,\cdots, d \}$ are the unique weak solutions of
\begin{equation}\label{eq.corrector}
    -\nabla\cdot (\phi_{\eta}^2 \nabla \chi^j_\eta) = \nabla\cdot (\phi^2_{\eta} e_j) \quad \text{in } Y_{\eta},
\end{equation}
satisfying $\chi^j_\eta \in H^1_{\phi, {\rm per}}(Y_{\eta})$ and $\int_{Y_{\eta}} \chi^j_\eta = 0$. Then the homogenized matrix $\overline{A}_\eta$ is given by
\begin{equation*}
    \overline{A}_\eta = I + \int_{Y_{\eta}} \phi_{\eta}^2 \nabla \chi_\eta, \quad \text{or in the component form } (\bar{a}_\eta)_{ij} = \delta_{ij} +  \int_{Y_{\eta}} \phi_{\eta}^2 \partial_i \chi^j_\eta.
\end{equation*}
Here and after, we will view $\nabla \chi_\eta$ as a $d\times d$ matrix given by $(\nabla \chi_\eta)_{ij} = \partial_i \chi^j_\eta$.
It was shown in \cite{V81} that $\overline{A}_\eta$ is elliptic, i.e., there exists $\Lambda>0$ (independent of $\eta$) such that $\xi\cdot \overline{A}_\eta \xi \ge \Lambda |\xi|^2$ for any $\xi \in \R^d$. From the equation \eqref{eq.corrector}, one can also see
\begin{equation}\label{eq.bar.aij}
    (\bar{a}_\eta)_{ij} = \int_{Y_{\eta}} \phi_{\eta}^2 (\delta_{ki} + \partial_k \chi^i_\eta)(\delta_{kj} + \partial_k \chi^j_\eta ).
\end{equation}
It should be pointed out that the matrix $\overline{A}_\eta$ depends only on the geometric structure of the holes in $T_\eta$.

There is an alternative way to see the corrector equation \eqref{eq.corrector}. Indeed, if we set $\widehat{\chi}_\eta^j = \phi_{\eta} \chi^j_\eta$ (called weighted correctors), then the equation \eqref{eq.corrector} becomes
\begin{equation}\label{eq.corrector2}
    -\Delta \widehat{\chi}_\eta^j = \overline{\lambda}_\eta \widehat{\chi}_\eta^j + 2\partial_j \phi_{\eta} \qquad \text{in } Y_{\eta},
\end{equation}
with $\widehat{\chi}_\eta^j \in H^1_{0,{\rm per}}(Y_{\eta})$.
The homogenized matrix $\overline{A}_\eta$ can be written as
\begin{equation*}
    \overline{A}_\eta = I + 2\int_{Y_{\eta}} \phi_{\eta} \nabla \widehat{\chi}_\eta.
\end{equation*}
This can be verified directly; also see \cite[Theorem 8.1]{V81}.

When $\eta$ is small, by Lemma \ref{lemma-7.2} (i) or (iv), we know that $\phi$ is close to the constant $1$. This fact will pass to the smallness of the correctors $\chi_\eta$ through the energy estimate. 
More related estimates are included in the following proposition.

\begin{proposition}\label{prop.chi}
There exists a constant $C_1$ such that the following estimates hold.
\begin{itemize}
    \item[(i)] Norm estimate:
    \begin{equation}
    \| \phi_{\eta} \nabla \chi_\eta \|_{L^2(Y_{\eta})} +\| \phi_{\eta} \chi_\eta \|_{L^{2^*}(Y_{\eta})} + \| \chi_\eta \|_{L^2(Y_{\eta})} \le C_1 \eta^{\frac{d-2}{2}}.
\end{equation}

    \item[(ii)] Gradient estimate: For any $x\in Y_{\eta}$,
    \begin{equation}\label{est.phiDchi.pw}
        |\phi_{\eta} \nabla \chi_\eta(x)| \le C_1 \eta^{-1}  \Big\{ \Big( \frac{\eta}{{\rm dist}(x, T_\eta)} \Big)^{\frac{d}{2}} \wedge 1 \Big\}.
    \end{equation}

    \item[(iii)] Pointwise estimate: For any $x\in Y_{\eta}$,
    \begin{equation}\label{est.chieta.pw}
        |\chi_\eta(x)| \le C_1 \Big\{ \Big( \frac{\eta
        }{{\rm dist}(x, T_\eta)} \Big)^{\frac{d-2}{2}} \wedge 1  \Big\}.
    \end{equation}

    \item[(iv)] Homogenized matrix estimate: $|\overline{A}_\eta - I| \le C_1 \eta^\frac{d-2}{2}$.
\end{itemize}

\begin{proof}
    To prove (i), we writing the equation \eqref{eq.corrector} as
\begin{equation}\label{eq.corrector-2}
    -\nabla\cdot (\phi_{\eta}^2 \nabla \chi^j_\eta) = 2\phi_{\eta} \partial_j \phi_{\eta} \quad \text{in } Y_{\eta},
\end{equation}
and using the energy estimate, we have
\begin{equation}\label{est.phiDchi.L2}
    \| \phi_{\eta} \nabla \chi_\eta^j \|_{L^2(Y_{\eta})} \le C \| \partial_j \phi_{\eta} \|_{L^2(Y_{\eta})} \le C\eta^{\frac{d-2}{2}}.
\end{equation}
The estimates of $\| \phi_{\eta} \chi_\eta \|_{L^{2^*}(Y_{\eta})}$ and $\| \chi_\eta \|_{L^2(Y_{\eta})}$ follows from the weighted Sobolev-Poincar\'{e} inequality; see the proofs of \cite[Lemma 2.2, 2.3]{SZ23} for the case of small $\eta$. Observe that (iv) follows from \eqref{est.phiDchi.L2} and the definition of $\overline{A}_\eta$.

Next we prove the boundary estimate of (ii) and (iii). Consider the equation \eqref{eq.corrector2} for $\widehat{\chi}_\eta = \phi_\eta \chi_\eta$. Recall that $\widehat{\chi}_\eta = 0$ on $\partial T_\eta$. Suppose $x_0 \in \partial T_\eta$. Applying the elliptic regularity in $Y_\eta \cap B_\eta(x_0)$, we obtain
\begin{equation}
\begin{aligned}
    \| \nabla \widehat{\chi}_\eta \|_{L^\infty(Y_\eta \cap B_\eta(x_0))} & \le C \eta^{-1} \bigg( \fint_{Y_\eta \cap B_{2\eta}(x_0)} |\widehat{\chi}_\eta|^{2^*} \bigg)^{1/2^*} + C\eta  \| \nabla \phi_\eta \|_{L^\infty(Y_\eta \cap B_{2\eta}(x_0))} \\
    & \le C\eta^{-1}.
\end{aligned}
\end{equation}
Thus for $x\in Y_\eta$ with $\text{dist}(x,T_\eta) \le \eta$,
\begin{equation}
    |\widehat{\chi}_\eta(x)| = |\phi_\eta(x) \chi_\eta(x)| \le C\dist(x,T_\eta) \eta^{-1}.
\end{equation}
In view of \eqref{7.2-1}, this implies $|\chi_\eta(x)| \le C$ for $x\in Y_\eta$ with $\text{dist}(x,T_\eta) \le \eta$.
Moreover, using $\phi_\eta \nabla \chi_\eta(x) = \nabla \widehat{\chi}_\eta - \nabla \phi_\eta \chi_\eta$, Lemma \ref{lemma-7.2} (ii) and the previous estimates, we have $|\phi_\eta \nabla \chi_\eta(x)| \le C\eta^{-1}$ for any $x\in Y_\eta$ with $\text{dist}(x,T_\eta) \le \eta$.

Finally, we prove the interior estimates of (ii) and (iii). Note that for $\dist(x,T_\eta) > \eta$, the equation \eqref{eq.corrector} is not degenerate. Let $r = \dist(x,T_\eta) > \eta$. Then the elliptic estimate combined with \eqref{7.2-2} and Proposition \ref{lemma-7.2} (iv) leads to
\begin{equation}
\begin{aligned}
    |\chi_\eta(x)| & \le C \bigg( \fint_{B_{r/2}(x)} |\chi_\eta|^{2^*} \bigg)^{1/2^*} + C r \| \phi_\eta - 1 \|_{L^\infty(B_{r/2}(x))} \\
    & \le Cr^{-\frac{d}{2^*}} \eta^{\frac{d-2}{2}} + C r \Big( \frac{\eta}{r} \Big)^\frac{d-2}{2} \le C\Big( \frac{\eta}{r} \Big)^\frac{d-2}{2},
\end{aligned}
\end{equation}
and
\begin{equation}
\begin{aligned}
    |\nabla \chi_\eta(x)| & \le C \bigg( \fint_{B_{r/2}(x)} |\nabla \chi_\eta|^{2} \bigg)^{1/2} + C \| \phi_\eta - 1 \|_{L^\infty(B_{r/2}(x))} \\
    & \le Cr^{-\frac{d}{2}} \eta^{\frac{d-2}{2}} + C \Big( \frac{\eta}{r} \Big)^\frac{d-2}{2} \le C \eta^{-1} \Big( \frac{\eta}{r} \Big)^\frac{d}{2}.
\end{aligned}
\end{equation}
The proof is complete.
\end{proof}

\begin{remark}
    The estimate \eqref{est.phiDchi.pw} implies $\| \phi_{\eta} \nabla \chi_\eta \|_{L^\infty(Y_{\eta})} \le C\eta^{-1}$. By an interpolation inequality between this and the $L^2$ estimate in \eqref{7.2-2}, we have $\| \phi_{\eta} \nabla \chi_\eta \|_{L^p(Y_{\eta})} \le C\eta^{\frac{d}{p}-1}$ for any $p>2$. Similarly, we can estimate $\| \phi_\eta \chi_\eta \|_{L^p(Y_\eta)}$ and $\| \chi_\eta \|_{L^p(Y_\eta)}$.
\end{remark}

% \begin{equation}\label{7.2-0}
%     c_1 \le \phi_{\eta,\e} (x)  \le C_1 \quad \text{ if } x \in Y \text{ and } \text{dist} (x, T_\eta) \ge \eta,
% \end{equation}
% and
% \begin{equation}\label{7.2-1}
%     c_1 \eta^{-1} \text{dist}(x, T_\eta)
%     \le \phi_{\eta,\e} (x) \le C_1 \eta^{-1} \text{dist} (x, T_\eta) \quad  
% \end{equation}
% if $ x\in Y \text{ and } 0< \text{dist}(x, T_\eta) \le \eta$.

\end{proposition}

%It was proved in \cite{V81} that $\lim_{\e\to 0} \mu_{\eta,\e}^k = \mu^k_\eta$. However, the method in \cite{V81} is not able to give an explicit convergence rate.

%We first prove several inequalities over in a periodic cell $Y_{\eta}$.

\subsection{Three eigenvalue problems}\label{sec.2.5}
We summarize the three different eigenvalue problems, listed as follows:
\begin{align}
    \text{Degenerate problem: } \cL_{\eta,\e}(\rho_{\eta,\e}^k) & = \mu_{\eta,\e}^k \phi_{\eta,\e}^2 \rho_{\eta,\e}^k   \quad \text{ in } \Omega_{\eta,\e},  \label{eq.eigen-1} \\
    \text{Intermediate problem: } \cL_{\eta}(\tilde{\rho}_{\eta,\e}^k) & = \tilde{\mu}_{\eta,\e}^k \phi_{\eta,\e}^2 \tilde{\rho}_{\eta,\e}^k  
    \quad \text{ in } \Omega, \label{eq.eigen-2}  \\
    \text{Homogenized problem: } \cL_{\eta}({\rho}_\eta^k) & = \mu_\eta^k {\rho}_\eta^k   \qquad   \text{ in } \Omega. \label{eq.eigen-3}
\end{align}
The three types of elliptic equations corresponding to the eigenvalue problems are:
\begin{align}
    \text{Degenerate equation: } \cL_{\eta,\e}(u_{\eta,\e}) &= \phi_{\eta,\e}^2 f, & \quad& \text{with } f\in L_{\phi_{\eta,\e}}^2(\Omega_{\eta,\e}); \label{eq.degnerate} \\
    \text{Intermediate equation: } \cL_{\eta}(\tilde{u}_{\eta,\e}) & = \phi_{\eta,\e}^2 f, &\quad& \text{with } f\in L_{\phi_{\eta,\e}}^2(\Omega_{\eta,\e}); \label{eq.intermediate}\\
    \text{Homogenized equation: } \cL_{\eta}(u_{\eta}) & = f, &\quad& \text{with } f\in L^2(\Omega). \label{eq.homogenized}
\end{align}
All three equations are solvable in suitable spaces. Precisely, \eqref{eq.degnerate} is solvable in $V_{\eta,\e} = H^1_{\phi_{\eta,\e},0}(\Omega_{\eta,\e})$ if $f\in L^2_{\phi_{\eta,\e}}(\Omega_{\eta,\e})$; \eqref{eq.intermediate} is solvable in $H^1_0(\Omega)$ if $f\in L^2_{\phi_{\eta,\e}}(\Omega_{\eta,\e})$; and \eqref{eq.homogenized} is solvable in $H^1_0(\Omega)$ if $f\in L^2(\Omega)$.
Define the bounded operators (resolvents) corresponding to the above equations:
\begin{align}
    &\mathscr{T}_{\eta,\e}: L^2_{\phi_{\eta,\e}}(\Omega_{\eta,\e}) \to V_{\eta,\e}, \text{ with } \mathscr{T}_{\eta,\e}(f) = u_{\eta,\e};  \label{eq.degenerate T}\\ 
    & \widetilde{\mathscr{T}}_{\eta,\e}: L^2_{\phi_{\eta,\e}}(\Omega_{\eta,\e}) \to H^1_0(\Omega), \quad \text{ with } \widetilde{\mathscr{T}}_{\eta,\e}(f) = \tilde{u}_{\eta,\e}; \label{eq.intermediate T} \\
    &\mathscr{T}_{\eta}: L^2(\Omega) \to H^1_0(\Omega), \qquad\text{ with } \mathscr{T}_{\eta}(f) = u_{\eta}. \label{eq.homogenized T}
\end{align}

Previously we have seen that $\mathscr{T}_{\eta,\e}$ is a self-ajoint compact operator in $L^2_{\phi_{\eta,\e}}(\Omega_{\eta,\e})$ such that the spectral theory applies to the degenerate eigenvalue problem. It is classical that $\mathscr{T}_{\eta}$ is a self-adjoint compact operator in $L^2(\Omega)$ and thus the spectral theory applies to the homogenized eigenvalue problem. Now we consider the intermediate eigenvalue problem.
We will first view $\widetilde{\mathscr{T}}_{\eta,\e}$ as a compact self-adjoint operator on $L^2_{\phi_{\eta,\e}}(\Omega_{\eta,\e})$ since $H^1_0(\Omega)$ restricted on $\Omega_{\eta,\e}$ is a subspace of $V_{\eta,\e}$. Therefore, $\widetilde{\mathscr{T}}_{\eta,\e}$ is also a self-adjoint compact operator in the space $L^2_{\phi_{\eta,\e}}(\Omega_{\eta,\e})$. Note that the eigenfunctions of $\mathscr{T}_{\eta,\e}$ and $\widetilde{\mathscr{T}}_{\eta,\e}$ form two different orthonormal basis of the same space $L^2_{\phi_{\eta,\e}}(\Omega_{\eta,\e})$.
On the other hand, since $L^2(\Omega)$ restricted in $\Omega_{\eta,\e}$ is a subspace of $L^2_{\phi_{\eta,\e}}(\Omega_{\eta,\e})$, we can also view $\widetilde{\mathscr{T}}_{\eta,\e}$ as a compact operator on $L^2(\Omega)$. Thus, we may also compare $\mathscr{T}_{\eta}$ and $\widetilde{\mathscr{T}}_{\eta,\e}$ as operators on $L^2(\Omega)$. However, it is important to point out that $\widetilde{\mathscr{T}}_{\eta,\e}$ is not self-adjoint in the space $L^2(\Omega)$.
As usual, the eigenfunctions are orthonormal in the corresponding spaces, i.e., for any $i,j \ge 1$,
\begin{equation}\label{eq.normal}
    \int_{\Omega_{\eta,\e}} \phi_{\eta,\e}^2 \rho_{\eta,\e}^i \rho_{\eta,\e}^j = \int_{\Omega_{\eta,\e}} \phi_{\eta,\e}^2 \tilde{\rho}_{\eta,\e}^i \tilde{\rho}_{\eta,\e}^j = \int_{\Omega} \rho_\eta^i \rho_\eta^j = \delta_{ij}.
\end{equation}

\section{Regularity of degenerate equations}\label{Sec.3}
In this subsection, we first prove some basic inequality in weighted Sobolev spaces and then use them to derive certain regularity of eigenfunctions.

% Let $\widetilde{T}$ be the closure of an open set with $C^1$ boundary such that
% $$
% \{x\in Y: \text{dist}(x, T) \le c_0\}
% \subset \widetilde{T} \subset 
% \{x\in Y: \text{dist}(x, T) \le c_1\}.
% $$

\subsection{Weighted Sobolev inequalities}
The proofs of the weighted Sobolev inequalities use harmonic extension. We will not include all the details of the proofs in this section as similar arguments may be found in \cite[Section 2]{SZ23}.

Let $T_{\eta}^* = \{ x\in Y: \dist(x,T_\eta) \le \eta \}$.
Let 
\begin{equation}
\Omega_{\eta,\e}^*
=\Omega\setminus \bigcup_z \e (z+ T_{\eta}^*).
\end{equation}
For $u\in W^{1,p}(\Omega_{\eta,\e}^*) $, we use $u^* $ to denote the harmonic extension of $u$ to $\Omega$,
i.e., the value of $u^*$ in $\e (z+ T_{\eta}^*) \subset \Omega$ is defined by the solution of
$$
\left\{
\aligned
\Delta u^*   & =0 & \quad & \text{ in } \e (z+ T_{\eta}^*),\\
u^* & =u & \quad & \text{ on }  \e (z+\partial T_{\eta}^*).
\endaligned
\right.
$$
By the boundedness of harmonic extension in $W^{1,p}$ space, we have $u^*\in W^{1,p}(\e (z+T_{\eta}^*))$.

For the holes intersecting with the boundary $\partial \Omega$, we need to modify the harmonic extension. We distinguish between two cases. If $u = 0$ on $\partial \Omega \cap \partial \Omega_{\eta,\e}^*$, then for each hole $\e(z+T_\eta^*) \cap \partial \Omega \neq \emptyset$, we define $u^*$ in $T_\eta^* \cap \Omega$ by
\begin{equation}
   \left\{
\aligned
\Delta u^*   & =0 & \quad & \text{ in } \e (z+ T_{\eta}^*) \cap \Omega,\\
u^* & =u & \quad & \text{ on }   \e (z+ \partial T_{\eta}^*) \cap \Omega,\\
u^* & = 0 & \quad & \text{ on }  \e (z+ T_{\eta}^*) \cap \partial \Omega.
\endaligned
\right. 
\end{equation}
This guarantees that $u^* = 0$ on $\partial \Omega$ and therefore $u^* \in W^{1,p}_0(\Omega)$.

For general $u\in W^{1,p}(\Omega_{\eta,\e}^*)$, we first flatten the boundary $\partial \Omega$ in a neighborhood of a hole by a change of variables. Now since the boundary is flat, a function in $W^{1,p}(\Omega_{\eta,\e}^*)$ can be extended by an even reflection across the flat boundary with a reflected hole symmetric about the flat boundary. Then a standard harmonic extension applies to this extended hole. Finally, changing variables back to the original boundary, we get the extended function $u^* \in W^{1,p}(\Omega)$. 

%If $u =0$ on $\partial \Omega \cap \partial \Omega_{\eta,\e}^*$, then $u^*\in W^{1,p}(\Omega)$ can be an extended function such that $u = 0$ on $\partial \Omega$.

\begin{proposition}
Let $D_{\eta, \e} =\{ x\in \Omega_{\eta,\e}^* :  {\rm dist}(x, \Omega \setminus \Omega_{\eta,\e}^* )\le  \e \eta \} \subset \Omega_{\eta,\e}^*.$ Let $u\in W^{1,p}(\Omega_{\eta,\e}^*)$ and $u^*$ be the harmonic extension defined above.
Then we have the following properties:
    \begin{align}
    \|u^*  \|_{L^p(\Omega\setminus \Omega_{\eta,\e}^* )} &\le C\big(  \| u \|_{L^p(D_{\eta, \e}) }
    +\e\eta \|\nabla u \|_{L^p (D_{\eta, \e})} \big), \\
    \|\nabla u^*  \|_{L^p(\Omega\setminus \Omega_{\eta,\e}^*)} &\le C \| \nabla u \|_{L^p(D_{\eta, \e}) }, \\
    \| \nabla u^* \|_{L^p(\Omega)} & \le C \| \nabla u \|_{L^p(\Omega_{\eta,\e}^*) }, \label{est.Dtu.Exteta}
\end{align}
where $C$ is independent of $\e$ and $\eta$.
\end{proposition}

\begin{proof}
    These estimates may be proved by using \cite[Lemma 2.1]{SZ23} around each hole.
\end{proof}

%The following theorem is an analog of Theorem \ref{thm.unweigthed} and Theorem \ref{thm.unweigthed2}.
\begin{theorem}\label{thm.unweigthed.eta}
    Let $1< p< \infty$ and $u\in W^{1,p}_{\phi_{\eta,\e}}(\Omega_{\eta,\e})$. Then there exists $C>0$ such that
    \begin{equation}\label{est.multi.eta}
        \| u\|_{L^p(\Omega_{\eta, \e}) } \le C \e\eta \|\phi_{\eta, \e} \nabla u \|_{L^p(\Omega_{\eta, \e})} 
        + C \| \phi_{\eta, \e}  u\|_{L^p(\Omega_{\eta, \e} )}
    \end{equation}
    Moreover, if $u\in W^{1,p}_{\phi_{\eta,\e},0}(\Omega_{\eta,\e})$, then
    \begin{equation}\label{est.Poincare.eta}
        \| u\|_{L^p(\Omega_{\eta, \e}) } \le C \|\phi_{\eta, \e} \nabla u \|_{L^p(\Omega_{\eta, \e})} 
    \end{equation}
\end{theorem}

\begin{proof}
    %Choose $T\subset T_1\subset T_2\subset Y$ such that dist$(\partial T, \partial T_1)\ge c>0$ and dist$(\partial T_1, \partial T_2)\ge c>0$. 
    Using Hardy's inequality, we obtain 
    \begin{equation}\label{est.Hardy.eta}
    \aligned
    \int_{\e (z+ T_{\eta}^* \setminus T_\eta ) \cap \Omega} |u|^p
     & \le  C \int_{\e (z+ T_{\eta}^* \setminus T_\eta) \cap \Omega} |\nabla u |^p \{ \text{dist} (x, \e (z+ T_\eta)) \}^p \, dx 
    +C \int_{\e (z +Y_{\eta}\setminus T_{\eta}^* ) \cap \Omega}  |u|^p\\
    & \le  C(\e \eta)^p \int_{\e (z+ T_{\eta}^* \setminus T_\eta) \cap \Omega} | \phi_{\eta, \e} \nabla u |^p
    + C \int_{\e (z +Y_{\eta} \setminus T_{\eta}^* ) \cap \Omega}  |u|^p,
    \endaligned
    \end{equation}
    where we have used the fact  $ \phi_{\eta, \e} (x) \approx (\e \eta )^{-1} \text{dist}(x, \e T_\eta)$ for $x\in Y$ with $\dist(x, \e T_\eta) \le C \e \eta$; see \eqref{7.2-1}.
    This gives \eqref{est.multi.eta} by summation over $z \in \Z^d$.
    
    Note that \eqref{est.Hardy.eta} also implies
    \begin{equation}\label{est.Lpmulti.eta}
    \| u \|_{L^p(\Omega_{\eta, \e}) } \le C \e \eta \| \phi_{\e, \eta } \nabla u \|_{L^p(\Omega_{\eta, \e})}
    + C \| u \|_{L^p(\Omega_{\eta,\e}^*)  }.
    \end{equation}
    If $u \in W^{1,p}_{\phi_{\eta,\e},0}(\Omega_{\eta,\e})$ and $u^*$ is its harmonic extension according to the above construction, then $u^* \in W^{1,p}_0(\Omega)$. Hence
    \begin{equation}
    \| u \|_{L^p(\Omega_{\eta,\e}^*  )}\le \| u^* \|_{L^p(\Omega)} 
    \le C \|\nabla u^* \|_{L^p(\Omega)}
    \le C \| \nabla u \|_{L^p(\Omega_{\eta,\e}^*)} \le C\|\phi_{\eta, \e} \nabla u \|_{L^p(\Omega_{\eta, \e})},
    \end{equation}
    where we have used the Poincar\'{e} inequality in the second inequality and \eqref{est.Dtu.Exteta} in the third inequality. This combined with \eqref{est.Lpmulti.eta} leads to \eqref{est.Poincare.eta}.
\end{proof}

%The following theorem is an analog of Theorem \ref{thm.WSPI} (i).
\begin{theorem} \label{thm.WSPI.eta}
Let $1<p<d$ and $ u\in W^{1,p}_{\phi_{\eta,\e},0}(\Omega_{\eta,\e})$. Then there exists a constant $C>0$ such that
\begin{equation}\label{est.WSPI}
    \| \phi_{\eta, \e} u \|_{L^{p^*} (\Omega_{\eta, \e} )}
    \le C \| \phi_{\eta, \e} \nabla u \|_{L^p(\Omega_{\eta, \e} ) }
\end{equation}
\end{theorem}

\begin{proof}
    Let $u^*$ be the harmonic extension of $u|_{\Omega_{\eta,\e}^*}$ to $\Omega$. Then the classical Sobolev-Poincar\'{e} inequality in $\Omega$ leads to
    \begin{equation}\label{est.Sobolev.eta}
    \aligned
    \| u \|_{L^{p^*}(\Omega_{\eta,\e}^*  )} & \le \| u^* \|_{L^{p^*}(\Omega)} 
    \le C \|\nabla u^* \|_{L^p(\Omega)}\\
    &  \le C \| \nabla u \|_{L^p(\Omega_{\eta,\e}^*)}
    \le C \| \phi_{\eta, \e} \nabla u \|_{L^p(\Omega_{\eta, \e})}.
    \endaligned
    \end{equation}
    Let $T_{\eta}^{**} = \{ x\in Y: \dist(x,T_\eta) \le 2\eta \}$. Note that $T_{\eta} \subset T_{\eta}^* \subset T_{\eta}^{**} \subset Y$. By the Sobolev inequality in $\e(z+T_{\eta}^{**} \setminus T_\eta)$
    \begin{equation}
    \begin{aligned}
    \left(\int_{\e (z+ T_{\eta}^{**} \setminus T_\eta)} |\phi_{\eta, \e} u |^{p^*}\right)^{1/{p^*}}
    & \le C\left(\int_{\e (z+ T_{\eta}^{**} \setminus T_\eta)} |\phi_{\eta, \e} \nabla u |^p\right)^{1/p}
    + \frac{C}{\e \eta} \left(\int_{\e (z+ T_{\eta}^{**} \setminus T_{\eta}^*)} |u |^p\right)^{1/p} \\
    & \le C\left(\int_{\e (z+ T_{\eta}^{**} \setminus T_\eta)} |\phi_{\eta, \e} \nabla u |^p\right)^{1/p}
    + C \left(\int_{\e (z+ T_{\eta}^{**} \setminus T_{\eta}^*)} |u |^{p^*}\right)^{1/{p^*}},
    \end{aligned}
    \end{equation}
    where we have used the facts $|\nabla \phi_{\eta, \e} |\le C(\e \eta)^{-1}$ in $\e (z+ T_{\eta}^{**} \setminus T_{\eta}^*)$ and $|T_{\eta}^{**} \setminus T_{\eta}^*| \le C(\e \eta)^d$.    
    % This, together with
    % $$
    %  \left(\int_{\e (z+ \eta (\widetilde{T}\setminus T) )} |u |^p\right)^{1/p}
    % \le C \e \eta
    % \left(\int_{\e (z+ \eta \widetilde{T})} |\phi_{\eta, \e} \nabla u |^p\right)^{1/p}
    % + C  \left(\int_{\e (z+ \eta (T_1 \setminus \widetilde{T} )} |u |^p\right)^{1/p}
    % $$
    % shows that 
    % $$
    % \left(\int_{\e (z+ \eta \widetilde{T})} |\phi_{\eta, \e} u |^q\right)^{1/q}
    % \le \left(\int_{\e (z+ \eta \widetilde{T})} |\phi_{\eta, \e} \nabla u |^p\right)^{1/p}
    % + C \left(\int_{\e (z+ \eta  T_1\setminus \widetilde{T})} |u |^q\right)^{1/q}
    % $$
    By the summation over $z$, we have
    $$
    \aligned
    \int_{\Omega_{\eta, \e} \setminus \Omega_{\eta,\e}^*}  | \phi_{\eta, \e} u |^{p^*}
    &  \le C \sum_z  \left( \int_{\e (z+ T_{\eta}^{**} \setminus T_\eta)} |\phi_{\eta, \e} \nabla u |^p\right)^{p^*/p}
    +C \int_{\Omega_{\eta,\e}^* } |u|^{p^*}\\
    & \le C \| \phi_{\eta, \e} \nabla u \|^{p^*}_{L^p (\Omega_{\eta, \e})} 
    + C \| u \|^{p^*}_{L^{p^*}(\Omega_{\eta, \e}^*) }\\
    & \le C \| \phi_{\eta, \e} \nabla u \|^{p^*}_{L^p (\Omega_{\eta, \e})},
    \endaligned
    $$
    where we have used \eqref{est.Sobolev.eta} in the last inequality. This and \eqref{est.Sobolev.eta} give the desired estimate.
\end{proof}

\subsection{Global $L^p$ estimate}
Now, we consider the equation
\begin{equation}\label{eq.ue.eta}
\left\{
\begin{aligned}
    \cL_{\eta,\e}(u_{\eta,\e}) & = \phi_{\eta,\e} f, &\quad& \text{in } \Omega_{\eta,\e},\\
    u_{\eta,\e} & = 0, &\quad & \text{on } \Gamma_{\eta,\e}.
\end{aligned}
    \right.
\end{equation}
We say $u_{\eta,\e} \in V_{\eta,\e}$ is a weak solution of \eqref{eq.ue.eta}, if for any $\varphi \in V_{\eta,\e}$, we have
\begin{equation}
    \int_{\Omega_{\eta,\e}} \phi_{\eta,\e}^2 \nabla u_{\eta,\e} \cdot \nabla \varphi = \int_{\Omega_{\eta,\e}} \phi_{\eta,\e} f \varphi.
\end{equation}
By the Lax-Milgram theorem, the weak solution exists and is unique if $f\in L^2(\Omega_{\eta,\e})$.

\begin{lemma}\label{lem.L^p integrability.eta}
    Let $p\in [2,\infty)$ and $f\in L^p(\Omega_{\eta,\e})$. Let $u_{\eta,\e} \in V_{\eta,\e}$ be a weak solution of \eqref{eq.ue.eta}. Then
    \begin{itemize}
        \item[(i)] If $p\in [2,d)$, then
        \begin{equation}\label{est.global Lp}
        \| u_{\eta,\e} \|_{L^{p^*}(\Omega_{\eta,\e})} \le C\| f \|_{L^p(\Omega_{\eta,\e})}.
    \end{equation}

    \item[(ii)] If $p>d$, then
    \begin{equation}\label{est.global Linfty}
        \| u_{\eta,\e} \|_{L^\infty(\Omega_{\eta,\e})} \le C\| f \|_{L^p(\Omega_{\eta,\e})}.
    \end{equation}
    \end{itemize}
    
\end{lemma}

\begin{proof}
    (i) Let $u_{\eta,\e}^M = \max\{-M, \min\{u_{\eta,\e}, M\} \}$ for $M>0$. We will prove a uniform estimate for $u_{\eta,\e}^M$ and then take the limit as $M\to \infty$. Since now $u_{\eta,\e}^M$ is bounded, $|u_{\eta,\e}^M|^{p^*-2} u_{\eta,\e}^M$ can serve as a test function in $V_{\eta,\e}$ for the equation \eqref{eq.ue.phief}. It follows that 
    \begin{equation}\label{eq.Moser}
        (p^*-1) \int_{\Omega_{\eta,\e}} \phi_{\eta,\e}^2 |u_{\eta,\e}^M|^{p^*-2} \nabla u_{\eta,\e} \cdot \nabla u_{\eta,\e}^M = \int_{\Omega_{\eta,\e}} f\phi_{\eta,\e} |u_{\eta,\e}^M|^{p^*-2} u_{\eta,\e}^M.
    \end{equation}
    Since $\nabla u_{\eta,\e}^M$ is supported in $\{ |u_{\eta,\e}| \le M \}$, in which $\nabla u_{\eta,\e} = \nabla u_{\eta,\e}^M$, \eqref{eq.Moser} can be simplified to
    \begin{equation}\label{Moser-1}
        (p^*-1) \int_{\Omega_{\eta,\e}} \phi_{\eta,\e}^2 |u_{\eta,\e}^M|^{p^*-2} |\nabla u_{\eta,\e}^M|^2 = \int_{\Omega_{\eta,\e}} f\phi_{\eta,\e} |u_{\eta,\e}^M|^{p^*-2} u_{\eta,\e}^M.
    \end{equation}

    Let $v_{\eta,\e}^M = |u_{\eta,\e}^M|^{\frac{p^*-2}{2}} u_{\eta,\e}^M$. Then $\nabla v_{\eta,\e}^M = \frac{p^*}{2} |\nabla u_{\eta,\e}^M|^{\frac{p^*-2}{2}} \nabla u_{\eta,\e}^M$. Let $\frac{1}{2^+} = \frac{1}{2} - \frac{1}{d}$ and $\frac{1}{2^+} + \frac{1}{2^-} = 1$. It follows from  \eqref{Moser-1} that 
    \begin{equation*}
    \begin{aligned}
        \int_{\Omega_{\eta,\e}} \phi_{\eta,\e}^2 |\nabla v_{\eta,\e}^M|^2 & \le C_p \int_{\Omega_{\eta,\e}} f|u_{\eta,\e}^M|^{\frac{p^*-2}{2}} \phi_{\eta,\e} v_{\eta,\e}^M\\
        & \le C_p \bigg( \int_{\Omega_{\eta,\e}} | f|u_{\eta,\e}^M|^{\frac{p^*-2}{2}} |^{2^-} \bigg)^{1/2^-} \bigg( \int_{\Omega_{\eta,\e}} |\phi_{\eta,\e} v_{\eta,\e}^M|^{2^+} \bigg)^{1/2^+} \\
        & \le C_p \bigg( \int_{\Omega_{\eta,\e}} | f|u_{\eta,\e}^M|^{\frac{p^*-2}{2}} |^{2^-} \bigg)^{1/2^-} \bigg( \int_{\Omega_{\eta,\e}} |\phi_{\eta,\e} \nabla v_{\eta,\e}^M|^{2} \bigg)^{1/2},
    \end{aligned}
    \end{equation*}
    where we have used Theorem \ref{thm.WSPI.eta} in the last inequality. 
    Hence, Theorem \ref{thm.unweigthed.eta} implies
    \begin{equation}\label{est.ueLp.2-}
        \int_{\Omega_{\eta,\e}} |u_{\eta,\e}^M|^{p^*} = \int_{\Omega_{\eta,\e}} |v_{\eta,\e}^M|^2 \le C \int_{\Omega_{\eta,\e}} \phi_{\eta,\e}^2 |\nabla v_{\eta,\e}^M|^2 \le C_p \bigg( \int_{\Omega_{\eta,\e}} | f|u_{\eta,\e}^M|^{\frac{p^*-2}{2}} |^{2^-} \bigg)^{2/2^-}.
    \end{equation}
    Now let $\kappa>1$ be such that $\frac{p^*-2}{2} \kappa = p^*$ and $\frac{1}{\kappa} + \frac{1}{\kappa'} = \frac{1}{2^-}$. Then
    \begin{equation*}
        \bigg( \int_{\Omega_{\eta,\e}} | f|u_{\eta,\e}^M|^{\frac{p^*-2}{2}} |^{2^-} \bigg)^{1/2^-} \le \bigg( \int_{\Omega_{\eta,\e}} |f|^{\kappa'} \bigg)^{1/\kappa'} \bigg( \int_{\Omega_{\eta,\e}} |u_{\eta,\e}^M|^{p^*} \bigg)^{1/\kappa}.
    \end{equation*}
    Substituting this into \eqref{est.ueLp.2-}, we obtain
    \begin{equation*}
        \bigg(\int_{\Omega_{\eta,\e}} |u_{\eta,\e}^M|^{p^*} \bigg)^{(1/2)-(1/\kappa)} \le C \bigg( \int_{\Omega_{\eta,\e}} |f|^{\kappa'} \bigg)^{1/\kappa'}.
    \end{equation*}
    This implies
    \begin{equation*}
        \| u_{\eta,\e}^M \|_{L^{p^*}(\Omega_{\eta,\e})} \le C\| f \|_{L^p(\Omega_{\eta,\e})}.
    \end{equation*}
    by noting $\frac12-\frac{1}{\kappa} = \frac{1}{p^*}$ and $\frac{1}{\kappa'} = \frac{1}{p} = \frac{1}{p^*} + \frac{1}{d}$. Observe that in the last inequality, the constant $C$ is independent of $M$. Letting $M\to \infty$, we obtain the desired estimate.

    (ii) We apply the De Giorgi's iteration. We can even prove the more general case $u_{\eta,\e} = g$ on $\partial \Gamma_{\eta,\e}$ with bounded $g$. Let $\ell_0 := \sup_{\Gamma_{\eta,\e}} g$. Let $k>\ell_0$. Define $v_k = (u_{\eta,\e}- k)_+ := \max\{ u_{\eta,\e}- k, 0 \}$ and $E_k = \{ x\in \Omega_{\eta,\e}: u_{\eta,\e} > k \}$. Clearly, $v_k \in V_{\eta,\e}$ and $\nabla v_k = \nabla u_{\eta,\e}$ in $E_k$ and $\nabla v_k = 0$ in $E_k^c$. It follows that
    \begin{equation}
    \begin{aligned}
        \| \phi_{\eta,\e} \nabla v_k \|_{L^2(\Omega_{\eta,\e})}^2 & = \bigg| \int_{\Omega_{\eta,\e}} \phi_{\eta,\e}^2 \nabla u_{\eta,\e} \cdot \nabla v_k \bigg| \\
        & = \bigg| \int_{E_k} \phi_{\eta,\e} f v_k \bigg| \\
        & \le \| f\|_{L^p(\Omega_{\eta,\e})} \|  \phi_{\eta,\e} v_k \|_{L^{2^*}(E_k)} |E_k|^{\frac{1}{2}+\frac{1}{d}-\frac{1}{p}} \\
        & \le C \| f\|_{L^p(\Omega_{\eta,\e})} \|  \phi_{\eta,\e}  \nabla v_k \|_{L^2(E_k)} |E_k|^{\frac{1}{2}+\frac{1}{d}-\frac{1}{p}}.
    \end{aligned}
    \end{equation}
    Thus,
    \begin{equation}
        \| \phi_{\eta,\e} \nabla v_k \|_{L^2(\Omega_{\eta,\e})} \le C \| f\|_{L^p(\Omega_{\eta,\e})} |E_k|^{\frac{1}{2}+\frac{1}{d}-\frac{1}{p}}.
    \end{equation}
        On the other hand, for any $h>k>\ell_0$
    \begin{equation}
    \begin{aligned}
        \| \phi_{\eta,\e} \nabla v_k \|_{L^2(\Omega_{\eta,\e})} \ge C^{-1} \| v_k \|_{L^2(\Omega_{\eta,\e})} \ge C^{-1} |E_{h}|^\frac12 (h-k).
    \end{aligned}
    \end{equation}
    It follows that
    \begin{equation}
        |E_h| \le C\| f\|_{L^p(\Omega_{\eta,\e})}^2 \frac{|E_k|^{1+\beta}}{|h-k|^2},
    \end{equation}
    where $\beta = 2(\frac{1}{d}-\frac{1}{p}) > 0$. Therefore, by a standard iteration argument,
    \begin{equation}
        |E_{\ell_0 + k_0}| = 0,
    \end{equation}
    with
    \begin{equation}
        k_0 = C^\frac{1}{2} 2^\frac{1+\beta}{\beta} \| f\|_{L^p(\Omega_{\eta,\e})} |E_{\ell_0}|^\frac{\beta}{2}.
    \end{equation}
    This implies that
    \begin{equation}\label{est.maximum}
        \sup_{\Omega_{\eta,\e}} u_{\eta,\e} \le \sup_{\Gamma_{\eta,\e}} g + C_p |\Omega_{\eta,\e}|^{\frac{1}{d}-\frac{1}{p}} \| f\|_{L^p(\Omega_{\eta,\e})}.
    \end{equation}
    The desired estimate follows easily.
\end{proof}

The above lemma can be used to show the boundedness of the eigenfunctions $\{ \rho_{\eta,\e}^j \}$.
\begin{proposition}\label{prop.Linfty.rho}
    For any $1\le j\le k$, we have
    \begin{equation}
        \| \rho_{\eta,\e}^j \|_{L^\infty(\Omega_{\eta,\e})} \le C_{k}.
    \end{equation}
\end{proposition}
\begin{proof}
    We apply \eqref{est.global Lp} repeatedly to the equation $\cL_{\eta,\e}( \rho_{\eta,\e}^j )  = \mu_{\eta,\e}^j \phi_{\eta,\e}^2 \rho_{\eta,\e}^j$ in $\Omega_{\eta,\e}$ until $p>d$. Then using \eqref{est.global Linfty} one more time, we obtain the $L^\infty$ bound.
\end{proof}

%Let $Y_{\eta}^{\eta,+} = \cup_{|k|_\infty\le 1} (k+Y_{\eta}^\eta)$ be the enlarged cube of $Y_{\eta}^\eta$. 

\subsection{Local weighted Lipschitz estimate}

%The next lemma is an analog of Lemma \ref{lem.Local Linfinity} regarding the local Lipschitz estimate. 
%This estimate is valid for the holes entirely contained in $\Omega$ and is not valid near the holes intersecting with $\partial \Omega$.
Recall that $T_{\eta}^{*} = \{ x\in Y: \dist(x,T_\eta) \le \eta \}$ and $T_{\eta}^{**} = \{ x\in Y: \dist(x,T_\eta) \le 2\eta \}$.
\begin{lemma}\label{lem.Local Linfinity.eta}
    Let $p>d$. Let $u_{\eta,\e} \in V_{\eta,\e}$ be a weak solution of \eqref{eq.ue.eta}. Let $k\in \Z^d$ be such that $\e(k+T_{\eta}^{**}) \subset \Omega$. Then
\begin{equation}\label{est.local.Qe}
    \sup_{\e(k+T_{\eta}^* \setminus T_\eta)} |u_{\eta,\e}| \le C\bigg( \fint_{\e(k+T_{\eta}^{**} \setminus T_\eta)} |\phi_{\eta,\e} u_{\eta,\e}|^2 \bigg)^{1/2} +C(\e \eta)^{2} \bigg( \fint_{\e(k+T_{\eta}^{**} \setminus T_\eta)} |f|^p \bigg)^{1/p}, 
\end{equation}
and
\begin{equation}
    \sup_{\e(k+T_{\eta}^* \setminus T_\eta)} |\phi_{\eta,\e} \nabla u_{\eta,\e}| \le C\bigg( \fint_{\e(k+T_{\eta}^{**} \setminus T_\eta)} |\phi_{\eta,\e} \nabla u_{\eta,\e}|^2 \bigg)^{1/2} +C\e\eta \bigg( \fint_{\e(k+T_{\eta}^{**} \setminus T_\eta)} |f|^p \bigg)^{1/p}.
\end{equation}
\end{lemma}

\begin{proof}
    The local $L^\infty$ estimate \eqref{est.local.Qe} can be proved via De Giorgi or Moser's iteration as in \cite[Lemma 3.1]{SZ23}. Here we take advantage of the particular properties of the equation to give a simple proof. Let $v(x) = u_{\eta,\e}(\e x)$. Then \eqref{eq.ue.eta} is reduced to
\begin{equation}\label{eq.v.eta}
    -\nabla\cdot (\phi_{\eta}^2 \nabla v) = \e \phi_{\eta} f(\e x), \quad \text{in } \e^{-1} \Omega_{\eta,\e}.
\end{equation}
Let $\hat{v}(x) = \phi_{\eta}(x) v(x)$. Then $\hat{v}$ satisfies
\begin{equation}\label{eq.hatv}
    -\Delta \hat{v} = \overline{\lambda}_\eta \hat{v} + \e^2 f(\e x), \quad \text{in } \e^{-1} \Omega_{\eta,\e}.
\end{equation}
Consider the equation in a typical $k+ Y_{\eta} \subset \e^{-1} \Omega$. By the classical $L^\infty$ estimate of elliptic equations, we have
\begin{equation}
    \| \hat{v} \|_{L^\infty(k + T_{\eta}^* \setminus T_\eta)} \le C \bigg( \fint_{k + T_{\eta}^{**} \setminus T_\eta} |\hat{v}|^2 \bigg)^\frac12 + C\eta^2 \e^2 \bigg( \fint_{k + T_{\eta}^{**} \setminus T_\eta} |f(\e x)|^p \bigg)^{\frac{1}{p}}.
\end{equation}
Now, using the the gradient estimate $T_{\eta}^* \setminus T_\eta$, we have
\begin{equation}\label{est.Dhatv.eta}
\begin{aligned}
    \| \nabla \hat{v} \|_{L^\infty(k + T_{\eta}^* \setminus T_\eta)} \le \frac{C}{\eta}\bigg( \fint_{k + T_{\eta}^{**} \setminus T_\eta} | \hat{v}|^2 \bigg)^\frac12 + C\eta \e^2 \bigg( \fint_{k + T_{\eta}^{**} \setminus T_\eta} |f(\e x)|^p \bigg)^{\frac{1}{p}}.
\end{aligned}
\end{equation}
By integration, for $x\in k + T_{\eta}^* \setminus T_\eta$,
\begin{equation}\label{est.hatv.eta}
    |\hat{v}(x)| \le C\frac{\dist(x,k+T_\eta)}{\eta} \bigg\{ \bigg( \fint_{k + T_{\eta}^{**} \setminus T_\eta} |\hat{v}|^2 \bigg)^\frac12 + C\eta^2 \e^2 \bigg( \fint_{k + T_{\eta}^{**} \setminus T_\eta} |f(\e x)|^p \bigg)^{\frac{1}{p}} \bigg\}. 
\end{equation}
%For convenience, denote $R_{\eta,\e}(k) := \| \hat{v} \|_{L^2(k+Y_{\eta}^\eta)} + (\e \eta)^{2-\frac{d}{p}} \| f(\e x) \|_{L^p(k + Y_{\eta}^\eta)} $.
Due to $\hat{v} = \phi_{\eta} v$ and Lemma \ref{lemma-7.2} (iii), we obtain from the above inequality
\begin{equation}\label{est.v.eta}
    \| v \|_{L^\infty(k+ T_{\eta}^* \setminus T_\eta)} \le C \bigg\{ \bigg( \fint_{k + T_{\eta}^{**} \setminus T_\eta} |\hat{v}|^2 \bigg)^\frac12 + C\eta^2 \e^2 \bigg( \fint_{k + T_{\eta}^{**} \setminus T_\eta} |f(\e x)|^p \bigg)^{\frac{1}{p}} \bigg\}. 
\end{equation}
Now, using $\nabla \hat{v} = \nabla \phi_{\eta} v + \phi_{\eta} \nabla v$, we derive from \eqref{est.Dhatv.eta} and \eqref{est.v.eta} that
\begin{equation}\label{est.phietaDv}
\begin{aligned}
    \| \phi_{\eta} \nabla v \|_{L^\infty(k+ T_{\eta}^* \setminus T_\eta)} & \le \| \nabla \hat{v} \|_{L^\infty(k+ T_{\eta}^* \setminus T_\eta)} + \| \nabla \phi_{\eta} v \|_{L^\infty(k+ T_{\eta}^* \setminus T_\eta)} \\
    & \le C\bigg\{ \frac{1}{\eta} \bigg( \fint_{k + T_{\eta}^{**} \setminus T_\eta} | \phi_{\eta} v|^2 \bigg)^\frac12 + \eta \e^2 \bigg( \fint_{k + T_{\eta}^{**} \setminus T_\eta} |f(\e x)|^p \bigg)^{\frac{1}{p}} \bigg\}.
\end{aligned}
\end{equation}
Note that, $v-L$ with arbitrary constant $L$ also satisfies the equation \eqref{eq.v.eta}. Thus the first term on the right-hand side of \eqref{est.phietaDv} can be replaced by the weighted Poincar\'{e} inequality
\begin{equation}
    \inf_{L\in \R} \bigg( \fint_{k + T_{\eta}^{**} \setminus T_\eta} | \phi_{\eta} (v-L)|^2 \bigg)^\frac12 \le C\eta \bigg( \fint_{k + T_{\eta}^{**} \setminus T_\eta} | \phi_{\eta} \nabla v|^2 \bigg)^\frac12.
\end{equation}
It follows that
\begin{equation}\label{est.phietaDv.final}
    \| \phi_{\eta,\e} \nabla v \|_{L^\infty(k+ T_{\eta}^* \setminus T_\eta)} \le C \bigg( \fint_{k + T_{\eta}^{**} \setminus T_\eta} | \phi_{\eta,\e} \nabla v|^2 \bigg)^\frac12 + C\eta \e^2 \bigg( \fint_{k + T_{\eta}^{**} \setminus T_\eta} |f(\e x)|^p \bigg)^{\frac{1}{p}}.
\end{equation}
Finally, \eqref{est.v.eta} and \eqref{est.phietaDv.final} imply the desired estimates by rescaling back to $u_{\eta,\e}$.
\end{proof}

%\jz{The local Lipschitz estimate is still valid when the hole intersects with the boundary $\partial \Omega$. This is because the local domain satisfies the exterior ball condition. This implies the Lipschitz estimate by a barrier argument.}

\begin{lemma}\label{lem.bdry Linfinity.eta}
    Let $p>d$. Let $u_{\eta,\e} \in V_{\eta,\e}$ be a weak solution of \eqref{eq.ue.eta}. Let $k\in \Z^d$ be such that $\e(k+T_{\eta}) \cap \partial \Omega \neq \emptyset$. Then
\begin{equation}\label{est.bdry.Qe}
\begin{aligned}
    \sup_{\e(k+T_{\eta}^* \setminus T_\eta) \cap \Omega_{\eta,\e}} |u_{\eta,\e}| & \le C\bigg( \fint_{\e(k+T_{\eta}^{**} \setminus T_\eta) \cap \Omega_{\eta,\e} } |\phi_{\eta,\e} u_{\eta,\e}|^2 \bigg)^{1/2} \\
    & \qquad +C(\e \eta)^{2} \bigg( \fint_{\e(k+T_{\eta}^{**} \setminus T_\eta) \cap \Omega_{\eta,\e}} |f|^p \bigg)^{1/p}, 
    \end{aligned}
\end{equation}
and
\begin{equation}
\begin{aligned}
    \sup_{\e(k+T_{\eta}^* \setminus T_\eta) \cap \Omega_{\eta,\e}} |\phi_{\eta,\e} \nabla u_{\eta,\e}| & \le C\bigg( \fint_{\e(k+T_{\eta}^{**} \setminus T_\eta) \cap \Omega_{\eta,\e}} |\phi_{\eta,\e} \nabla u_{\eta,\e}|^2 \bigg)^{1/2} \\
    & \qquad +C\e\eta \bigg( \fint_{\e(k+T_{\eta}^{**} \setminus T_\eta) \cap \Omega_{\eta,\e} } |f|^p \bigg)^{1/p}.
    \end{aligned}
\end{equation}
\end{lemma}
\begin{proof}
    We apply a similar argument as Lemma \ref{lem.Local Linfinity.eta}. Let $\hat{v}$ be the same as in \eqref{eq.hatv}. We consider a boundary cell $k + Y_\eta$ such that $(k + T_\eta) \cap \e^{-1} \partial \Omega \neq \emptyset$. Without loss of generality assume $T_\eta$ has only one hole. By the geometric assumption \textbf{A}, $(k + T_\eta) \cap \e^{-1} \Omega$ is a Lipschitz domain. Moreover, since both $\Omega$ and $T_\eta$ are $C^{1,1}$, the intersection $(k + T_\eta) \cap \e^{-1} \Omega$ satisfies the exterior ball condition.

    By the classical $L^\infty$ estimate of elliptic equation in Lipschitz domains, we have
    \begin{equation}
    \| \hat{v} \|_{L^\infty((k + T_{\eta}^* \setminus T_\eta) \cap \e^{-1}\Omega}) \le C \bigg( \fint_{(k + T_{\eta}^{**} \setminus T_\eta) \cap \e^{-1}\Omega} |\hat{v}|^2 \bigg)^\frac12 + C\eta^2 \e^2 \bigg( \fint_{(k + T_{\eta}^{**} \setminus T_\eta) \cap \e^{-1}\Omega} |f(\e x)|^p \bigg)^{\frac{1}{p}}.
\end{equation}
    To obtain the gradient estimate, we can apply a barrier argument in the Lipschitz domain satisfying the exterior ball condition to obtain
    \begin{equation}\label{est.bdry.hatv}
    \begin{aligned}
    |\hat{v}(x)| & \le \frac{C \text{dist}(x, k+T_\eta)}{\eta} \\
    & \qquad \times \bigg\{ \bigg( \fint_{(k + T_{\eta}^{**} \setminus T_\eta) \cap \e^{-1}\Omega} |\hat{v}|^2 \bigg)^\frac12 + \eta^2 \e^2 \bigg( \fint_{(k + T_{\eta}^{**} \setminus T_\eta) \cap \e^{-1}\Omega} |f(\e x)|^p \bigg)^{\frac{1}{p}} \bigg\}.
    \end{aligned}
\end{equation}
Then for each $x\in (k + T_{\eta}^{**} \setminus T_\eta) \cap \e^{-1}\Omega$ with $r = \text{dist}(x, k+T_\eta)$, we apply the interior Lipschitz estimate to get for any $x\in (k+T_\eta^* \setminus T_\eta) \cap \e^{-1} \Omega$,
\begin{equation}\label{est.bdry.Dhatv}
\begin{aligned}
    |\nabla \hat{v}(x)| & \le 
    \frac{C}{r}\bigg( \fint_{(k + T_{\eta}^{**} \setminus T_\eta) \cap B_{r/2}(x)} | \hat{v}|^2 \bigg)^\frac12 + Cr \e^2 \bigg( \fint_{(k + T_{\eta}^{**} \setminus T_\eta) \cap B_{r/2}(x)} |f(\e x)|^p \bigg)^{\frac{1}{p}} \\
    & \le 
    \frac{C}{\eta}\bigg( \fint_{(k + T_{\eta}^{**} \setminus T_\eta) \cap \e^{-1}\Omega} | \hat{v}|^2 \bigg)^\frac12 + C\eta \e^2 \bigg( \fint_{(k + T_{\eta}^{**} \setminus T_\eta) \cap \e^{-1}\Omega} |f(\e x)|^p \bigg)^{\frac{1}{p}}.
    \end{aligned}
\end{equation}
Notice that \eqref{est.bdry.hatv} and $\eqref{est.bdry.Dhatv}$ are analogs of \eqref{est.hatv.eta} and \eqref{est.Dhatv.eta}. Then the desired estimates follow from a similar argument as Lemma \ref{lem.Local Linfinity.eta}.
\end{proof}

\section{First-order approximation by harmonic extension}
\label{sec.4}

% In this section, we will establish the convergence rate for the degenerate elliptic equation \eqref{eq.weighted elliptic}. In particular, we consider the 
% boundary value problem,
% \begin{equation}\label{eq.ue.phief}
%     \mathcal{L}_{\eta,\e}(u_{\eta,\e}) = \phi_{\eta,\e} f \quad \text{in } \Omega_{\eta,\e} \quad \text{and} \quad u_{\eta,\e} = 0 \quad \text{on } \Gamma_{\eta,\e},
% \end{equation}
% with $f\in L^2(\Omega_{\eta,\e})$. We drop one $\phi_{\eta,\e}$ on the right-hand side as we may simply replace $f$ by $\phi_{\eta,\e} f$ in applications. Clearly \eqref{eq.ue.phief} is solvable in $H^1_{\phi_{\eta,\e},0}(\Omega)$ with the energy estimate
% \begin{equation*}
%     \| \phi_{\eta,\e} \nabla u_{\eta,\e} \|_{L^2(\Omega_{\eta,\e})} \le C \| f\|_{L^2(\Omega_{\eta,\e})}.
% \end{equation*}

In this subsection, we derive the error estimate of the first-order approximation for the boundary value problem,
\begin{equation}\label{eq.ue}
    \cL_{\eta,\e}(u_{\eta,\e}) = F \quad \text{in } \Omega_{\eta,\e} \quad \text{and}\quad  u_{\eta,\e} = 0 \quad \text{on } \Gamma_{\eta,\e},
\end{equation}
where $F\in L^2(\Omega_{\eta,\e})$. 
We show that the homogenized problem as $\e \to 0$ is
\begin{equation}\label{eq.u00}
    \cL_\eta(u_{\eta}) = \widetilde{F} \quad \text{in } \Omega, \quad \text{and}\quad  u_{\eta} = 0 \quad \text{on } \partial \Omega,
\end{equation}
where $\widetilde{F}$ is the zero extension of $F$ in $\Omega$. We mention that for the degenerate equation \eqref{eq.ue} in the perforated domain $\Omega_{\eta,\e}$, a natural boundary condition on $\Sigma_{\eta,\e} = \partial T_\e \cap \Omega$ is satisfied automatically via the variational form of the weak solutions analogous to \eqref{eq.weaksol}.

%Throughout this subsection, we only assume that $\Omega$ is a Lispchitz domain satisfying the geometric assumption $H$.

Let $c_1 \in (0,\frac14 c_0]$ be a constant. Let $\theta_\e \in C_0^\infty(\Omega)$ be a cutoff function such that $\theta_\e = 1$ if $\dist(x,\partial \Omega)>2c_1 \e$, $\theta_\e = 0$ if $\dist(x,\partial \Omega)< c_1 \e$, and $|\nabla\theta_\e| \le C\e^{-1}$. Define
\begin{equation*}
    \Omega(t\e) = \left\{ x\in \Omega: \dist(x,\partial \Omega) < tc_1 \e \right\}.
\end{equation*}
Observe that $\nabla \theta_\e$ is supported in a thin layer $\Omega(2\e) \setminus \Omega(\e) = \{ x\in \Omega: c_1 \e \le \dist(x,\partial \Omega) < 2c_1 \e \}$. 

Another technical tool we need is the smoothing operator. Let $0\le \zeta \in C_0^\infty(B_{c_1}(0))$ and $\int_{B_{c_1}(0)} \zeta = 1$ and define the standard smoothing operator by
\begin{equation*}
    \mathscr{K}_\e (f)(x) = \int_{\mathbb{R}^d}  \e^{-d} \zeta(\frac{x-y}{\e})f(y) dy.
\end{equation*}
Many properties about the above smoothing operators can be found in \cite[Chapter 3.1]{Shen18}. We refer to the Appendix of \cite{SZ23} for several lemmas that will be used below.

Let
\begin{equation}\label{def.we}
    w_{\eta,\e} = u_{\eta,\e} - u_{\eta} - \e \chi_\eta^\ell(x/\e) \mathscr{K}_\e(\partial_\ell u_{\eta}) \theta_\e.
\end{equation}
Since $\widetilde{F}\in L^2(\Omega)$, $u_{\eta}\in H_{\rm loc}^2(\Omega)$. 

%Also by the local elliptic estimates in Proposition \ref{prop.chi}, we know $\chi_\eta^\ell$ and $\phi_{\eta} \nabla \chi_\eta^\ell$ are both bounded. 

\begin{lemma}\label{lem.Dwe.Dh}
    Assume $u_\eta \in W^{2,d}(\Omega)$. Then for any $h\in H^1_0(\Omega)$, we have
    \begin{equation}\label{est.DweDh}
    \begin{aligned}
        & \bigg| \int_{\Omega_{\eta,\e}} \phi_{\eta,\e}^2 \nabla w_{\eta,\e} \cdot \nabla h \bigg| \\
        & \le C \e \eta^\frac{d-2}{2} \| \nabla^2 u_{\eta} \|_{L^d(\Omega)} \| \nabla h \|_{L^2(\Omega )} + C\e^\frac{1}{2} \eta^{\frac{d-2}{2}} \| \nabla u_{\eta} \|_{W^{1,d}(\Omega)} \| \nabla h \|_{L^2(\Omega(2\e))} .
    \end{aligned}
    \end{equation}
\end{lemma}

\begin{proof}
First of all,  $h\in H^1_0(\Omega)$ implies $h|_{\Omega_{\eta,\e}} \in V_{\eta,\e}$. Thus $h$ can be used as a test function.
We calculate directly
\begin{equation}\label{eq.Dwe.Dh}
\begin{aligned}
    & \int_{\Omega_{\eta,\e}} \phi_{\eta,\e}^2 \nabla w_{\eta,\e} \cdot \nabla h \\
    & = \int_{\Omega_{\eta,\e}} \phi_{\eta,\e}^2 \nabla u_{\eta,\e} \cdot \nabla h - \int_{\Omega_{\eta,\e}} \phi_{\eta,\e}^2 \nabla u_{\eta} \cdot \nabla h  - \int_{\Omega_{\eta,\e}} \phi_{\eta,\e}^2 (\nabla \chi_\eta^\ell)_\e \mathscr{K}_\e(\partial_\ell u_{\eta}) \theta_\e \cdot \nabla h \\
    & \qquad - \e \int_{\Omega_{\eta,\e}} \phi_{\eta,\e}^2 (\chi_\eta^\ell)_\e \mathscr{K}_\e(\nabla \partial_\ell u_{\eta}) \theta_\e \cdot \nabla h - \e \int_{\Omega_{\eta,\e}} \phi_{\eta,\e}^2 (\chi_\eta^\ell)_\e \mathscr{K}_\e(\partial_\ell u_{\eta}) \nabla \theta_\e \cdot \nabla h.
\end{aligned}
\end{equation}
Using the equations \eqref{eq.ue} and \eqref{eq.u00}, we have
\begin{equation*}
    \int_{\Omega_{\eta,\e}} \phi_{\eta,\e}^2 \nabla u_{\eta,\e} \cdot \nabla h = \int_{\Omega} \overline{A}_\eta\nabla u_{\eta} \cdot \nabla h. 
\end{equation*}
Inserting this equation into \eqref{eq.Dwe.Dh}, we obtain
\begin{equation}\label{eq.Dwe.Dh2}
\begin{aligned}
    & \int_{\Omega_{\eta,\e}} \phi_{\eta,\e}^2 \nabla w_{\eta,\e} \cdot \nabla h \\
    & = \int_{\Omega} (\overline{A}_\eta - \phi_{\eta,\e}^2 I - \phi_{\eta,\e}^2 (\nabla \chi_\eta )_\e ) \mathscr{K}_\e(\nabla u_{\eta}) \theta_\e \cdot \nabla h \\
%    & \qquad + \int_{\Omega_{\eta,\e}} \phi_{\eta,\e}^2 (\nabla \chi_\eta^\ell)_\e ( \partial_\ell u_{\eta} - \mathscr{K}_\e(\partial_\ell u_{\eta}) ) \theta_\e \cdot \nabla h \\
    &\qquad + \int_{\Omega} (\overline{A}_\eta - \phi_{\eta,\e}^2 I ) \nabla u_{\eta} (1-\theta_\e) \cdot \nabla h\\
    & \qquad + \int_{\Omega} (\overline{A}_\eta - \phi_{\eta,\e}^2 I ) (\nabla u_{\eta} -\mathscr{K}_\e(\nabla u_{\eta})) \theta_\e \cdot \nabla h \\
    &\qquad  - \e \int_{\Omega_{\eta,\e}} \phi_{\eta,\e}^2 (\chi_\eta^\ell)_\e \mathscr{K}_\e(\nabla \partial_\ell u_{\eta}) \theta_\e \cdot \nabla h
     - \e \int_{\Omega_{\eta,\e}} \phi_{\eta,\e}^2 (\chi_\eta^\ell)_\e  \mathscr{K}_\e( \partial_\ell u_{\eta}) \nabla \theta_\e \cdot \nabla h \\
     & =: \sum_{i=1}^5 I_i.
\end{aligned}
\end{equation}

We estimate $I_k$ for $1\le k\le 5$.
To estimate $I_1$, we define the flux correctors. Let $B_\eta = \overline{A}_\eta - \phi^2_{\eta} I - \phi^2_{\eta} \nabla \chi_\eta$, or in component form
\begin{equation*}
    (b_\eta)_{ij} = (\bar{a}_\eta)_{ij} - \phi^2_{\eta} \delta_{ij} - \phi^2_{\eta} \partial_i \chi_\eta^j.
\end{equation*}
Clearly, by Proposition \ref{prop.chi},  $(b_\eta)_{ij} \in L^2(Y)$ with obvious zero extension, and
\begin{equation*}
    \int_{Y} (b_\eta)_{ij} = 0, \qquad \partial_i (b_\eta)_{ij} = 0.
\end{equation*}
Then it is well-known (see, e.g., \cite{Shen18}) that we can find the flux correctors $\Phi_{\eta} = (\Phi_{\eta})_{kij} \in H_{\rm per}^1(Y)$ such that
\begin{equation*}
    (b_\eta)_{ij} = \partial_k (\Phi_{\eta})_{kij}, \qquad (\Phi_{\eta})_{kij} = -(\Phi_{\eta})_{ikj}.
\end{equation*}
Moreover, $(\Phi_{\eta})_{kij} \in C^\alpha(Y)$ for any $\alpha \in (0,1)$. 

Note that we can write
\begin{equation}
    B_\eta = (\overline{A}_\eta - I) + I(1-\phi_\eta^2) + \phi_\eta^2 \nabla \chi_\eta.
\end{equation}
By the smallness of $\phi_{\eta} - 1, \overline{A}_\eta - I, \phi_{\eta} \nabla \chi_\eta $ and $ \chi_\eta$ obtained in Lemma \ref{lemma-7.2} (i) and Proposition \ref{prop.chi} (i)(iv), and the construction of $\Phi_\eta$, we have
\begin{equation}\label{est.allsmall}
    \| \phi_{\eta} -1 \|_{L^{2^*}(Y_{\eta})} + |\overline{A}_\eta - I| + \|  \phi_{\eta} \nabla \chi_\eta \|_{L^2(Y_{\eta})} + \| \chi_\eta \|_{L^2(Y_{\eta})} + \|  \Phi_{\eta} \|_{L^{2^*}(Y_{\eta})}\le C\eta^{\frac{d-2}{2}}.
\end{equation}
Therefore, by the skew-symmetry of $\Phi_\eta$ and a standard argument,
\begin{equation*}
\begin{aligned}
    I_1
    & = \int_{\Omega} \partial_k(\e (\Phi_{\eta,\e})_{kij} ) \mathscr{K}_\e(\partial_j u_{\eta})( \partial_i h) \theta_\e \\
    & = - \e \int_\Omega (\Phi_{\eta,\e})_{kij} \mathscr{K}_\e(\partial_k\partial_j u_{\eta})( \partial_i h ) \theta_\e -\e \int_{\Omega} (\Phi_{\eta,\e})_{kij} \mathscr{K}_\e (\partial_j u_{\eta}) \partial_i h \partial_k \theta_\e.
\end{aligned}
\end{equation*}
The first integral on the right-hand side is bounded by
\begin{equation}\label{est.Du0.int}
   C \e \| \Phi_\eta \|_{L^{2}(Y^*)}  \| \nabla^2 u_{\eta} \|_{L^2(\Omega)} \| \nabla h \|_{L^2(\Omega )} \le C\e \eta^{\frac{d-2}{2}} \| \nabla^2 u_{\eta} \|_{L^2(\Omega)} \| \nabla h \|_{L^2(\Omega )}.
\end{equation}
The second integral is bounded by
\begin{equation}\label{est.Du0.layer}
    C \e^\frac12 \| \Phi_\eta \|_{L^{2}(Y^*)}  \| \nabla u_{\eta} \|_{H^1( \Omega)} \| \nabla h \|_{L^2(\Omega(2\e) )} \le C\e^\frac12 \eta^{\frac{d-2}{2}} \| \nabla u_{\eta} \|_{H^1( \Omega)} \| \nabla h \|_{L^2(\Omega(2\e) )}.
\end{equation}
Consequently,
\begin{equation}
    |I_1| \le C\e \eta^{\frac{d-2}{2}} \| \nabla^2 u_{\eta} \|_{L^2(\Omega)} \| \nabla h \|_{L^2(\Omega )} + C\e^\frac12 \eta^{\frac{d-2}{2}} \| \nabla u_{\eta} \|_{H^1( \Omega)} \| \nabla h \|_{L^2(\Omega(2\e) )}.
\end{equation}

To estimate $I_2$, we apply \eqref{est.allsmall} and the H\"{o}lder's inequality to obtain
\begin{equation}
\begin{aligned}
    |I_2| & \le C\Big\{ \| \overline{A}_\eta - I \|_{L^{2^*}(\Omega(2\e))} + \| \phi_{\eta,\e}^2 - 1 \|_{L^{2^*}(\Omega(2\e))} \Big\} \| \nabla u_\eta \|_{L^d(\Omega(2\e))} \| \nabla h \|_{L^2(\Omega(2\e))} \\
    & \le C \eta^{\frac{d-2}{2}} \e^{\frac{1}{2^*}}\cdot \e^\frac{1}{d} \| \nabla u_\eta \|_{W^{1,d}(\Omega)} \| \nabla h \|_{L^2(\Omega(2\e))} \\
    & \le C \eta^{\frac{d-2}{2}} \e^\frac{1}{2} \| \nabla u_\eta \|_{W^{1,d}(\Omega)} \| \nabla h \|_{L^2(\Omega(2\e))}.
\end{aligned}
\end{equation}

We estimate $I_3$ as follows:
\begin{equation}
\begin{aligned}
    |I_3| & \le C\Big\{ \| \overline{A}_\eta - I \|_{L^{2^*}(\Omega)} + \| \phi_{\eta,\e}^2 - 1 \|_{L^{2^*}(\Omega)} \Big\} \| \nabla u_\eta - \mathscr{K}_\e(\nabla u_{\eta}) \|_{L^d(\Omega \setminus \Omega(\e))} \| \nabla h \|_{L^2(\Omega)} \\
    & \le C\e \eta^\frac{d-2}{2} \| \nabla^2 u_\eta \|_{L^d(\Omega)} \| \nabla h \|_{L^2(\Omega)}.
\end{aligned}
\end{equation}

Finally, note that $I_4$ has the same bound as \eqref{est.Du0.int}, and $I_5$ has the same bound as \eqref{est.Du0.layer}. Summing up all these estimates, we obtain the desired estimate.
\end{proof}

% \begin{corollary}
%     For any $h\in H^1_0(\Omega)$, we have
%     \begin{equation}\label{est.DweDh}
%         \bigg| \int_{\Omega_{\eta,\e}} \phi_{\eta,\e}^2 \nabla w_{\eta,\e} \cdot \nabla h \bigg| \le C \e^{\frac12} \| \nabla h \|_{L^2(\Omega )} \big( \| F \|_{L^2(\Omega)} + \| g \|_{H^1(\partial \Omega)} \big).
%     \end{equation}
% \end{corollary}

Now we are going to pick a particular test function $h$. We point out that $w_{\eta,\e}$ itself, even with a cut-off on the boundary is not a legal choice for $h$ since $\nabla u_{\eta,\e}$ may not lie in $L^2(\Omega)^d$ (we only know $\phi_{\eta,\e} \nabla u_{\eta,\e} \in L^2(\Omega)^d$). We will use the harmonic extension to handle the possible singularity of $\nabla u_{\eta,\e}$ near the holes. 
Let $\delta\in (0,c_1]$. Define
\begin{equation*}
    \Omega^\delta_{\eta,\e}: = \left\{x\in \Omega_{\eta,\e}: \dist(x,T_{\eta,\e})>\delta \e \eta \right \}.
\end{equation*}
Note that $|\Omega_{\eta,\e} \setminus \Omega^\delta_{\eta,\e}| \le C\delta \eta^d$. We shall 
extend the function $u_{\eta,\e}$ and $\chi_\eta(x/\e)$ from $\Omega^\delta_{\eta,\e}$ to $\Omega$ with a suitable choice of $\delta$. We define $T^\delta_\eta = \{ x\in Y: \dist(x,T_\eta) \le \delta \eta \}$ and $Y^{\delta}_\eta = Y\setminus T^\delta_\eta$.

\begin{lemma}\label{lem.Dwet}
Let $\Omega_{\eta,\e}^\delta$ be given as above. Let $u_{\eta,\e}^*$ be the harmonic extension of $u_{\eta,\e}$ from $\Omega_{\eta,\e}^\delta$ to $\Omega_{\eta,\e}$. It holds
    \begin{equation}\label{est.FullRate}
        \begin{aligned}
            \int_{\Omega_{\eta,\e}} \phi_{\eta,\e}^2 |\nabla w_{\eta,\e}|^2 & \le C\e^\frac12 \delta^{-1} \eta^\frac{d-2}{2} \| \nabla^2 u_{\eta} \|_{W^{1,d}(\Omega)} \big( \| \phi_{\eta,\e} \nabla u_{\eta,\e} \|_{L^2(\Omega_{\eta,\e})} + \| \nabla u_{\eta}\|_{L^2(\Omega)} \big) \\
        & \qquad + C\e \eta^{d-2} \| \nabla u_{\eta} \|_{W^{1,d}(\Omega)}^2  + C\delta \eta^{d-2} \| \nabla u_{\eta}\|_{L^2(\Omega)}^2 \\
        & \qquad + C\int_{\Omega_{\eta,\e} \setminus \Omega_{\eta,\e}^\delta} \phi_{\eta,\e}^2 (|\nabla u_{\eta,\e}|^2 + |\nabla u_{\eta,\e}^*|^2).
        \end{aligned}
    \end{equation}
    % \begin{equation}
    % \begin{aligned}
    %     \int_{\Omega_{\eta,\e}^\delta} \phi_{\eta,\e}^2 |\nabla w_{\eta,\e}^*|^2 & \le C\e\delta^{-1} \| \nabla^2 u_{\eta} \|_{L^2(\Omega\setminus \Omega(c\e))} \| u_{\eta,\e} \|_{H^1_{\phi_{\eta,\e}}(\Omega)}  + C\e \| \nabla^2 u_{\eta} \|_{L^2(\Omega\setminus \Omega(c\e))} \| \nabla u_{\eta} \|_{L^2(\Omega)} \\
    %     & \qquad + C\delta \| \nabla u_{\eta} \|_{L^2(\Omega)}^2 + C \| \nabla u_{\eta} \|_{L^2(\Omega(2\e))}^2 + C \int_{\Omega_{\eta,\e} \setminus \Omega_{\eta,\e}^\delta} \phi_{\eta,\e}^2 |\nabla u_{\eta,\e}|^2 + C\int_{\Omega_{\eta,\e} \setminus \Omega_{\eta,\e}^\delta} |u_{\eta,\e}|^2.
    % \end{aligned}
    % \end{equation}
\end{lemma}

\begin{proof}
     Let $\chi_{\eta}^* = (\chi_\eta^{\ell *} )$ be the harmonic extension of $\chi_{\eta}$ from $Y_{\eta}^\delta$ to $Y$, and $\delta \in (0,1)$ be a small parameter that can vary. Recall that we sometimes also use the notations $\chi_{\eta,\e}^\ell =(\chi_\eta^\ell)_\e = \chi_\eta^\ell(x/\e)$ and $(\nabla \chi_\eta )_\e = (\nabla \chi_\eta)(x/\e)$ (similar notations also apply to $\chi_\eta^{*}$ and $\nabla \chi_\eta^{*}$).
     
     Let $\theta_\e$ be given as before. Define
    \begin{equation*}
        w_{\eta,\e}^* = u_{\eta,\e}^* - u_{\eta} - \e \chi_\eta^{\ell*}(x/\e) \mathscr{K}_\e(\partial_\ell u_{\eta}) \theta_\e.
    \end{equation*}
    % and
    % \begin{equation*}
    % w_{\eta,\e}^* = u_{\eta,\e}^*  - u_{\eta} -\e \chi_\eta^{\ell*}(x/\e) \mathscr{K}_\e(\partial_\ell u_{\eta}) \theta_\e.
    % \end{equation*}
    % \jz{We need to modify the harmonic extension near $\partial \Omega$ such that $u_{\eta,\e}^* = 0$ on $\partial \Omega$. This can be done.}
    It is easy to verify that $w_{\eta,\e}^* \in H^1_0(\Omega)$. Write
    \begin{equation}\label{eq.Dwe.I123}
    \begin{aligned}
        & \int_{\Omega_{\eta,\e}} \phi_{\eta,\e}^2 \nabla w_{\eta,\e} \cdot \nabla w_{\eta,\e} \\
        & = \int_{\Omega_{\eta,\e}} \phi_{\eta,\e}^2 \nabla w_{\eta,\e} \cdot \nabla w_{\eta,\e}^* + \int_{\Omega_{\eta,\e}} \phi_{\eta,\e}^2 \nabla w_{\eta,\e} \cdot \nabla (w_{\eta,\e} - w_{\eta,\e}^*)\\
        & = J_1 + J_2.        
    \end{aligned}
    \end{equation}

    \textbf{Estimate of $J_1$:} Since $w_{\eta,\e}^* \in H^1_0(\Omega)$, we apply \eqref{lem.Dwe.Dh} to obtain
    \begin{equation}\label{est.I1}
        |J_1| \le C \e \eta^\frac{d-2}{2} \| \nabla^2 u_{\eta} \|_{L^d(\Omega)} \| \nabla w_{\eta,\e}^* \|_{L^2(\Omega )} + C \e^\frac{1}{2} \eta^\frac{d-2}{2} \| \nabla u_{\eta} \|_{W^{1,d}(\Omega)} \| \nabla w_{\eta,\e}^* \|_{L^2(\Omega(2\e))} .
    \end{equation}
    % Taking the test function $h = w_{\eta,\e}^*$ in \eqref{est.DweDh} and using the fact $w_{\eta,\e} = w_{\eta,\e}^*$ in $\Omega_{\eta,\e}^\delta$, we have
    % \begin{equation}\label{est.wedelta-1}
    % \begin{aligned}
    %     \int_{\Omega_{\eta,\e}^\delta} \phi_{\eta,\e}^2 |\nabla w_{\eta,\e}^*|^2 & \le C\e\| \nabla^2 u_{\eta}\|_{L^2(\Omega \setminus \Omega(\e) )} \| \nabla w_{\eta,\e}^* \|_{L^2(\Omega)} + C\| \nabla u_{\eta}\|_{L^2(\Omega(2\e))} \| \nabla w_{\eta,\e}^* \|_{L^2(\Omega(2\e))} \\
    %     & \qquad + \bigg| \int_{\Omega_{\eta,\e} \setminus \Omega_{\eta,\e}^\delta} \phi_{\eta,\e}^2 \nabla w_{\eta,\e} \cdot \nabla w_{\eta,\e}^* \bigg|.
    % \end{aligned}
    % \end{equation}
    By the triangle inequality,
    \begin{equation}\label{est.Dwedelta}
    \begin{aligned}
        & \| \nabla w_{\eta,\e}^* \|_{L^2(\Omega)} \\
        & \le \| \nabla u_{\eta,\e}^* \|_{L^2(\Omega)} + \| \nabla u_{\eta} \|_{L^2(\Omega)} + \e \| \nabla (\chi_\eta^{\ell*}(x/\e) \mathscr{K}_\e(\partial_\ell u_{\eta}) \theta_\e) \|_{L^2(\Omega)} \\
        & \le \| \nabla u_{\eta,\e}^* \|_{L^2(\Omega)} + \| \nabla u_{\eta} \|_{L^2(\Omega)} + \| \chi_{\eta,\e}^* \mathscr{K}_\e( \nabla u_{\eta}) \|_{L^2( \Omega(2\e))} \\
        & \qquad + C\| (\nabla\chi_\eta^{*})_\e \mathscr{K}_\e(\nabla u_{\eta}) \|_{L^2(\Omega\setminus \Omega(\e))} + C\e\| \chi_{\eta,\e}^*  \mathscr{K}_\e( \nabla^2 u_{\eta}) \|_{L^2(\Omega\setminus \Omega(\e))}.
%        & \le C\delta^{-1} \| \phi_{\eta,\e} \nabla u^\e \|_{L^2(\Omega_{\eta,\e})} + C\delta^{-1} \| u_{\eta,\e} \|_{L^2(\Omega)} + C\| \nabla u_{\eta} \|_{L^2(\Omega)} + C\e \| \nabla^2 u_{\eta} \|_{L^2(\Omega\setminus \Omega(\e))} \\
%        & \le C\delta^{-1} \| u_{\eta,\e} \|_{H^1_{\phi_{\eta,\e}}(\Omega_{\eta,\e})} + C\| \nabla u_{\eta} \|_{L^2(\Omega)} + C\e \| \nabla^2 u_{\eta} \|_{L^2(\Omega\setminus \Omega(\e))},
    \end{aligned}
    \end{equation}
    The $L^p$ boundedness of the harmonic extension (see \eqref{est.Dtu.Exteta}) implies
    \begin{equation}\label{est.uee*}
        \| \nabla u_{\eta,\e}^* \|_{L^2(\Omega)} \le C\| \nabla u_{\eta,\e} \|_{L^2(\Omega_{\eta,\e}^\delta)} \le C\delta^{-1} \| \phi_{\eta,\e} \nabla u_{\eta,\e} \|_{L^2(\Omega_{\eta,\e})} \le C\delta^{-1},
    \end{equation}
    where we used the fact $\phi_{\eta,\e} \ge c\delta$ on $\Omega_{\eta,\e}^\delta$. The bounds of $\chi_\eta$ and $ \phi_\eta \nabla \chi_\eta$ (see Proposition \ref{prop.chi}) imply
    \begin{equation*}
        \| \chi_\eta^* \|_{L^2(Y)} \le \| {\chi_\eta} \|_{L^2(Y^\delta_\eta)} + C\| \chi_\eta \|_{L^\infty(Y_\eta)} |Y \setminus Y_\eta^\delta |^{\frac12} \le C\eta^{\frac{d-2}{2}},
    \end{equation*}
    and
    \begin{equation*}
        \| \nabla \chi_\eta^{*} \|_{L^2(Y)} \le C \| \nabla {\chi} \|_{L^2(Y^\delta_\eta)} \le C\delta^{-1} \| \phi_\eta \nabla {\chi} \|_{L^2(Y^\delta_\eta)}  C\delta^{-1} \eta^{\frac{d-2}{2}},
    \end{equation*}
    Hence, we can use a property of the smoothing operator (see Lemma \cite[Lemma A.1]{SZ23}) to obtain
    the estimates of the last three terms in \eqref{est.Dwedelta}, 
    \begin{equation*}
    \begin{aligned}
       & \| \chi_{\eta,\e}^* \mathscr{K}_\e( \nabla u_{\eta}) \|_{L^2( \Omega(2\e))} + C\| (\nabla\chi_\eta^{*})_\e \mathscr{K}_\e(\nabla u_{\eta}) \|_{L^2(\Omega\setminus \Omega(\e))} + C\e\| \chi_{\eta,\e}^*  \mathscr{K}_\e( \nabla^2 u_{\eta}) \|_{L^2(\Omega\setminus \Omega(\e))} \\
       & \le C \| \chi_\eta^{*} \|_{L^2(Y)} \| \nabla u_{\eta} \|_{L^2(\Omega)} + C\| \nabla \chi_\eta^{*} \|_{L^2(Y)} \| \nabla u_{\eta} \|_{L^2(\Omega)} + C \e \| \chi_\eta^{*} \|_{L^2(Y)} \| \nabla^2 u_{\eta} \|_{L^2(\Omega)} \\
       & \le C\eta^\frac{d-2}{2} \| \nabla u_{\eta} \|_{L^2(\Omega)} + C\delta^{-1} \eta^\frac{d-2}{2} \| \nabla u_\eta \|_{L^2(\Omega)} + C\e \eta^\frac{d-2}{2} \| \nabla^2 u_{\eta} \|_{L^2(\Omega)}.
    \end{aligned}  
    \end{equation*}

    Inserting this and \eqref{est.uee*} into \eqref{est.Dwedelta}, we obtain
    \begin{equation}
    \begin{aligned}
        \| \nabla w_{\eta,\e}^* \|_{L^2(\Omega)} & \le  C\delta^{-1} \| \phi_{\eta,\e} \nabla u_{\eta,\e} \|_{L^2(\Omega_{\eta,\e})} \\
        & \qquad  + C(1+ \delta^{-1} \eta^\frac{d-2}{2}) \| \nabla u_\eta \|_{L^2(\Omega)} + C\e \eta^\frac{d-2}{2} \| \nabla^2 u_{\eta} \|_{L^2(\Omega)}.
    \end{aligned}
    \end{equation}

    As a result, we have
    \begin{equation*}
    \begin{aligned}
        |J_1| & \le C\e^\frac12 \eta^\frac{d-2}{2} \delta^{-1} \| \nabla^2 u_{\eta} \|_{W^{1,d}(\Omega)} \Big\{ \| \phi_{\eta,\e} \nabla u_{\eta,\e} \|_{L^2(\Omega_{\eta,\e})} + \| \nabla u_{\eta}\|_{L^2(\Omega)} \Big\} \\
        & \qquad + C\e^{\frac32} \eta^{d-2} \| \nabla u_{\eta} \|_{W^{1,d}(\Omega)}^2.
    \end{aligned}
    \end{equation*}

    \textbf{Estimate of $J_2$:} 
    Recall that $w_{\eta,\e} = w_{\eta,\e}^*$ in $\Omega_{\eta,\e}^\delta$. Thus $\phi_{\eta,\e} \nabla (w_{\eta,\e} -w_{\eta,\e}^*)$ is supported in $\Omega_{\eta,\e} \setminus \Omega_{\eta,\e}^\delta$. Moreover,
    \begin{equation}\label{est.I3-2}
    \begin{aligned}
        \phi_{\eta,\e} \nabla (w^*_\e -w_{\eta,\e}^*) & = \phi_{\eta,\e} \nabla u_{\eta,\e} - \phi_{\eta,\e} \nabla u_{\eta,\e}^* \\
        & \qquad - \phi_{\eta,\e} \nabla (\e \chi_\eta^\ell(x/\e) \mathscr{K}_\e(\partial_\ell u_{\eta}) \theta_\e ) + \phi_{\eta,\e} \nabla (\e \chi_\eta^{\ell*}(x/\e) \mathscr{K}_\e(\partial_\ell u_{\eta}) \theta_\e ).
    \end{aligned}
    \end{equation}
    We estimate the last term of the above identity over $\Omega_{\eta,\e} \setminus \Omega_{\eta,\e}^\delta$. Note that $\phi_{\eta,\e} \le C\delta$ in $\Omega_{\eta,\e} \setminus \Omega_{\eta,\e}^\delta$. It follows from the triangle inequality and the boundedness of $\chi_\eta^{*}$,
    \begin{equation*}
    \begin{aligned}
        & \int_{\Omega_{\eta,\e} \setminus \Omega_{\eta,\e}^\delta} \phi_{\eta,\e}^2 |\nabla (\e \chi_\eta^{\ell*}(x/\e) \mathscr{K}_\e(\partial_\ell u_{\eta}) \theta_\e )|^2 \\
        & \le C\delta^2 \int_{\Omega_{\eta,\e} \setminus \Omega_{\eta,\e}^\delta} |\e \chi_\eta^{\ell*}(x/\e) \mathscr{K}_\e(\nabla \partial_\ell u_{\eta}) \theta_\e |^2 
        \\
        & \qquad + C\delta^2 \int_{\Omega_{\eta,\e} \setminus \Omega_{\eta,\e}^\delta} |\e \chi_\eta^{\ell*}(x/\e) \mathscr{K}_\e( \partial_\ell u_{\eta}) \nabla \theta_\e |^2 + C\delta^2 \int_{\Omega_{\eta,\e} \setminus \Omega_{\eta,\e}^\delta} | \nabla \chi_\eta^{\ell*}(x/\e) \mathscr{K}_\e( \partial_\ell u_{\eta}) \theta_\e |^2
        \\
        & \le C\delta^2 \e^2 \| \chi_\eta^* \|_{L^2(T_\eta^\delta \setminus T_\eta )}^2  \| \nabla^2 u_\eta \|_{L^2(\Omega)}^2  \\
        & \qquad + C\delta^2 \e \| \chi_\eta^* \|_{L^2(T_\eta^\delta \setminus T_\eta)}^2  \| \nabla u_\eta \|_{H^1(\Omega)}^2 + C\delta^2 \| \nabla \chi_\eta^* \|_{L^2(T_\eta^\delta \setminus T_\eta )}^2  \| \nabla u_\eta \|_{L^2(\Omega)}^2 \\
        & \le C\delta^3 \e^2 \eta^d \| \nabla^2 u_\eta \|_{L^2(\Omega)}^2 + C\delta^3 \e \eta^d \| \nabla u_\eta \|_{H^1(\Omega)}^2 + C\delta^2 \| \nabla \chi_\eta^* \|_{L^2(T_\eta^\delta)}^2  \| \nabla u_\eta \|_{L^2(\Omega)}^2,
    \end{aligned}
    \end{equation*}
    By the $L^2$ regularity estimate of the nontangential maximal function for the harmonic function $\chi_\eta^{*}$ in the Lipschitz holes $T^\delta_\eta = \{ x\in Y: \dist(x,T)<\delta \}$ and the pointwise estimate of $\nabla \chi_\eta$ in Proposition \ref{prop.chi} (ii), we have (see \cite[Lemma A.5]{SZ23})
    \begin{equation*}
        \| \nabla \chi_\eta^{*} \|_{L^2(T^\delta_\eta)}^2  \le C\delta \eta \| \nabla_{\rm \tan} \chi \|_{L^2(\partial T^\delta_\eta)}^2 \le C(\delta \eta)^{-1} |\partial T_\eta^\delta| \le C\delta^{-1} \eta^{d-2}. %    C\delta \| (\nabla \chi_\eta^{*}) \|_{L^2(\partial T^\delta)}^2
    \end{equation*}
    Hence,
    \begin{equation*}
        \int_{\Omega_{\eta,\e} \setminus \Omega_{\eta,\e}^\delta} \phi_{\eta,\e}^2 |\nabla (\e \chi_\eta^{\ell*}(x/\e) \mathscr{K}_\e(\partial_\ell u_{\eta}) \theta_\e )|^2 \le C\delta^3 \e \eta^d \| \nabla u_\eta \|_{H^1(\Omega)}^2 + C\delta \eta^{d-2} \| \nabla u_\eta \|_{L^2(\Omega)}^2 .
    \end{equation*}
    Similarly (and easier), using the boundedness of $\chi$ and $\phi \nabla \chi$, we have
    \begin{equation*}
        \int_{\Omega_{\eta,\e} \setminus \Omega_{\eta,\e}^\delta} \phi_{\eta,\e}^2 |\nabla (\e {\chi}_\ell(x/\e) \mathscr{K}_\e(\partial_\ell u_{\eta}) \theta_\e )|^2 \le C\delta^3 \e \eta^d \| \nabla u_\eta \|_{H^1(\Omega)}^2 + C\delta \eta^{d-2} \| \nabla u_\eta \|_{L^2(\Omega)}^2.
    \end{equation*}
    % Combining these with \eqref{est.I3} and \eqref{est.I3-2}, we obtain
    % \begin{equation*}
    % \begin{aligned}
    %     |I_3| & \le \frac14 \| \phi_{\eta,\e} \nabla w_{\eta,\e} \|_{L^2(\Omega_{\eta,\e})}^2 + C\int_{\Omega_{\eta,\e} \setminus \Omega_{\eta,\e}^\delta} \phi_{\eta,\e}^2 (|\nabla u_{\eta,\e}|^2 + |\nablau_{\eta,\e}^*|^2) \\
    %     & \qquad + C \e^2 \|\nabla^2 u_{\eta}\|_{L^2(\Omega_{\eta,\e} \setminus \Omega(\e))}^2 + C\delta \| \nabla u_{\eta}\|_{L^2(\Omega)}^2.
    % \end{aligned}
    % \end{equation*}
    Hence, it follows from \eqref{est.I3-2} and the last two estimates that
    \begin{equation}
    \begin{aligned}
        |J_2| & \le \frac{1}{2} \int_{\Omega_{\eta,\e}} \phi_{\eta,\e}^2 |\nabla w_{\eta,\e}|^2 + 2 \int_{\Omega_{\eta,\e}} \phi_{\eta,\e}^2 |\nabla (w_{\eta,\e} - w_{\eta,\e}^* )|^2 \\
        & \le \frac{1}{2} \int_{\Omega_{\eta,\e}} \phi_{\eta,\e}^2 |\nabla w_{\eta,\e}|^2 + C\int_{\Omega_{\eta,\e} \setminus \Omega_{\eta,\e}^\delta} (| \phi_{\eta,\e} \nabla u_{\eta,\e}|^2 +| \phi_{\eta,\e} \nabla u_{\eta,\e}^*|^2 )   \\
        & \qquad + C\delta^3 \e \eta^d \| \nabla u_\eta \|_{H^1(\Omega)}^2 + C\delta \eta^{d-2} \| \nabla u_\eta \|_{L^2(\Omega)}^2.
    \end{aligned}
    \end{equation}

    Finally, combining the estimates of $J_1$ and $J_2$, we arrive at \eqref{est.FullRate}.
    \end{proof}

The above lemma indicates that the right-hand side of \eqref{est.FullRate} is small by choosing $\delta$ appropriately small, except for the last integral. The smallness of the last integral follows from the small-scale higher integrability of $\phi_{\eta,\e} \nabla u_{\eta,\e}$ and $\phi_{\eta,\e} \nabla u_{\eta,\e}^*$. This can be achieved if $F = \phi_{\eta,\e} f$ in \eqref{eq.ue}. Precisely, we consider the 
boundary value problem,
\begin{equation}\label{eq.ue.phief}
    \mathcal{L}_{\eta,\e}(u_{\eta,\e}) = \phi_{\eta,\e} f \quad \text{in } \Omega_{\eta,\e} \quad \text{and} \quad u_{\eta,\e} = 0 \quad \text{on } \Gamma_{\eta,\e},
\end{equation}
with $f\in L^p(\Omega_{\eta,\e})$ with $p\ge 2$.  Clearly \eqref{eq.ue.phief} is solvable in $V_{\eta,\e}$ with the energy estimate
\begin{equation*}
    \| \phi_{\eta,\e} \nabla u_{\eta,\e} \|_{L^2(\Omega_{\eta,\e})} \le C \| f\|_{L^2(\Omega_{\eta,\e})}.
\end{equation*}
We will show that \eqref{eq.ue.phief} can be approximated by the equation
\begin{equation}\label{eq.ueta.phief}
    \mathcal{L}_{\eta}(\widetilde{u}_{\eta,\e}) = \phi_{\eta,\e} f \quad \text{in } \Omega \quad \text{and} \quad \widetilde{u}_{\eta,\e} = 0 \quad \text{on } \partial \Omega.
\end{equation}
In order to apply Lemma \ref{lem.Dwet}, we need the following two lemmas.

    \begin{lemma}\label{lem.deltaLayer}
    Let $\Omega$ be a Lipschitz domain. Let $f\in L^p(\Omega_{\eta,\e})$ for some $p >d$ and $u_{\eta,\e}$ a solution of \eqref{eq.ue.phief}. Then
    \begin{equation*}
        \int_{\Omega_{\eta,\e} \setminus \Omega_{\eta,\e}^\delta} \phi_{\eta,\e}^2 |\nabla u_{\eta,\e}|^2 \le C\delta \big( \| \phi_{\eta,\e} \nabla u_{\eta,\e} \|_{L^2(\Omega_{\eta,\e})}^2 + \| f \|_{L^p(\Omega_{\eta,\e})}^2\big).
        %+ \int_{\Omega_{\eta,\e} \setminus \Omega_{\eta,\e}^\delta}  |u_{\eta,\e}|^2
    \end{equation*}
\end{lemma}
\begin{proof}
    This lemma essentially relies only on the small-scale interior Lipschitz estimate, i.e., Lemma \ref{lem.Local Linfinity.eta} and Lemma \ref{lem.bdry Linfinity.eta}. Indeed, for $\delta<c_0/8$,
    \begin{equation*}
        \begin{aligned}
            & \int_{\Omega_{\eta,\e} \setminus \Omega_{\eta,\e}^\delta} \phi_{\eta,\e}^2 |\nabla u_{\eta,\e}|^2 \\
            & = \sum_{\e(z+T_\eta^\delta \setminus T_\eta) \cap \Omega \neq \emptyset } \int_{\e(z+T^\delta_\eta \setminus T_\eta) \cap \Omega} \phi_{\eta,\e}^2 |\nabla u_{\eta,\e}|^2 \\
            & \le \sum_{\e(z+T)\subset \Omega} C(\e \eta)^d \delta \sup_{\e(z+T^\delta_\eta \setminus T_\eta) \cap \Omega} |\phi_{\eta,\e} \nabla u_{\eta,\e}|^2 \\
            & \le \sum_{\e(z+T_\eta^\delta \setminus T_\eta) \cap \Omega \neq \emptyset } C(\e \eta)^d \delta \bigg\{ \fint_{\Omega_{\eta,\e} \cap \e(z+T_\eta^{**} \setminus T_\eta )} |\phi_{\eta,\e} \nabla u_{\eta,\e}|^2 + (\e \eta)^{2} \bigg( \fint_{\Omega_{\eta,\e} \cap \e(z+T_\eta^{**} \setminus T_\eta )} |f|^p \bigg)^{2/p} \bigg\} \\
            & \le \sum_{\e(z+T_\eta^\delta \setminus T_\eta) \cap \Omega \neq \emptyset } C\delta \bigg\{ \int_{\Omega_{\eta,\e} \cap \e(z+Y_*^+)} |\phi_{\eta,\e} \nabla u_{\eta,\e}|^2 +  (\e \eta)^{d+ 2-\frac{2d}{p}} \| f \|_{L^p(\Omega_{\eta,\e})}^2  \bigg\} \\
            & \le C\delta \big( \| \phi_{\eta,\e} \nabla u_{\eta,\e} \|_{L^2(\Omega_{\eta,\e})}^2 + \eta^d (\e \eta) ^{2-\frac{2d}{p}} \| f \|_{L^p(\Omega_{\eta,\e})}^2\big).
        \end{aligned}
    \end{equation*}
    % Similarly, by \eqref{est.local.Qe} we also have
    % \begin{equation}
    %     \int_{\Omega_{\eta,\e} \setminus \Omega_{\eta,\e}^\delta} u_{\eta,\e}^2 \le C\delta \| f \|_{L^q(\Omega_{\eta,\e})}^2.
    % \end{equation}
    The proof is complete.
%    Note that in Lemma \ref{lem.Local Linfinity}, we only have the interior estimate for $Q_{2\e}^\e \subset \Omega_{\eta,\e}$. But actually, $\Omega_{\eta,\e} \setminus \Omega_{\eta,\e}^\delta$ do not intersect with the boundary $\partial \Omega$ and $\dist(\Omega_{\eta,\e} \setminus \Omega_{\eta,\e}^\delta, \partial \Omega) \ge c\e$. Hence, by a similar argument, the proof of Lemma \ref{lem.Local Linfinity} applies to each cell $\e(k+Y_*)$ even if it is close to the boundary.
\end{proof}

\begin{lemma}\label{lem.deltaLayer-2}
     Let $u_{\eta,\e}$ be the same as in Lemma \ref{lem.deltaLayer} and
     $u_{\eta,\e}^*$  the harmonic extension of $u_{\eta,\e}$ from $\Omega_{\eta,\e}^\delta$ to $\Omega_{\eta,\e}$. Then for $\delta<c_0/8$,
     \begin{equation*}
         \int_{\Omega_{\eta,\e} \setminus \Omega_{\eta,\e}^\delta} \phi_{\eta,\e}^2 |\nabla u_{\eta,\e}^*|^2 \le C\delta \big( \| \phi_{\eta,\e} \nabla u_{\eta,\e} \|_{L^2(\Omega_{\eta,\e})}^2 +  \| f\|_{L^p(\Omega_{\eta,\e})}^2 \big).
     \end{equation*}
\end{lemma}
\begin{proof}
    First, we write
    \begin{equation}\label{est.Dtue.split}
        \int_{\Omega_{\eta,\e} \setminus \Omega_{\eta,\e}^\delta} \phi_{\eta,\e}^2 |\nabla u_{\eta,\e}^*|^2 = \sum_{\e(z+T_\eta^\delta \setminus T_\eta) \cap \Omega \neq \emptyset } \int_{\e(z+T^\delta_\eta \setminus T_\eta) \cap \Omega} \phi_{\eta,\e}^2 |\nabla u_{\eta,\e}^*|^2.
    \end{equation}
    We then consider a single cell $\e(z+T_\eta^{**} \setminus T_\eta) \cap \Omega$. By Lemma \ref{lem.Local Linfinity.eta} and Lemma \ref{lem.bdry Linfinity.eta},
    \begin{equation}\label{est.Due.Linfty}
    \begin{aligned}
        \| \phi_{\eta,\e} \nabla u_{\eta,\e} \|_{L^\infty(\e(z+T^{*}_\eta \setminus T_\eta) \cap \Omega)} & \le C \bigg( \fint_{\e(z+ T_\eta^{**} \setminus T_{\eta}) \cap \Omega_{\eta,\e}} |\phi_{\eta,\e} \nabla u_{\eta,\e}|^2 \bigg)^{1/2}  \\
        & \qquad +C\e\eta \bigg( \fint_{\e(z+ T_\eta^{**} \setminus T_{\eta}) \cap \Omega_{\eta,\e}} |f|^p \bigg)^{1/p} 
    \end{aligned}
    \end{equation}
    Recall that $\Delta u_{\eta,\e}^* = 0$ in $\e(z + T^\delta_\eta) \cap \Omega$ and $u_{\eta,\e}^* = u_{\eta,\e}$ on $\e(z + \partial T^\delta_\eta) \cap \Omega$. Also, $T^\delta_\eta$ is a union of mutually disjoint Lipschitz holes with connected boundaries. 
    In each hole, we can apply the $L^2$ regularity estimate of the nontangential estimate (see e.g.  Appendix of \cite{SZ23}) to obtain 
    \begin{equation*}
    \begin{aligned}
        & \| (\nabla u_{\eta,\e}^*)^* \|_{L^2(\e(z+\partial T^\delta_\eta) \cap \Omega)} \\
        &\le C \| \nabla_{\rm tan} u_{\eta,\e}^* \|_{L^2(\e(z+\partial T^\delta_\eta) \cap \Omega)}  \le C \| \nabla u_{\eta,\e} \|_{L^2(\e(z+\partial T^\delta_\eta ) \cap \Omega)} \\
        & \le C (\e \eta)^{\frac{d-1}{2}} \| \nabla u_{\eta,\e} \|_{L^\infty(\e(z+\partial T^\delta_\eta ) \cap \Omega)} \le C \delta^{-1} (\e \eta)^{\frac{d-1}{2}} \| \phi_{\eta,\e} \nabla u_{\eta,\e} \|_{L^\infty(\e(z+\partial T^\delta_\eta ) \cap \Omega)}.
    \end{aligned}
    \end{equation*}
    Consequently,
    \begin{equation}\label{est.Tdelta.Dtue}
    \begin{aligned}
        \int_{\e(z+ T^\delta_\eta \setminus T_\eta ) \cap \Omega} \phi_{\eta,\e}^2 |\nabla u_{\eta,\e}^*|^2 & \le 
        C\delta^3 \e \eta \int_{\e(z+\partial T^\delta_\eta) \cap \Omega} |(\nabla u_{\eta,\e}^*)^* |^2 \\
        & \le C\delta (\e \eta)^d \| \phi_{\eta,\e} \nabla u_{\eta,\e} \|_{L^\infty(\e(z+\partial T^\delta_\eta ) \cap \Omega)}^2.
    \end{aligned}
    \end{equation}
    Combining \eqref{est.Due.Linfty} and \eqref{est.Tdelta.Dtue}, we see that 
    \begin{equation*}
    \begin{aligned}
        \int_{\e(z+ T^\delta_\eta \setminus T_\eta) \cap \Omega} \phi_{\eta,\e}^2 |\nabla u_{\eta,\e}^*|^2 \le C\delta \int_{\e(z+ T_\eta^{**} \setminus T_{\eta}) \cap \Omega_{\eta,\e} } |\phi_{\eta,\e} \nabla u_{\eta,\e}|^2 + C\delta  (\e \eta)^{d+2-\frac{2d}{p}} \| f \|_{L^p(\Omega_{\eta,\e})}^2.
    \end{aligned}
    \end{equation*}
    Summing over $z$, we obtain
    \begin{equation*}
        \int_{\Omega_{\eta,\e} \setminus \Omega_{\eta,\e}^\delta} \phi_{\eta,\e}^2 |\nabla u_{\eta,\e}^*|^2 \le C\delta \big( \| \phi_{\eta,\e} \nabla u_{\eta,\e} \|_{L^2(\Omega_{\eta,\e})}^2 + \eta^d (\e \eta)^{2-\frac{2d}{p}} \| f\|_{L^p(\Omega_{\eta,\e})}^2 \big),
    \end{equation*}
    as desired.
\end{proof}

\begin{theorem}\label{thm.we.flp}
Let $\Omega$ be a bounded $C^2$ domain. Let $f\in L^p(\Omega_{\eta,\e})$ for some $p >d$. Let $u_{\eta,\e}$ and $\widetilde{u}_{\eta,\e}$ be the weak solutions of \eqref{eq.ue.phief} and \eqref{eq.ueta.phief}, respectively. Let
\begin{equation}
    w_{\eta,\e} = u_{\eta,\e} - \widetilde{u}_{\eta,\e} - \e \chi_{\eta,\e} \cdot \mathscr{K}_\e(\nabla \widetilde{u}_{\eta,\e}) \theta_\e.
\end{equation}
Then
    \begin{equation}
        \| u_{\eta,\e} - \widetilde{u}_{\eta,\e} \|_{L^2(\Omega_{\eta,\e})} + \| \phi_{\eta,\e} \nabla w_{\eta,\e} \|_{L^2(\Omega_{\eta,\e})} \le C \e^\frac18 \eta^\frac{d-2}{8} \| f \|_{L^p(\Omega_{\eta,\e})}.
    \end{equation}
\end{theorem}

\begin{proof}
    By Lemma \ref{lem.Dwet}, Lemma \ref{lem.deltaLayer} and Lemma \ref{lem.deltaLayer-2}, we have
        \begin{equation}
        \begin{aligned}
            \int_{\Omega_{\eta,\e}} \phi_{\eta,\e}^2 |\nabla w_{\eta,\e}|^2 & \le C\e^\frac12 \delta^{-1} \eta^\frac{d-2}{2} \| \nabla^2 u_{\eta} \|_{W^{1,d}(\Omega)} \big( \| \phi_{\eta,\e} \nabla u_{\eta,\e} \|_{L^2(\Omega_{\eta,\e})} + \| \nabla u_{\eta}\|_{L^2(\Omega)} \big) \\
        & \qquad + C\e \eta^{d-2} \| \nabla u_{\eta} \|_{W^{1,d}(\Omega)}^2  + C\delta \eta^{d-2} \| \nabla u_{\eta}\|_{L^2(\Omega)}^2 \\
        & \qquad + C\delta \big( \| \phi_{\eta,\e} \nabla u_{\eta,\e} \|_{L^2(\Omega_{\eta,\e})}^2 + \| f \|_{L^p(\Omega_{\eta,\e})}^2\big) \\
        & \le C(\e^\frac12 \eta^\frac{d-2}{2} \delta^{-1} + \e \eta^{d-2} + \delta \eta^{d-2} + \delta) \| f\|_{L^p(\Omega_{\eta,\e})}^2.
        \end{aligned}
    \end{equation}
    We choose $\delta = \e^\frac14 \eta^\frac{d-2}{4}$ to get
    \begin{equation}
        \| \phi_{\eta,\e} \nabla w_{\eta,\e} \|_{L^2(\Omega_{\eta,\e})} \le C \e^\frac18 \eta^\frac{d-2}{8} \| f \|_{L^p(\Omega_{\eta,\e})}.
    \end{equation}
    By \eqref{est.Poincare.eta} and \eqref{7.2-2}, we obtain
    \begin{equation}
    \begin{aligned}
        \| u_{\eta,\e} - \widetilde{u}_{\eta,\e} \|_{L^2(\Omega_{\eta,\e})} & \le \| w_{\eta,\e} \|_{L^2(\Omega_{\eta,\e})} + \| \e \chi_{\eta,\e} \cdot \mathscr{K}_\e(\nabla \widetilde{u}_{\eta,\e}) \theta_\e \|_{L^2(\Omega_{\eta,\e})} \\
        & \le C\e^\frac18 \eta^\frac{d-2}{8} \| f \|_{L^p(\Omega_{\eta,\e})}.
        \end{aligned}
    \end{equation}
    The proof is complete.
\end{proof}

As a consequence, we apply the above theorem to the equations \eqref{eq.degnerate} - \eqref{eq.intermediate} and obtain the convergence rate of the corresponding operators.
\begin{theorem}\label{thm.L2rate}
    Let $f \in L^p_{\phi_{\eta,\e}}(\Omega_{\eta,\e})$ for some $p>d$.
    Then
    \begin{equation*}
        \| \mathscr{T}_{\eta,\e} f - \widetilde{\mathscr{T}}_{\eta,\e} f\|_{L^2_{\phi_{\eta,\e}}(\Omega_{\eta,\e})} \le C \e^\frac18 \eta^\frac{d-2}{8} \|f \|_{L^p_{\phi_{\eta,\e}}(\Omega_{\eta,\e})}.
    \end{equation*}
\end{theorem}

The convergence rate from \eqref{eq.intermediate} to \eqref{eq.homogenized} is simpler and given below. It essentially relies on the regularity of $f$.
\begin{theorem}\label{thm.tT0-T0}
    Let $f\in W^{1,d}(\Omega)$. Then
    \begin{equation}\label{est.T0-tT0.rate}
        \| \widetilde{\mathscr{T}}_{\eta,\e} f - \mathscr{T}_{\eta} f\|_{H^1(\Omega)} \le C\e \eta^\frac{d-2}{2} \| f\|_{W^{1,d}(\Omega)}.
    \end{equation}
\end{theorem}

\begin{proof}
    Let $g\in H^{-1}(\Omega)$ and let $w\in H^1_0(\Omega)$ solve $\cL_\eta(w) = g$ in $\Omega$. Recall that $\tilde{u}_{\eta} = \widetilde{\mathscr{T}}_\eta f$ and $u_{\eta} = \mathscr{T}_{\eta,\e} f$ are the solutions of \eqref{eq.intermediate} and $\eqref{eq.homogenized}$, respectively. Then,
    \begin{equation*}
        \begin{aligned}
            \Ag{(\mathscr{T}_{\eta} - \widetilde{\mathscr{T}}_{\eta,\e})f, g}_{H_0^1\times H^{-1} } = \int_{\Omega} (1- \phi_{\eta,\e}^2 )f\cdot w.
        \end{aligned}
    \end{equation*}
    Since $\phi_{\eta,\e}^2 - 1$ is periodic and has mean value zero, we can find a bounded periodic function $\Psi(x/\e)$ such that
    \begin{equation}\label{eq.Psi}
        \phi_{\eta,\e}^2 - 1= \nabla\cdot (\e \Psi(x/\e)),
    \end{equation}
    and
    \begin{equation}\label{est.Psi}
        \| \Psi(x/\e) \|_{L^{2^*}(\Omega)} \le C\eta^\frac{d-2}{2}.
    \end{equation}
    Using the integration by parts, we obtain
    \begin{equation*}
    \begin{aligned}
        | \Ag{(\mathscr{T}_{\eta} - \widetilde{\mathscr{T}}_{\eta,\e})f, g}_{H_0^1\times H^{-1} }| & \le C\e  \| \Psi(x/\e) \|_{L^{2^*}(\Omega)} \| f\|_{W^{1,d}(\Omega)} \| w\|_{H^1(\Omega)} \\
        & \le C\e \eta^\frac{d-2}{2} \| f\|_{W^{1,d}(\Omega)} \| g\|_{H^{-1}(\Omega)}.
    \end{aligned}
    \end{equation*}
    This implies \eqref{est.T0-tT0.rate} by duality.
\end{proof}
    
\section{Convergence of eigenvalues}

In this section, we will quantify the convergence rates from $\mu^k_{\eta,\e}$ to $\mu^k_\eta$. These convergence rates, though not optimal, allows us to show the Weyl's law and find the common large spectral gaps. Our main tool is the minimax principle both in \eqref{eq.minimax1} and \eqref{eq.minimax2}.%We will consider $C^{1,1}$ or convex domains.

\subsection{Optimal upper bound}
The optimal upper bound is obtained directly by \eqref{eq.minimax1} without using the convergence rates in the previous section.
\begin{proposition}\label{prop.mue<mu0+}
Let $\Omega$ be a bounded $C^2$ domain satisfying the geometric assumption \textbf{A}. Then, for  $k\ge 1$, 
    \begin{equation}
    \mu^k_{\eta,\e} \le \mu^k_\eta + C_k \e \eta^\frac{d-2}{2}.
    \end{equation}
\end{proposition}
\begin{proof}

Fix $k\ge 1$. We may assume that $\e \eta^{\frac{d-2}{2}}$ is sufficiently small. 
We  first construct a subspace $S^k_{\rm app}$ of $H^1_{\phi_{\eta,\e},0}(\Omega_{\eta,\e})$ with dimension $k$. Actually, for $1\le j\le k$, let
\begin{equation*}
    v^j_{\eta,\e} = \rho^j_{\eta} + \e \chi^\ell_{\eta,\e} (\partial_\ell \rho_{\eta}^j) \theta_\e,
\end{equation*}
where we simply restrict the function $\rho^j_{\eta}$ in $\Omega_{\eta,\e}$. The cutoff function $\theta_\e$ (the same as in Section \ref{sec.4}) is used here to make sure that $v^j_{\eta,\e} = 0$ on $\Gamma_\e$. The functions $v^j_{\eta,\e}$ are supposed to be good approximations of the eigenfunctions $\rho_{\eta,\e}^j$ of the degenerate eigenvalue problem in $V_{\eta,\e}$. Recall that the regularity of the homogenized eigenvalue problem in $C^2$ domains implies
\begin{equation}
    \| \rho_\eta^j \|_{W^{2,p}(\Omega)} \le C_{j,p},
\end{equation}
for any $p<\infty$.

Let $S^k_{\rm app} = \text{span}\{ v^j_{\eta,\e}: 1\le j\le k \}$.

Claim: $\text{dim}\ S^k_{\rm app} = k$ if $\e \eta^\frac{d-2}{2} \le c_k$ for sufficiently small $c_k$ depending on $k$. In fact, we can use the orthogonality of $\rho_\eta^k$ in the space $L^2(\Omega)$ to show that for $\e\eta^\frac{d-2}{2}$ small enough (depending on $k$ or $\mu^k_\eta$)
\begin{equation*}
    \sum_{j=1}^k \alpha_j v^j_{\eta,\e} = 0 \text{ if and only if } \alpha_j = 0 \text{ for all } 1\le j\le k.
\end{equation*}
To show this, consider a general $v = \sum_{j=1}^k \alpha_j v^j_{\eta,\e} \in S^k_{\rm app}$. We compute
\begin{equation}\label{eq.phiv2}
\begin{aligned}
    \int_{\Omega_{\eta,\e}} \phi_{\eta,\e}^2 v^2 & = \sum_{1\le i,j\le k} \alpha_i\alpha_j \int_{\Omega_{\eta,\e}} \phi_{\eta,\e}^2 \rho^i_{\eta} \rho^j_{\eta} + 2\sum_{1\le i,j\le k} \alpha_i\alpha_j \int_{\Omega_{\eta,\e}} \phi_{\eta,\e}^2 \rho^i_{\eta} \e \chi^\ell_{\eta,\e} (\partial_\ell \rho_{\eta}^j) \theta_\e.
\end{aligned}
\end{equation}
For each pair of $i,j$,
\begin{equation}
    \bigg| \int_{\Omega_{\eta,\e}} \phi_{\eta,\e}^2 \rho^i_{\eta} \e \chi^\ell_{\eta,\e} (\partial_\ell \rho_{\eta}^j) \theta_\e  \bigg| \le C\e \| \phi_{\eta,\e} \chi_{\eta,\e} \|_{L^{2^*}(\Omega_{\eta,\e})} \| \rho_\eta^i \|_{L^d(\Omega)} \| \nabla \rho_\eta^j \|_{L^2(\Omega)} \le C_k \e \eta^{\frac{d-2}{2}}.
\end{equation}
and
\begin{equation}
\begin{aligned}
    \bigg| \int_{\Omega_{\eta,\e}} \phi_{\eta,\e}^2 \e^2 \chi^\ell_{\eta,\e} (\partial_\ell \rho_\eta^i) \chi^\tau_{\eta,\e} (\partial_\tau \rho_{\eta}^j) \theta_\e^2 \bigg| & \le C\e^2 \| \phi_{\eta,\e} \chi_{\eta,\e} \|_{L^{2^*}(\Omega_{\eta,\e})}^2 \| \nabla \rho_\eta^i \|_{L^d(\Omega)} \| \nabla \rho_\eta^j \|_{L^d(\Omega)} \\
    & \le C_k \e^2 \eta^{d-2}.
\end{aligned}
\end{equation}
Hence, \eqref{eq.phiv2} gives
\begin{equation}
    \int_{\Omega_{\eta,\e}} \phi_{\eta,\e}^2 v^2 =  \sum_{1\le i,j\le k} \alpha_i\alpha_j \int_{\Omega_{\eta,\e}} \phi_{\eta,\e}^2 \rho^i_{\eta} \rho^j_{\eta} + O_k(\e \eta^\frac{d-2}{2}) \sum_{1\le i, j\le k}|\alpha_i \alpha_j|.
\end{equation}

% We point out that the $O(\e)$ depends on $\mu_{\eta}^k$ since
% \begin{equation*}
%     \int_{\Omega} |\nabla \rho_{\eta}^j|^2 \approx \int_{\Omega} \overline{A} \nabla \rho_{\eta}^j \cdot \nabla \rho_{\eta}^j = \mu_{\eta}^j \int_{\Omega} |\rho_{\eta}^j|^2 = \mu_{\eta}^j \le \mu_{\eta}^k.
% \end{equation*}
Next, using the same argument as in the proof of Theorem \ref{thm.tT0-T0}, by \eqref{eq.Psi} and \eqref{est.Psi}, we have
\begin{equation}\label{est.rhoeta.alor}
\begin{aligned}
    \bigg| \int_{\Omega_{\eta,\e}} \phi_{\eta,\e}^2 \rho^i_{\eta} \rho^j_{\eta} - \int_{\Omega} \rho^i_{\eta} \rho^j_{\eta} \bigg| & = \bigg| \int_{\Omega_{\eta,\e}} \nabla\cdot (\e \Psi(x/\e)) \rho^i_{\eta} \rho^j_{\eta} \bigg| \\
    & \le \e \| \Psi_\e \|_{L^{2^*}(\Omega)} \| \rho_\eta^i \|_{W^{1,d}(\Omega)} \| \rho_\eta^j \|_{H^1(\Omega)} \\
    & \le C_k \e \eta^\frac{d-2}{2}.
\end{aligned}
\end{equation}
Hence, by the orthogonality $\int_{\Omega} \rho^i_{\eta} \rho^j_{\eta} = \delta_{ij}$,
\begin{equation}\label{est.dim=k}
    \int_{\Omega_{\eta,\e}} \phi_{\eta,\e}^2 v^2 = \sum_{1\le j\le k} |\alpha_j|^2 + O_k(\e \eta^\frac{d-2}{2}) \sum_{1\le j\le k} |\alpha_j|^2.
\end{equation}
Note that the constant in $O_k(\e \eta^\frac{d-2}{2})$ depends on $k$. Then if $\e \eta^\frac{d-2}{2} < c_k$ for some constant $c_k>0$, then $v = 0$ if and only if $\alpha_j = 0$ for all $1\le j\le k$. This finishes the proof of the claim.

Now, by the minimax principle \eqref{eq.minimax1},
\begin{equation}\label{est.minimax.upper}
    \mu^k_{\eta,\e} \le \max_{v \in S^k_{\rm app}} \frac{ \int_{\Omega_{\eta,\e}} \phi_{\eta,\e}^2 |\nabla v|^2}{\int_{\Omega_{\eta,\e}} \phi_{\eta,\e}^2 v^2}.
\end{equation}
Consider a general $v = \sum_{j=1}^k \alpha_j v^j_{\eta,\e} \in S^k_{\rm app}$. 
Without loss of generality, assume $\sum_{j=1}^k \alpha_j^2 = 1$. If $\e \eta^\frac{d-2}{2} < c_k$, \eqref{est.dim=k} implies
\begin{equation}\label{eq.vL2}
    \int_{\Omega_{\eta,\e}} \phi_{\eta,\e}^2 v^2 = 1 + O_k(\e \eta^\frac{d-2}{2}).
\end{equation}
% \begin{equation}
% \begin{aligned}
%     \int_{\Omega_{\eta,\e}} \phi_{\eta,\e}^2 v^2 & = \sum_{1\le i,k\le N} \alpha_i\alpha_k \int_{\Omega_{\eta,\e}} \phi_{\eta,\e}^2 \rho^i_{\eta} \rho^k_0 + 2\sum_{1\le i,k\le N} \alpha_i\alpha_k \int_{\Omega_{\eta,\e}} \phi_{\eta,\e}^2 \rho^i_{\eta} \e \chi_j(x/\e) \partial_j \rho_\eta^k \theta_\e \\
%     &\qquad + \sum_{1\le i,k\le N} \alpha_i\alpha_k \int_{\Omega_{\eta,\e}} \phi_{\eta,\e}^2 \e^2 \chi^\ell_\eta(x/\e) \partial_\ell \rho_\eta^i \chi_j(x/\e) \partial_j \rho_\eta^k \theta_\e^2 \\
%     & = 1 + O(\e).
% \end{aligned}
% \end{equation}
%Here and after the notation $a= b+O(\e)$ means $|a-b| \le C\e$.

Next, we estimate the upper bound of the numerator of \eqref{est.minimax.upper}. In view of \eqref{eq.bar.aij}, we perform the following calculation
\begin{equation}\label{est.Dv2.upper}
\begin{aligned}
    & \int_{\Omega_{\eta,\e}} \phi_{\eta,\e}^2 |\nabla v|^2 \\
    & = \sum_{1\le i,j\le k} \alpha_i \alpha_j \int_{\Omega_{\eta,\e}} \phi^2_{\eta,\e} (\partial_\ell \rho^i_{\eta} + (\partial_\ell \chi_\eta )_\e \cdot \nabla \rho^i_{\eta} \theta_\e) \cdot (\partial_\ell \rho^j_{\eta} + (\partial_\ell \chi_\eta )_\e \cdot \nabla \rho^j_{\eta} \theta_\e) + O_k(\e \eta^\frac{d-2}{2}) \\
     &= \sum_{1\le i,j\le k} \alpha_i \alpha_j \int_{\Omega_{\eta,\e}} \phi^2_{\eta,\e} (\delta_{m\ell} + (\partial_\ell \chi^m_\eta)_\e ) (\delta_{n\ell} + (\partial_\ell \chi^n_{\eta})_\e ) \partial_m \rho^i_{\eta} \partial_n \rho^j_{\eta} \theta_\e^2 \\
     & \qquad + \sum_{1\le i,j\le k} \alpha_i \alpha_j \int_{\Omega_{\eta,\e}} \phi^2_{\eta,\e} \partial_\ell \rho_\eta^i \partial_\ell \rho_\eta^j (1-\theta_\e^2)  + O_k (\e \eta^\frac{d-2}{2}).
\end{aligned}  
\end{equation}
The error terms contained in $O_k(\e \eta^\frac{d-2}{2})$ above are estimated by a familiar argument, including the boundary layer estimate $\| f \|_{L^p(\Omega(2\e))} \le C\e^\frac{1}{p} \| f \|_{W^{1,p}(\Omega)}$ (see \cite[Lemma A.5]{SZ23}) whenever $\nabla \theta_\e$ or $1-\theta_\e$ involves.

To proceed, we first notice that
\begin{equation}
    \| \phi_{\eta,\e}^2 I - \overline{A}_\eta \|_{L^{2^*}(\Omega(2\e))} \le \| \phi_{\eta,\e}^2 I - I \|_{L^{2^*}(\Omega(2\e))} + \| I - \overline{A}_\eta \|_{L^{2^*}(\Omega(2\e))} \le C\e^{\frac{1}{2^*}} \eta^\frac{d-2}{2}.
\end{equation}
Hence, for each pair of $i,j$, we estimate the last integral of \eqref{est.Dv2.upper} as
\begin{equation}\label{est.upper.layer}
\begin{aligned}
    & \bigg| \int_{\Omega_{\eta,\e}} \phi^2_{\eta,\e} \partial_\ell \rho_\eta^i \partial_\ell \rho_\eta^j (1-\theta_\e^2) - \int_{\Omega_{\eta,\e}} \bar{a}_{\eta,mn} \partial_m \rho_\eta^i \partial_n \rho_\eta^j (1-\theta_\e^2) \bigg| \\
    & \le \| \phi_{\eta,\e}^2 I - \overline{A}_\eta \|_{L^{2^*}(\Omega(2\e))} \| \nabla \rho_\eta^i \|_{L^d(\Omega(2\e))} \| \nabla \rho_\eta^j \|_{L^2(\Omega(2\e))} \\
    & \le  C\e^{\frac{1}{2^*}} \eta^\frac{d-2}{2} C_k \e^\frac{1}{d} C_k \e^\frac{1}{2} \le C_k \e \eta^\frac{d-2}{2}.
\end{aligned}
\end{equation}

Now consider the following periodic equation,
\begin{equation}\label{eq.Xi}
    -\Delta \Xi_{mn}  = \bar{a}_{\eta,mn} - \phi^2_{\eta} (\delta_{m\ell} + (\partial_\ell \chi^m_\eta) ) (\delta_{n\ell} + (\partial_\ell \chi^n_{\eta}) ), \quad \text{in } Y.
\end{equation}
The equation is solvable since the right-hand side has mean value zero due to \eqref{eq.bar.aij}. We would like to show
\begin{equation}\label{est.Xi}
    \| \nabla \Xi_{mn} \|_{L^2(Y)} \le C\eta^\frac{d-2}{2}.
\end{equation}
To see this, we write
\begin{equation}\label{eq.Xi-2}
\begin{aligned}
    & \bar{a}_{\eta,mn} - \phi^2_{\eta} (\delta_{m\ell} + (\partial_\ell \chi^m_\eta) ) (\delta_{n\ell} + (\partial_\ell \chi^n_{\eta}) ) \\
    & = \bar{a}_{\eta,mn} - \delta_{mn} + \delta_{mn} (1- \phi_\eta^2 )  - \phi_\eta^2 (\partial_m \chi_\eta^n + \partial_n \chi_\eta^m) -\phi_\eta^2 \partial_\ell \chi_\eta^m \partial_\ell \chi_\eta^n.
\end{aligned}
\end{equation}
Moreover, using the equation \eqref{eq.corrector}, we have
\begin{equation}\label{eq.Xi-3}
\begin{aligned}
    \phi_\eta^2 \partial_\ell \chi_\eta^m \partial_\ell \chi_\eta^n & = \partial_\ell (\phi_\eta^2 \chi_\eta^m \partial_\ell \chi_\eta^n) + \chi_\eta^m \partial_\ell (\phi_\eta^2 \partial_\ell \chi_\eta^n) \\
    & = \partial_\ell (\phi_\eta^2 \chi_\eta^m \partial_\ell \chi_\eta^n) + 2\chi_\eta^m \phi_\eta \partial_n \phi_\eta.
\end{aligned}
\end{equation}
Combining \eqref{eq.Xi}, \eqref{eq.Xi-2} and \eqref{eq.Xi-3}, we have
\begin{equation}
\begin{aligned}
    \| \Delta \Xi \|_{H^{-1}_{\rm per}(Y)} & \le \| A_\eta - I \|_{L^2(Y)} + 2\|1 - \phi_\eta^2 \|_{L^2(Y)} + 2\| \phi_\eta^2 \nabla \chi_\eta \|_{L^2(Y)} \\
    & \qquad + \| \phi_\eta^2 \chi_\eta \nabla \chi_\eta \|_{L^2(Y)} + 2\| \chi_\eta \phi_\eta \nabla \phi_\eta \|_{L^2(Y)} \le C\eta^\frac{d-2}{2}.
\end{aligned}
\end{equation}
This implies \eqref{est.Xi}.

Hence,
\begin{equation}\label{est.Xi-4}
\begin{aligned}
    & \int_{\Omega_{\eta,\e}} \phi^2_{\eta,\e} (\delta_{m\ell} + (\partial_\ell \chi^m_\eta)_\e ) (\delta_{n\ell} + (\partial_\ell \chi^n_{\eta})_\e ) \partial_m \rho^i_{\eta} \partial_n \rho^j_{\eta} \theta_\e^2 - \int_{\Omega } \bar{a}_{\eta,mn} \partial_m \rho^i_{\eta} \partial_n \rho^j_{\eta} \theta_\e^2 \\
    & = \int_{\Omega} \nabla\cdot (\e (\nabla \Xi_{mn})_\e ) \partial_m \rho^i_{\eta} \partial_n \rho^j_{\eta} \theta_\e^2 \\
    & = - \e \int_{\Omega} (\nabla \Xi_{mn})_\e  \nabla( \partial_m \rho^i_{\eta} \partial_n \rho^j_{\eta} ) \theta_\e^2 - \e \int_{\Omega}  (\nabla \Xi_{mn})_\e \partial_m \rho^i_{\eta} \partial_n \rho^j_{\eta} 2\theta_\e \nabla \theta_\e.
\end{aligned}
\end{equation}
The first integral is bounded by
\begin{equation}\label{est.Xi-5}
    C \e \| (\nabla \Xi)_\e \|_{L^2(\Omega)} \| \nabla \rho_\eta^j \|_{W^{1,4}(\Omega)}  \| \nabla \rho_\eta^i\|_{W^{1,4}(\Omega)} \le C_k \e \eta^\frac{d-2}{2}. 
\end{equation}
The second integral is bounded by
\begin{equation}\label{est.Xi-6}
    \| (\nabla \Xi)_\e \|_{L^2(\Omega(2\e))} \| \nabla \rho_\eta^j \|_{L^4(\Omega(2\e))}  \| \nabla \rho_\eta^i\|_{L^4(\Omega(2\e))} \le C\e^\frac12 \eta^\frac{d-2}{2} C_k \e^\frac14 C_k \e^\frac14 \le C_k \e \eta^\frac{d-2}{2}.
\end{equation}

It follows from \eqref{est.upper.layer} and \eqref{est.Xi-4}-\eqref{est.Xi-6} that
\begin{equation}\label{est.DvL2}
    \begin{aligned}
        \int_{\Omega_{\eta,\e}} \phi_{\eta,\e}^2 |\nabla v|^2  & = \sum_{1\le i,j\le k} \alpha_i \alpha_j \int_{\Omega } \bar{a}_{\eta,mn} \partial_m \rho^i_{\eta} \partial_n \rho^j_{\eta,\e} + O_k(\e \eta^\frac{d-2}{2}) \\
     & = \sum_{i,j = 1}^k \alpha_i \alpha_j \int_{\Omega} \nabla \rho^i_{\eta} \cdot \overline{A}_\eta \nabla \rho^j_{\eta} + O_k(\e \eta^\frac{d-2}{2}) \\
     & = \sum_{i = 1}^k \alpha_i^2 \mu^i_\eta + O_k (\e \eta^\frac{d-2}{2}) \\
     & \le \mu^k_\eta +C_k \e \eta^\frac{d-2}{2},
    \end{aligned}
\end{equation}
where we have used the orthogonality
\begin{equation}
    \int_{\Omega} \nabla \rho^i_{\eta} \cdot \overline{A}_\eta \nabla \rho^j_{\eta} = \mu_\eta^i \int_{\Omega} \rho_\eta^i \rho_\eta^j = \mu_\eta^i \delta_{ij}.
\end{equation}

Finally, combining \eqref{est.minimax.upper}, \eqref{eq.vL2} and \eqref{est.DvL2}, we obtain the desired upper bound for $\mu_{\eta,\e}^k$.
\end{proof}

\subsection{Suboptimal lower bound}\label{sec.5.2}
We will use both the ``maximin principle'' and ``minimax principle'' to show a lower bound. The quantitative convergence rates in Theorem \ref{thm.L2rate} and Theorem \ref{thm.tT0-T0} will be crucial.

% Recall the linear operator $\mathscr{T}_{\eta,\e}$ which is bounded  and actually compact on $L^2_{\phi_{\eta,\e}}(\Omega_{\eta,\e})$. Recall the minimax princple \eqref{eq.minimax2}. Let $\widetilde{\mathscr{T}}_\eta: L_{\phi_{\eta,\e}}^2(\Omega) \to H^1_0(\Omega)$ be the map given by $\tilde{u}_{\eta} = \widetilde{\mathscr{T}}_\eta(f)$, where $\tilde{u}_{\eta}$ solves
% \begin{equation}\label{eq.tu0}
%     \cL_\eta(\tilde{u}_{\eta}) = \phi_{\eta,\e}^2 f \quad \text{in } \Omega, \quad \text{and}\quad  \tilde{u}_{\eta} \in H^1_0(\Omega).
% \end{equation}
Recall the operators $\mathscr{T}_{\eta,\e}, \widetilde{\mathscr{T}}_\eta$ and $\mathscr{T}_{\eta}$ defined in Section \ref{sec.2.5}.
Recall that $\mathscr{T}_\eta$ is a compact self-adjoint operator in $L^2_{\phi_{\eta,\e}}(\Omega_{\eta,\e})$ and the maximin principle for the eigenvalue problem \eqref{eq.eigen-1} is given by  (see \eqref{eq.minimax2})
\begin{equation*}
    \frac{1}{\mu^k_{\eta,\e}}  = \max_{\substack{S\subset L^2_{\phi_{\eta,\e}}(\Omega_{\eta,\e}) \\ \text{dim} S = k }} \min_{v\in S} \frac{\int_{\Omega_{\eta,\e}}\phi_{\eta,\e}^2  \mathscr{T}_{\eta,\e}(v) \cdot v }{\int_{\Omega_{\eta,\e}} \phi_{\eta,\e}^2v^2 }.
\end{equation*}
In fact, this maximum is attained if we pick $S = S^{k}_{\eta,\e} = \text{span} \{ \rho^j_{\eta,\e}: 1\le j \le k \}$. Therefore,
\begin{equation}\label{est.new.muk}
    \begin{aligned}
        \frac{1}{\mu_{\eta,\e}^k} & = \min_{v \in S^{ k}_\e } \frac{\int_{\Omega_{\eta,\e}}\phi_{\eta,\e}^2  \mathscr{T}_{\eta,\e}(v) \cdot v }{\int_{\Omega_{\eta,\e}} \phi_{\eta,\e}^2v^2 } \\
        & \le \max_{v \in S^{ k}_\e } \frac{\int_{\Omega_{\eta,\e}}\phi_{\eta,\e}^2  (\mathscr{T}_{\eta,\e}(v) - \widetilde{\mathscr{T}}_{\eta,\e}(v) ) \cdot v }{\int_{\Omega_{\eta,\e}} \phi_{\eta,\e}^2v^2 } + \min_{v \in S^{ k}_\e } \frac{\int_{\Omega_{\eta,\e}}\phi_{\eta,\e}^2  \widetilde{\mathscr{T}}_{\eta,\e}(v) \cdot v }{\int_{\Omega_{\eta,\e}} \phi_{\eta,\e}^2v^2 } \\
        & \le \max_{v \in S^{ k}_\e } \frac{\int_{\Omega_{\eta,\e}}\phi_{\eta,\e}^2  (\mathscr{T}_{\eta,\e}(v) - \widetilde{\mathscr{T}}_{\eta,\e}(v) ) \cdot v }{\int_{\Omega_{\eta,\e}} \phi_{\eta,\e}^2v^2 } + \max_{\substack{S\subset L^2_{\phi_{\eta,\e}}(\Omega_{\eta,\e}) \\ \text{dim} S = k }} \min_{v \in S } \frac{\int_{\Omega_{\eta,\e}}\phi_{\eta,\e}^2  \widetilde{\mathscr{T}}_{\eta,\e}(v) \cdot v }{\int_{\Omega_{\eta,\e}} \phi_{\eta,\e}^2v^2 }.
    \end{aligned}
\end{equation}
%\jz{The last term is equal to $1/\tilde{\mu}_\e^k$ and it suffices to estimate the second term. It has the same form as in the previous subsection, but taking only test functions in $S_\e^{k}$ which are linear combinations of the first $k$ eigenfunctions. These functions have better regularity than general functions in $L^2_{\phi_{\eta,\e}}(\Omega_{\eta,\e})$ and thus lead to better regularity of $\mathscr{T}_{\eta,\e}(v)$.}

% Now we use a duality argument to estimate the first error term on the right-hand side of \eqref{est.new.muk}. Let $u_{\eta,\e} = \mathscr{T}_{\eta,\e}(v)$ with $v\in S_\e^{k}$ be the weak solution of $-\nabla\cdot (\phi_{\eta,\e}^2 \nabla u_{\eta,\e}) = \phi_{\eta,\e}^2 v$. Let $u_{\eta} = \widetilde{\mathscr{T}}_\eta(v)$. Consequently,
% \begin{equation}\label{eq.dual.Te-T0}
% \begin{aligned}
%     &\int_{\Omega_{\eta,\e}}\phi_{\eta,\e}^2  (\mathscr{T}_{\eta,\e}(v)-\widetilde{\mathscr{T}}_\eta(v) ) \cdot v \\
%     & =  \int_{\Omega_{\eta,\e}}\phi_{\eta,\e}^2  (\mathscr{T}_{\eta,\e}(v)-\widetilde{\mathscr{T}}_\eta(v)- \e \chi(x/\e)\cdot \nabla \widetilde{\mathscr{T}}_\eta(v) \theta_\e ) \cdot v + O(\e) \\
%     & = \int_{\Omega_{\eta,\e}} \phi_{\eta,\e}^2 \nabla( u_{\eta,\e} - u_{\eta} - \e\chi(x/\e)\cdot \nabla u_{\eta} \theta_\e) \cdot \nabla u_{\eta,\e} + O(\e) \\
%     & = \int_{\Omega_{\eta,\e} } \phi_{\eta,\e}^2 \nabla w_{\eta,\e} \cdot \nabla w_{\eta,\e} + \int_{\Omega_{\eta,\e}} \phi_{\eta,\e}^2 \nabla w_{\eta,\e} \cdot \nabla (u_{\eta} - \e\chi(x/\e)\cdot \nabla u_{\eta} \theta_\e ) + O(\e)
% \end{aligned}
% \end{equation}
The numerator of the first term can be handled by Theorem \ref{thm.L2rate}. Precisely, if $v = \sum_{j=1}^k \alpha_j \rho_{\eta,\e}^j \in S^k_{\eta,\e}$, then by the triangle inequality in $L^\infty_{\phi_{\eta,\e}}(\Omega_{\eta,\e})$, Proposition \ref{prop.Linfty.rho} and the orthogonality of $\rho_{\eta,\e}^j$'s in $L^2_{\phi_{\eta,\e}}(\Omega_{\eta,\e})$, we have
\begin{equation}\label{est.eigen.L2Lp}
\begin{aligned}
    \| \phi_{\eta,\e} v \|_{L^\infty(\Omega_{\eta,\e})} & \le \sum_{1\le j\le k} |\alpha_j| \| \phi_{\eta,\e} \rho_{\eta,\e}^j  \|_{L^\infty(\Omega_{\eta,\e})} \\
    & \le \sum_{1\le j\le k} |\alpha_j| C_k \\
    & \le \sqrt{k} C_k \Big( \sum_{1\le j\le k} |\alpha_j|^2 \Big)^{1/2} \\
    & \le \sqrt{k} C_k \| \phi_{\eta,\e} v \|_{L^2(\Omega_{\eta,\e})}.
\end{aligned}
\end{equation}
Consequenctly, Theorem \ref{thm.L2rate} with $p = \infty$ implies
\begin{equation}\label{est.upper-1}
\begin{aligned}
    & \bigg| \max_{v \in S^{ k}_\e } \frac{\int_{\Omega_{\eta,\e}}\phi_{\eta,\e}^2  (\mathscr{T}_{\eta,\e}(v) - \widetilde{\mathscr{T}}_\eta(v) ) \cdot v }{\int_{\Omega_{\eta,\e}} \phi_{\eta,\e}^2v^2 } \bigg| \\
    & \le \max_{v \in S^{ k}_\e } \frac{ \| \phi_{\eta,\e} (\mathscr{T}_{\eta,\e}(v) - \widetilde{\mathscr{T}}_\eta(v) ) \|_{L^2(\Omega_{\eta,\e})} \| \phi_{\eta,\e} v \|_{L^2(\Omega_{\eta,\e})} }{\| \phi_{\eta,\e} v \|_{L^2(\Omega_{\eta,\e})}^2}  \le C_k \e^\frac18 \eta^\frac{d-2}{8} .
\end{aligned}
\end{equation}

% \begin{equation}
%     \| \phi_{\eta,\e} v \|_{L^p(\Omega_{\eta,\e})} \le C(\mu^k_{\eta,\e})^{\gamma_d + \frac12} \| \phi_{\eta,\e} v \|_{L^2(\Omega_{\eta,\e})},
% \end{equation}
% for any $1<p<\infty$. Then by Proposition \ref{thm.Lpdual},
% \begin{equation}\label{est.dualL2.vv}
%     \bigg| \int_{\Omega_{\eta,\e}}\phi_{\eta,\e}^2  (\mathscr{T}_{\eta,\e}(v) - \widetilde{\mathscr{T}}_\eta(v) ) \cdot v \bigg| \le C\e^{\frac12} (\mu^k_{\eta,\e})^{\gamma_d + \frac12} \| \phi_{\eta,\e} v \|_{L^2(\Omega_{\eta,\e})}^2.
% \end{equation}
% As a result, the first term on the right-hand side of \eqref{est.new.muk} is bounded by $C\e^{\frac12} (\mu^k_{\eta,\e})^{\gamma_d + \frac12}$.

Next, we handle the second term on the right-hand side of \eqref{est.new.muk}. The key observation is that the second term of \eqref{est.new.muk} is exactly the maximin principle for the intermediate eigenvalue problem \eqref{eq.eigen-2} since $\widetilde{\mathscr{T}}_{\eta,\e}$ is a compact self-adjoint operator in $L^2_{\phi_{\eta,\e}}(\Omega_{\eta,\e})$. Thus, it is equal to $1/\tilde{\mu}^k_{\eta,\e}$, where $\tilde{\mu}^k_{\eta,\e}$ is the $k$th eigenvalue of the intermediate eigenvalue problem \eqref{eq.eigen-2}.

On the other hand, we can employ the minimax principle of \eqref{eq.eigen-2} to get
\begin{equation*}
    \tilde{\mu}^k_{\eta,\e} = \min_{\substack{S\subset H^1_{0}(\Omega)\\  \text{dim} S = k} } \max_{v \in S} \frac{ \int_{\Omega} \overline{A}\nabla v\cdot \nabla v}{\int_{\Omega} \phi_{\eta,\e}^2 v^2}.
\end{equation*}

\begin{proposition}\label{prop.tmu-mu}
    For $\e \eta^\frac{d-2}{2}<c_k$ with sufficiently small $c_k$ depending on $k$, $|\tilde{\mu}^k_{\eta,\e} - \mu^k_\eta| \le C_k \e \eta^\frac{d-2}{2}$.
\end{proposition}
\begin{proof}
    We only show $\mu^k_\eta \le \tilde{\mu}^k_{\eta,\e} + C_k \e \eta^\frac{d-2}{2}$ (which is sufficient to get the lower bound). The other direction is similar. Recall that $\tilde{\rho}^j_{\eta,\e}$ is the $j$th eigenfunction of \eqref{eq.eigen-2} corresponding to $\tilde{\mu}^j_{\eta,\e}$. By the standard normalization and orthogonality in \eqref{eq.normal}, we have $\| \phi_{\eta,\e}  \tilde{\rho}^j_{\eta,\e}\|_{L^2(\Omega)} = 1$ and
\begin{equation}\label{eq.orthogonal.trho}
    \int_{\Omega} \phi_{\eta,\e}^2 \tilde{\rho}^i_{\eta,\e} \tilde{\rho}^j_{\eta,\e} = \delta_{ij}, \qquad \int_{\Omega} \overline{A} \nabla \tilde{\rho}^i_{\eta,\e} \cdot \nabla \tilde{\rho}^j_{\eta,\e} = \tilde{\mu}^i_{\eta,\e} \delta_{ij}.
\end{equation}
Let $\widetilde{S}^{k}_{\eta,\e} = \text{span}\{ \tilde{\rho}^j_{\eta,\e}: 1\le j \le  k  \}.$
Then \eqref{eq.orthogonal.trho} implies that $\dim \widetilde{S}^{k}_{\eta,\e} = k$ by viewing $\widetilde{S}^{k}_{\eta,\e}$ as a subspace of $H_0^1(\Omega)$.
Hence,
\begin{equation}\label{est.muk0}
    \begin{aligned}
        \mu^k_\eta = \min_{\substack{S\subset H^1_{0}(\Omega)\\  \text{dim} S = k} } \max_{v \in S} \frac{ \int_{\Omega} \overline{A}\nabla v\cdot \nabla v }{\int_{\Omega} v^2} \le \max_{v \in \widetilde{S}^{k}_{\eta,\e}} \frac{ \int_{\Omega} \overline{A}\nabla v\cdot \nabla v}{\int_{\Omega} v^2}.
    \end{aligned}
\end{equation}
Now consider a general $v \in \widetilde{S}^{k}_{\eta,\e}$ given by $v = \sum_{j=1}^k \alpha_j \tilde{\rho}^j_{\eta,\e}$ with $\sum_{j=1}^k \alpha_j^2 = 1$. By \eqref{eq.orthogonal.trho}, we obtain
\begin{equation*}
    \int_{\Omega} \phi_{\eta,\e}^2 v^2 = 1,\qquad \int_{\Omega} \overline{A}\nabla v\cdot \nabla v = \sum_{j=1}^k \alpha_j^2 \tilde{\mu}^j_{\eta,\e} \le \tilde{\mu}^k_{\eta,\e}.
\end{equation*}
Note that the elliptic regularity estimates (with a bootstrap argument) for the equation \eqref{eq.eigen-2} imply $\tilde{\rho}^j_{\eta,\e} \in W^{1,p}(\Omega)$ for any $p<\infty$. Thus, similar to \eqref{est.eigen.L2Lp}, one has
\begin{equation}\label{est.v.w1p}
    \| v \|_{W^{1,p}(\Omega)} \le \sum_{1\le j\le k} |\alpha_j| \| \tilde{\rho}^j_{\eta,\e} \|_{W^{1,p}(\Omega)}  \le \sqrt{k} C_k \Big( \sum_{1\le j\le k} |\alpha_j|^2 \Big)^{1/2} \le C_k \| \phi_{\eta,\e} v \|_{L^2(\Omega)}.
\end{equation}

We now claim
\begin{equation}\label{est.diff.v2}
    \bigg| \int_{\Omega} v^2 - \int_{\Omega} \phi_{\eta,\e}^2 v^2 \bigg| \le C_k \e \eta^\frac{d-2}{2} \int_{\Omega} \phi_{\eta,\e}^2 v^2.
\end{equation}
In fact, by \eqref{eq.Psi}, \eqref{est.Psi} and \eqref{est.v.w1p}, an integration by parts give
\begin{equation*}
\begin{aligned}
    \bigg| \int_{\Omega} v^2 - \int_{\Omega} \phi_{\eta,\e}^2 v^2 \bigg| & = \e \bigg| \int_{\Omega} \Psi_\e \cdot 2v \nabla v \bigg| \\
    &\le C\e \| \Psi_\e \|_{L^{2^*}(\Omega)} \|  v\|_{L^d(\Omega)} \| \nabla v \|_{L^2(\Omega)} \\
    & \le C_k \e \eta^\frac{d-2}{2} \int_{\Omega} \phi_{\eta,\e}^2 v^2.
\end{aligned}
\end{equation*}
This proves \eqref{est.diff.v2}, which yields
\begin{equation*}
    \int_{\Omega} v^2 \ge (1-C_k \e \eta^\frac{d-2}{2}) \int_{\Omega} \phi_{\eta,\e}^2 v^2.
\end{equation*}
By letting $\e \eta^\frac{d-2}{2} \le (2C_k)^{-1}$, the above estimate and \eqref{est.muk0} gives
\begin{equation*}
    \mu^k_\eta \le \frac{1}{1-C_k \e \eta^\frac{d-2}{2}} \max_{v \in \widetilde{S}^{k}_{\eta,\e}} \frac{ \int_{\Omega} \overline{A}\nabla v\cdot \nabla v}{\int_{\Omega} \phi_{\eta,\e}^2 v^2} = \frac{\tilde{\mu}^k_{\eta,\e}}{1-C_k \e \eta^\frac{d-2}{2}} \le \tilde{\mu}^k_{\eta,\e} + C_k \e \eta^\frac{d-2}{2}.
\end{equation*}
This proves the desired result.
\end{proof}

\begin{proposition}\label{prop.mue>mu0-}
    Let $\Omega$ be a bounded $C^2$ domain satisfying the geometric assumption \textbf{A}. Then for $k\ge 1$,
    \begin{equation*}
        \mu_{\eta,\e}^k \ge \mu_{\eta}^k - C_k \e^\frac18 \eta^\frac{d-2}{8}.
    \end{equation*}
\end{proposition}
\begin{proof}
    By \eqref{est.new.muk}, \eqref{est.upper-1} and Proposition \ref{prop.tmu-mu},
    \begin{equation*}
        \frac{1}{\mu_{\eta,\e}^k} \le C_k \e^\frac18 \eta^\frac{d-2}{8} + \frac{1}{\tilde{\mu}_{\eta,\e}^k} \le C_k \e^\frac18 \eta^\frac{d-2}{8} + \frac{1}{\mu_{\eta}^k} + \frac{|\tilde{\mu}_{\eta,\e}^k-\mu_{\eta}^k|}{\mu_{\eta}^k \tilde{\mu}_{\eta,\e}^k} \le C_k \e^\frac18 \eta^\frac{d-2}{8} + \frac{1}{\mu_{\eta}^k}.
    \end{equation*}
    Hence,
    \begin{equation*}
        \mu_{\eta}^k - \mu_{\eta,\e}^k \le C_k \e^\frac18 \eta^\frac{d-2}{8} \mu_{\eta}^k \mu_{\eta,\e}^k \le C_k \e^\frac18 \eta^\frac{d-2}{8},
    \end{equation*}
    as desired.
\end{proof}

% \begin{proof}[Proof of Theorem \ref{thm.main-EV}]
%     Theorem \ref{thm.main-EV} follows readily from Corollary \ref{coro.1stOrder}, Proposition \ref{prop.mue<mu0+} and Proposition \ref{prop.mue>mu0-}.
% \end{proof}

% We begin with three different eigenvalue problems studied previously, listed as follows:
% \begin{align}
%     \text{Oscillating problem: } \cL_\e(\rho_{\eta,\e}^k) & = \mu_{\eta,\e}^k \phi_{\eta,\e}^2 \rho_{\eta,\e}^k \label{eq.eigen-1},\\
%     \text{Intermediate problem: } \cL_\eta(\tilde{\rho}^k_\eta) & = \tilde{\mu}_\eta^k \phi_{\eta,\e}^2 \tilde{\rho}^k_\eta, \label{eq.eigen-2}\\
%     \text{Homogenized problem: } \cL_\eta({\rho}_\eta^k) & = \mu_\eta^k \phi_{\eta,\e}^2 {\rho}_\eta^k. \label{eq.eigen-3}
% \end{align}
\subsection{Weyl's law}

Throughout the subsection, we will fix $k\ge 1$ and allow $C_k$ depending on $k$.
The convergence rates between the three eigenvalue problems \eqref{eq.eigen-1}-\eqref{eq.eigen-3} can be summarized as follows (Proposition \ref{prop.mue<mu0+}-\ref{prop.mue>mu0-}):
\begin{equation}\label{est.2rates}
    -C_k \e^\frac18 \eta^\frac{d-2}{8} \le \mu_{\eta,\e}^k - \tilde{\mu}_{\eta,\e}^k \le C_k \e \eta^\frac{d-2}{2} \quad \text{and} \quad |\tilde{\mu}_{\eta,\e}^k - \mu_{\eta}^k| \le C_k \e \eta^\frac{d-2}{2}.
\end{equation}
% We recall three corresponding equations:
% \begin{align}
%     \cL_\e(u_{\eta,\e}) &= \phi_{\eta,\e}^2 f, \quad \text{with } f\in L_{\phi_{\eta,\e}}^2(\Omega_{\eta,\e}), u_{\eta,\e} \in H^1_{\phi_{\eta,\e},0}(\Omega_{\eta,\e}). \\
%     \cL_\eta(\tilde{u}_{\eta}) & = \phi_{\eta,\e}^2 f, \quad \text{with } f\in L_{\phi_{\eta,\e}}^2(\Omega_{\eta,\e}), \tilde{u}_{\eta} \in H^1_0(\Omega).\\
%     \cL_\eta(u_{\eta}) & = f, \quad \text{ with } f\in L^2(\Omega), u_{\eta} \in H^1_0(\Omega).
% \end{align}
% Recall the operators:
% \begin{align}
%     &\mathscr{T}_{\eta,\e}: L^2_{\phi_{\eta,\e}}(\Omega_{\eta,\e}) \to H^1_{\phi_{\eta,\e},0}(\Omega_{\eta,\e}), \text{ with } \mathscr{T}_{\eta,\e}(f) = u_{\eta,\e}, \\ 
%     & \widetilde{\mathscr{T}}_\eta: L^2_{\phi_{\eta,\e}}(\Omega_{\eta,\e}) \to H^1_0(\Omega), \text{ with } \widetilde{\mathscr{T}}_\eta(f) = \tilde{u}_{\eta},
% \end{align}
% and
% \begin{equation}
%     \mathscr{T}_{\eta,\e}: L^2(\Omega) \to H^1_0(\Omega), \text{ with } \mathscr{T}_{\eta,\e}(f) = u_{\eta}.
% \end{equation}

%The main tools in this section are the Weyl's law and functional calculus, as well as the (almost) orthogonality of eigenfunctions.

We first recall the classical Weyl's law for the homogenized operator in a bounded domain without holes.

\begin{lemma}[Hermann Weyl, 1911]\label{lem.Weyl}
    There exist $C_2>C_1>0$ such that for any $j\ge 1$, we have
    \begin{equation*}
        C_1 j^{\frac{2}{d}} \le \mu_{\eta}^j \le C_2 j^{\frac{2}{d}}.
    \end{equation*}
\end{lemma}
In the above lemma, $C_1$ and $C_2$ are universal constants. The convergence rates in \eqref{est.2rates} and the Weyl's law for the homogenized problem give the distributions of the first $k$ eigenvalues for the degenerate problem and intermediate problem, 
when $\e \eta^\frac{d-2}{2}$ is sufficiently small. In particular, we have the following crucial lemma.

\begin{lemma}[Existence of large common spectral gaps]\label{lem.N1gap}
    Fix $k \ge 1$. There exist $C_k>0$ (depending on $k$) and a universal constant $M>0$ such that whenever $\e \eta^\frac{d-2}{2}< C_k^{-1}$, there exists $N_1 = N_1(k,\e,\eta) \in [k,Mk)$ such that
    \begin{equation}\label{est.LargeGap-1}
        \min\{ \mu_{\eta,\e}^{N_1+1}, \tilde{\mu}_{\eta,\e}^{N_1+1}, \mu_\eta^{N_1+1}\} - \max\{ \mu_{\eta,\e}^{N_1}, \tilde{\mu}_{\eta,\e}^{N_1}, \mu_\eta^{N_1}\} \ge H_k \approx k^{\frac{2-d}{d}},
    \end{equation}
    and
    \begin{equation}\label{est.LargeGap-2}
        \min\Big\{ \frac{1}{\mu_{\eta,\e}^{N_1}} , \frac{1}{\tilde{\mu}_{\eta,\e}^{N_1}} , \frac{1}{\mu_\eta^{N_1}}  \Big\} - \max \Big\{\frac{1}{\mu_{\eta,\e}^{N_1 + 1}},\frac{1}{\tilde{\mu}_{\eta,\e}^{N_1 + 1}}, \frac{1}{\mu_\eta^{N_1 + 1}}\Big\} \ge G_k \approx k^{\frac{-2-d}{d}}.
    \end{equation}
\end{lemma}

\begin{proof}
     Let $M$ be a universal constant such that $C_1 M^{\frac{2}{d}} > 2C_2$, where $C_1, C_2$ are given in Lemma \ref{lem.Weyl}. The Wyel's law for $\mu_{\eta}^j$ tells us that $C_1 k^{\frac{2}{d}} \le \mu_{\eta}^k \le C_2 k^{\frac{2}{d}}$ and $C_1 (Mk)^{\frac{2}{d}} \le \mu_\eta^{Mk} \le C_2 (Mk)^{\frac{2}{d}}$. Thus,
     \begin{equation*}
         \mu_\eta^{Mk} - \mu_{\eta}^k \ge (C_1 M^{\frac{2}{d}} - C_2) k^\frac{2}{d} > C_2 k^\frac{2}{d},
     \end{equation*}
     by our choice of $M$. By the pigeonhole principle, there exists $N_1 = N_1(k,\e,\eta)  \in [k, Mk)$ such that
     \begin{equation*}
         \mu_\eta^{N_1+1} - \mu_\eta^{N_1} \ge \frac{C_2 k^{\frac{2}{d}}}{(M-1)k} = C_3 k^{\frac{2}{d} -1}.
     \end{equation*}
     As a result, we have found a large spectral gap between two successive eigenvalues (counted with multiplicity) for the homogenized problem.
     By the convergence rates of eigenvalues \eqref{est.2rates}, if $\e \eta^\frac{d-2}{2} < C_k^{-1}$ for sufficiently large $C_k$ depending on $k$, we have \eqref{est.LargeGap-1} with $H_k = \frac{1}{2} C_3 k^{\frac{2}{d} -1}$.

     Using Lemma \ref{lem.Weyl} again, we have
     \begin{equation*}
         \frac{1}{\mu_\eta^{N_1}} - \frac{1}{\mu_\eta^{N_1+1}} = \frac{\mu_\eta^{N_1+1} - \mu_\eta^{N_1}}{ \mu_\eta^{N_1+1} \mu_\eta^{N_1}} \ge \frac{C_3 k^{\frac{2}{d} -1}}{C_1^2 (N_1+1)^{\frac{4}{d}}} = C_4 k^{-\frac{2}{d}-1}.
     \end{equation*}
     Again, by the convergence rates, for $\e \eta^\frac{d-2}{2}<C_k^{-1}$, we obtain \eqref{est.LargeGap-2} with $G_k = \frac12 C_4 k^{-\frac{2}{d}-1}$.
     \end{proof}

\section{Optimal convergence rates}
For convenience in this section, we will use the notations,
\begin{equation*}
    \Ag{f,g}_{\phi_{\eta,\e}} = \int_{\Omega_{\eta,\e}} \phi_{\eta,\e}^2 f g \quad \text{and} \quad \Ag{f,g} = \int_{\Omega} fg,
\end{equation*}
to represent the inner products in $L^2_{\phi_{\eta,\e}}(\Omega)$ and $L^2(\Omega)$, respectively. 
%In Subsection \ref{sec.6.3}, the complex-valued functions will be involved.
\subsection{First-order approximation by higher regularity and duality}
\begin{theorem}\label{thm.SubOptRate}
Suppose that $\Omega$ is a bounded $C^3$ domain and $\Omega_{\eta,\e}$ satisfies the geometric assumption \textbf{A}. Then for any $f\in W^{1,p}(\Omega)$ with some $p>d$,
\begin{equation}
    \| \mathscr{T}_{\eta,\e} f - \mathscr{T}_{\eta} f \|_{L^2_{\phi_{\eta,\e}}(\Omega_{\eta,\e})} \le C\e^\frac12 \eta^{\frac{d-2}{2}} \| f \|_{W^{1,p}(\Omega)}.
\end{equation}
\end{theorem}

Let $u_{\eta,\e} = \mathscr{T}_{\eta,\e} f$ and $u_{\eta} = \mathscr{T}_{\eta} f$. Then $u_{\eta,\e} \in V_{\eta,\e}$ satisfies $\cL_{\eta,\e}(u_{\eta,\e}) = \phi_{\eta,\e}^2 f$ in $\Omega_{\eta,\e}$ and $u_{\eta} \in H^1_0(\Omega)$ satisfies $\cL_\eta(u_{\eta}) =  f$ in $\Omega$. Since $f\in W^{1,p}(\Omega)$ and $\Omega$ is a bounded $C^{3}$ domain, then $u_{\eta} \in W^{3,p}(\Omega)$.

    We redefine
    \begin{equation}
        w_{\eta,\e} = u_{\eta,\e} - u_{\eta} - \e \chi_{\eta,\e} \cdot \nabla u_{\eta} + v_{\eta,\e},
    \end{equation}
    where $v_{\eta,\e}$ is the solution of
    \begin{equation}\label{eq.def.bl}
        \cL_{\eta,\e}(v_{\eta,\e}) = 0 \text{ in } \Omega_{\eta,\e} \quad \text{and} \quad v_{\eta,\e} = \e \chi_{\eta,\e} \cdot \nabla u_{\eta} \text{ on } \Gamma_{\eta,\e}.
    \end{equation}
    It is not hard to see that $v_{\eta,\e} \in H^1_{\phi_{\eta,\e}}(\Omega_{\eta,\e})$. Thus we have $w_{\eta,\e} \in V_{\eta,\e}$ since $w_{\eta,\e} = 0$ on $\Gamma_{\eta,\e}$.

    \begin{lemma}\label{lem.OptH1rate}
        Under the assumptions of Theorem \ref{thm.SubOptRate}, it holds
        \begin{equation}
            \| \phi_{\eta,\e} w_{\eta,\e} \|_{L^2(\Omega_{\eta,\e})} + \| \phi_{\eta,\e} \nabla w_{\eta,\e} \|_{L^2(\Omega_{\eta,\e})} \le C\e \eta^{\frac{d-2}{2}} \| f \|_{W^{1,p}(\Omega)}.
        \end{equation}
    \end{lemma}
    This lemma should be compared with Theorem \ref{thm.we.flp}. In Theorem \ref{thm.we.flp}, we only require $f\in L^p(\Omega_{\eta,\e})$ for some $p>d$; but the convergence rate is worse. While in the above lemma, we require $f\in W^{1,p}(\Omega)$ with $p>d$ and have the optimal convergence rate. This improvement, though with additional boundary layers involved, will help us to prove the optimal convergence rates of eigenvalues.
    \begin{proof}[Proof of Lemma \ref{lem.OptH1rate}]
    The proof of this lemma is inspired by \cite{OSY92} and different from Lemma \ref{lem.Dwe.Dh}.
    We first perform a straightforward calculation as follows
    \begin{equation}\label{eq.we.V2}
        \begin{aligned}
            \cL_{\eta,\e}(w_{\eta,\e}) = \phi_{\eta,\e}^2 f + \partial_j (\phi_{\eta,\e}^2 \partial_j u_{\eta}) + \partial_j (\phi_{\eta,\e}^2 (\partial_j \chi^\ell_\eta)_\e \partial_\ell u_{\eta}) + \e \partial_j (\phi_{\eta,\e}^2 \chi^\ell_{\eta,\e} \partial_j \partial_\ell u_{\eta}).
        \end{aligned}
    \end{equation}
    Using the equation for $\chi_\eta$ \eqref{eq.corrector}, we have
    \begin{equation}
        \partial_j (\phi_{\eta,\e}^2 \partial_j u_{\eta}) + \partial_j (\phi_{\eta,\e}^2 (\partial_j \chi^\ell_\eta)_\e \partial_\ell u_{\eta}) = \phi_{\eta,\e}^2 \partial_j^2 u_{\eta} + \phi_{\eta,\e}^2 (\partial_j \chi^\ell_\eta)_\e \partial_j\partial_\ell u_{\eta}.
    \end{equation}
    Thus, by this identity and the equation $\phi_{\eta,\e}^2 f = -\phi_{\eta,\e}^2 (\bar{a}_{\eta})_{j\ell} \partial_j \partial_\ell u_{\eta}$, we obtain from \eqref{eq.we.V2} that
    \begin{equation}\label{eq.weeta-1}
        \cL_{\eta,\e}(w_{\eta,\e}) = \phi_{\eta,\e}^2 (-(\bar{a}_{\eta})_{j\ell} + \delta_{j\ell} + (\partial_j \chi^\ell_\eta)_\e) \partial_j \partial_\ell u_{\eta} + \e \partial_j (\phi_{\eta,\e}^2 \chi^\ell_{\eta,\e} \partial_j \partial_\ell u_{\eta}).
    \end{equation}

    Now we introduce new flux correctors $\Theta_{\eta}^{j\ell} \in H^1_{\phi_{\eta},{\rm per}}(Y_{\eta})$ satisfying the degenerate cell problem
    \begin{equation}\label{eq.Theta}
        \nabla\cdot (\phi_{\eta}^2 \nabla \Theta_\eta^{j\ell}) = \phi_{\eta}^2(-(\bar{a}_{\eta})_{j\ell} + \delta_{j\ell} + (\partial_j \chi^\ell_\eta)),\quad \text{in } Y_\eta.
    \end{equation}
    Note that the above equation is solvable since the right-hand side has mean value zero in $Y$. Moreover, the energy estimate implies
    \begin{equation}\label{est.Theta}
        \| \phi_\eta \nabla \Theta_\eta^{j\ell} \|_{L^2(Y_\eta)}  \le C\eta^\frac{d-2}{2}.
    \end{equation}

    As a result of \eqref{eq.weeta-1} and \eqref{eq.Theta},
    \begin{equation}
    \begin{aligned}
        \cL_{\eta,\e}(w_{\eta,\e}) & = \nabla\cdot (\e \phi_{\eta,\e}^2 (\nabla \Theta^{j\ell}_{\eta})_\e ) \partial_j \partial_\ell u_{\eta} + \e \partial_j (\phi_{\eta,\e}^2 \chi^\ell_{\eta,\e} \partial_j \partial_\ell u_{\eta}) \\
        & = \partial_i (\e \phi_{\eta,\e}^2 (\partial_i \Theta^{j\ell}_{\eta})_\e \partial_j \partial_\ell u_{\eta} ) - \e \phi_{\eta,\e}^2 (\partial_i \Theta^{j\ell}_{\eta})_\e \partial_i \partial_j \partial_\ell u_{\eta} + \e \partial_j (\phi_{\eta,\e}^2 \chi^\ell_{\eta,\e} \partial_j \partial_\ell u_{\eta}).
    \end{aligned}
    \end{equation}
    Integrating this equation against $w_{\eta,\e}$, we have
    \begin{equation}\label{eq.wee}
        \begin{aligned}
            \int_{\Omega_{\eta,\e}} \phi_{\eta,\e}^2 |\nabla w_{\eta,\e}|^2 & = -\e \int_{\Omega_{\eta,\e}} \phi_{\eta,\e} (\nabla \Theta_\eta^{j\ell})_\e \partial_j \partial_\ell u_\eta \cdot \phi_{\eta,\e} \nabla w_{\eta,\e} \\
            & \qquad - \e \int_{\Omega_{\eta,\e}} \phi_{\eta,\e}(\nabla \Theta_\eta^{j\ell})_\e \cdot \nabla \partial_j \partial_\ell u_\eta \phi_{\eta,\e} w_{\eta,\e} \\
            & \qquad - \e \int_{\Omega_{\eta,\e}} \phi_{\eta,\e}\chi_{\eta,\e}^\ell \nabla \partial_\ell u_\eta \cdot \phi_{\eta,\e} \nabla w_{\eta,\e}.
        \end{aligned}
    \end{equation}
    By the H\"{o}lder's inequality and \eqref{est.Theta}, the first integral is bounded by
    \begin{equation}\label{est.wee-1}
    \begin{aligned}
        & \e \| \phi_{\eta,\e} (\nabla \Theta_\eta^{j\ell})_\e  \|_{L^2(\Omega_{\eta,\e})} \| \nabla^2 u_\eta \|_{L^\infty(\Omega)} \| \phi_{\eta,\e} \nabla w_{\eta,\e} \|_{L^2(\Omega_{\eta,\e})} \\
        & \le C\e \eta^\frac{d-2}{2} \| \nabla^2 u_\eta \|_{L^\infty(\Omega)} \| \phi_{\eta,\e} \nabla w_{\eta,\e} \|_{L^2(\Omega_{\eta,\e})}.
    \end{aligned}
    \end{equation}
    The third integral of \eqref{eq.wee} has the same bound as \eqref{est.wee-1}. By Theorem \ref{thm.WSPI.eta}, the second integral of \eqref{eq.wee} is bounded by
    \begin{equation}
    \begin{aligned}
        & \e \| \phi_{\eta,\e} (\nabla \Theta_\eta^{j\ell})_\e  \|_{L^2(\Omega_{\eta,\e})} \| \nabla^3 u_\eta \|_{L^d(\Omega)} \| \phi_{\eta,\e} w_{\eta,\e} \|_{L^{2^*}(\Omega_{\eta,\e})} \\
        & \le C\e \eta^\frac{d-2}{2} \| \nabla^3 u_\eta \|_{L^d(\Omega)} \| \phi_{\eta,\e} \nabla w_{\eta,\e} \|_{L^{2}(\Omega_{\eta,\e})}.
    \end{aligned}       
    \end{equation}
    
    % By a standard energy estimate, we obtain
    % \begin{equation}
    % \begin{aligned}
    %      \| \phi_{\eta,\e} \nabla w_{\eta,\e} \|_{L^2(\Omega_{\eta,\e})} & \le C\e \| \phi_{\eta,\e} (\nabla \Theta_{\eta})_\e \|_{L^2(\Omega_{\eta,\e})} \| \nabla^2 u_{\eta} \|_{L^\infty(\Omega)} \\
    %     & \qquad + C\e \| \phi_{\eta,\e} (\nabla \Theta_{\eta})_\e \|_{L^2(\Omega_{\eta,\e})} \| \nabla^3 u_{\eta} \|_{L^d(\Omega)} \\
    %     & \qquad + C\e \| \phi_{\eta,\e} \chi_{\eta,\e} \|_{L^2(\Omega_{\eta,\e})} \| \nabla^2 u_{\eta} \|_{L^\infty(\Omega)}.
    % \end{aligned}
    % \end{equation}
    % It is not hard to show
    % \begin{equation}
    %     \| \phi_{\eta} \nabla \Theta_\eta \|_{L^2(Y_{\eta})} \le C\eta^{\frac{d-2}{2}}.
    % \end{equation}
    Combining these estimates together with the Sobolev embedding thoerem, we arrive at
    \begin{equation}
        \| \phi_{\eta,\e} \nabla w_{\eta,\e} \|_{L^2(\Omega_{\eta,\e})} \le C\e \eta^\frac{d-2}{2} \| u_{\eta} \|_{W^{3,p}(\Omega)} \le C\e \eta^\frac{d-2}{2} \| f\|_{W^{1,p}(\Omega)},
    \end{equation}
    for $p>d$. The proof is complete by Theorem \ref{thm.WSPI.eta}.
    % Finally since $\Omega_{\eta,\e}$ satisfies the geometric assumption $H$ and due to \eqref{est.chieta.pw}, we have $\chi_\eta \le C\eta^\frac{d-2}{2}$ on $\partial \Omega$. Thus, the maximum principle implies
    % \begin{equation}
    %     \| v_{\eta,\e} \|_{L^\infty(\Omega_{\eta,\e})} \le C\e \eta^{\frac{d-2}{2}} \| \nabla u_{\eta} \|_{L^\infty(\Omega)}.
    % \end{equation}
    % Moreover, 
    % \begin{equation}
    %     \| \e \phi_{\eta,\e} \chi_{\eta,\e} \nabla u_{\eta} \|_{L^2(\Omega_{\eta,\e})} \le C\e \eta^\frac{d-2}{2} \| \nabla u_{\eta} \|_{L^\infty(\Omega)}.
    % \end{equation}
    % Combining the last three inequalities and by the triangle inequality, we obtain
    % \begin{equation}
    %     \| \phi_{\eta,\e}(u_{\eta,\e} - u_{\eta}) \|_{L^2(\Omega_{\eta,\e})} \le C\e \eta^\frac{d-2}{2} \| f\|_{W^{1,p}(\Omega)}.
    % \end{equation}
    % This is the desired estimate.
    \end{proof}

\begin{proof}
    [Proof of Theorem \ref{thm.SubOptRate}]
    In view of Lemma \ref{lem.OptH1rate}, it suffices to show
    \begin{equation}\label{est.pf6.1-0}
        \| \phi_{\eta,\e} (\e \chi_{\eta,\e}\cdot \nabla u_{\eta} - v_{\eta,\e}) \|_{L^2(\Omega_{\eta,\e})} \le C\e^\frac12 \eta^{\frac{d-2}{2}} \| f\|_{W^{1,p}(\Omega)}.
    \end{equation}
    Clearly,
    \begin{equation}\label{est.pf6.1-1}
        \| \phi_{\eta,\e} \e \chi_{\eta,\e}\cdot \nabla u_{\eta}  \|_{L^2(\Omega_{\eta,\e})} \le C\e \| \phi_{\eta,\e} \chi_{\eta,\e} \|_{L^2(\Omega_{\eta,\e})} \| \nabla u_{\eta} \|_{L^\infty(\Omega)} \le C\e \eta^{\frac{d-2}{2}} \| f\|_{W^{1,p}(\Omega)}.
    \end{equation}
    To estimate $v_{\eta,\e}$, we consider
    \begin{equation}\label{def.tildev}
        \widetilde{v}_{\eta,\e} = v_{\eta,\e} - \e \chi_{\eta,\e}\cdot \nabla u_{\eta} \zeta_\e,
    \end{equation}
    where $\zeta_\e = 1-\theta_\e$ and $\theta_\e$ is given as in Section \ref{sec.4}. Thus $\zeta_\e = 1$ on $\Gamma_{\eta,\e}$ and $\zeta_\e$ is supported in $\Omega(2\e)$ and $|\nabla \zeta_\e| \le C\e^{-1}$. 
    Then $\widetilde{v}_{\eta,\e} \in V_{\eta,\e}$ and
    \begin{equation}\label{eq.tildev}
        \cL_{\eta,\e}(\widetilde{v}_{\eta,\e}) = -\nabla\cdot (\phi_{\eta,\e}^2 \nabla (\e \chi_{\eta,\e}\cdot \nabla u_{\eta} \zeta_\e)).
    \end{equation}
    By the energy estimate, we have
    \begin{equation}\label{est.tildev-1}
    \begin{aligned}
        \| \phi_{\eta,\e} \nabla \widetilde{v}_{\eta,\e} \|_{L^2(\Omega_{\eta,\e})} & \le C \| \phi_{\eta,\e} \nabla (\e \chi_{\eta,\e}\cdot \nabla u_{\eta} \zeta_\e) \|_{L^2(\Omega_{\eta,\e})} \\
        & \le C\| \phi_{\eta,\e}(\nabla \chi_\eta)_\e\cdot \nabla u_{\eta} \zeta_\e \|_{L^2(\Omega_{\eta,\e})} + C\e \| \phi_{\eta,\e}\chi_{\eta,\e} \cdot \nabla^2 u_{\eta} \zeta_\e \|_{L^2(\Omega_{\eta,\e})} \\
        & \qquad + C\e \| \phi_{\eta,\e}\chi_{\eta,\e} \cdot \nabla u_{\eta} \nabla \zeta_\e \|_{L^2(\Omega_{\eta,\e})} \\
        & \le C\e^\frac12 \| \nabla u_{\eta} \|_{L^\infty(\Omega)} \| \phi_{\eta} \nabla \chi_\eta \|_{L^2(Y_{\eta})} + C\e \| \nabla^2 u_{\eta} \|_{L^d(\Omega)} \| \phi_{\eta} \chi_\eta \|_{L^{2^*}(Y_{\eta})} \\
        & \qquad + C\e^\frac12 \| \nabla u_{\eta} \|_{L^\infty(\Omega)} \| \phi_{\eta} \chi_\eta \|_{L^2(Y_{\eta})} \\
        & \le C\e^\frac12 \eta^{\frac{d-2}{2}} \| f\|_{W^{1,p}(\Omega)}.
    \end{aligned}
    \end{equation}
    The weighted Poincar\'{e} inequality \eqref{est.WSPI} then implies
    \begin{equation}
        \| \phi_{\eta,\e} \widetilde{v}_{\eta,\e} \|_{L^2(\Omega_{\eta,\e})} \le C\| \phi_{\eta,\e} \nabla \widetilde{v}_{\eta,\e} \|_{L^2(\Omega_{\eta,\e})} \le C\e^\frac12 \eta^{\frac{d-2}{2}} \| f\|_{W^{1,p}(\Omega)}.
    \end{equation}
    It follows by \eqref{def.tildev} that
    \begin{equation}\label{est.pf6.1-10}
        \| \phi_{\eta,\e} v_{\eta,\e} \|_{L^2(\Omega_{\eta,\e})} + \| \phi_{\eta,\e} \nabla v_{\eta,\e} \|_{L^2(\Omega_{\eta,\e})} \le C\e^\frac12 \eta^{\frac{d-2}{2}} \| f\|_{W^{1,p}(\Omega)}.
    \end{equation}
    This and \eqref{est.pf6.1-1} lead to \eqref{est.pf6.1-0}.
\end{proof}

The above estimate shows that the boundary layer term $v_{\eta,\e}$ causes the suboptimal rate of convergence. We can improve the convergence rate if we can improve the estimate of $v_{\eta,\e}$. One way used by \cite{OSY92} is the maximum principle, which leads to a rate of $O(\e)$ and the factor $\eta^\frac{d-2}{2}$ is lost (unless the holes do not touch the boundary $\partial \Omega$). In our case, we can only show an optimal convergence rate in a very weak sense by duality. Surprisingly, this weak convergence is sufficient for the optimal convergence rates of eigenvalues.

\begin{lemma}\label{lem.dual.ve}
    Let $v_{\eta,\e}$ be given as above. Let $g \in W^{1,p}(\Omega)$ for some $p>d$. Then
    \begin{equation}\label{est.duality.vg}
        |\Ag{ v_{\eta,\e}, g }_{\phi_{\eta,\e}} | \le C\e \eta^{\frac{d-2}{2}} \| f\|_{W^{1,p}(\Omega)} \| g\|_{W^{1,p}(\Omega)}.
    \end{equation}
\end{lemma}

\begin{proof}
    Let $\widetilde{v}_{\eta,\e}$ be given by \eqref{def.tildev}. Due to \eqref{est.pf6.1-1}, it suffices to prove \eqref{est.duality.vg} for $\widetilde{v}_{\eta,\e}$ in place of $v_{\eta,\e}$. Let $h_{\eta,\e} = \mathscr{T}_{\eta,\e}(g)$ and $h_\eta = \mathscr{T}_{\eta} (g)$. It follows from Lemma \ref{lem.OptH1rate} and \eqref{est.pf6.1-10} that
    \begin{equation}\label{est.dual.he}
        \| \phi_{\eta,\e} \nabla (h_{\eta,\e} - h_\eta - \e \chi_{\eta,\e} \cdot \nabla h_\eta) \|_{L^2(\Omega_{\eta,\e})} \le C\e^\frac12 \eta^\frac{d-2}{2} \| g\|_{W^{1,p}(\Omega)}
    \end{equation}

    Now, observe that
    \begin{equation}
    \begin{aligned}
        \Ag{ \widetilde{v}_{\eta,\e}, g }_{\phi_{\eta,\e}} & = \int_{\Omega_{\eta,\e}} \phi_{\eta,\e}^2 \nabla \widetilde{v}_{\eta,\e}\cdot \nabla h_{\eta,\e} \\
        & = \int_{\Omega_{\eta,\e}} \phi_{\eta,\e}^2 \nabla \widetilde{v}_{\eta,\e}\cdot \nabla ( h_{\eta,\e} - h_\eta - \e \chi_{\eta,\e} \cdot \nabla h_\eta) \\
        & \qquad + \int_{\Omega_{\eta,\e}} \phi_{\eta,\e}^2 \nabla \widetilde{v}_{\eta,\e}\cdot \nabla h_\eta + \int_{\Omega_{\eta,\e}} \phi_{\eta,\e}^2 \nabla \widetilde{v}_{\eta,\e}\cdot \nabla (\e \chi_{\eta,\e} \cdot \nabla h_\eta) \\
        & = I_1 + I_2 + I_3.
    \end{aligned}
    \end{equation}
    By \eqref{est.tildev-1} and \eqref{est.dual.he}, we have
    \begin{equation}
        |I_1| \le C\e \eta^{d-2} \| f\|_{W^{1,p}(\Omega)} \| g\|_{W^{1,p}(\Omega)}.
    \end{equation}
    To estimate $I_2$, we use the variational equation of \eqref{eq.tildev} to get
    \begin{equation}
    \begin{aligned}
        I_2 & = \int_{\Omega_{\eta,\e}} \phi_{\eta,\e}^2 \nabla(\e \chi_{\eta,\e}\cdot \nabla u_{\eta} \zeta_\e) \cdot \nabla h_\eta \\
        & = \int_{\Omega_{\eta,\e}} \phi_{\eta,\e}^2 (\nabla \chi_{\eta}^\ell)_\e \partial_\ell u_{\eta} \zeta_\e \cdot \nabla h_\eta + \int_{\Omega_{\eta,\e}} \phi_{\eta,\e}^2 \e \chi_{\eta,\e}\cdot \nabla^2 u_{\eta} \zeta_\e \cdot \nabla h_\eta \\
        & \qquad + \int_{\Omega_{\eta,\e}} \phi_{\eta,\e}^2 \e \chi_{\eta,\e}\cdot \nabla u_{\eta} \nabla \zeta_\e \cdot \nabla h_\eta\\
        & = I_{2,1} + I_{2,2} + I_{2,3}.
    \end{aligned}
    \end{equation}
    Note that $\zeta_\e$ is supported in $\Omega(2\e)$. Thus
    \begin{equation}
    \begin{aligned}
        |I_{2,1}| & \le \| \phi_{\eta,\e} (\nabla \chi_{\eta}^\ell)_\e \|_{L^2(\Omega_{\eta,\e} \cap \Omega(2\e))} C \e^\frac12 \| \nabla u_\eta \|_{L^\infty(\Omega)} \| \nabla h_\eta \|_{L^\infty(\Omega)} \\
        & \le
        C\e \eta^\frac{d-2}{2} \| \nabla u_\eta \|_{L^\infty(\Omega)} \| \nabla h_\eta \|_{L^\infty(\Omega)}.
    \end{aligned}
    \end{equation}
    The estimates for $I_{2,2}$ and $I_{2,3}$ are similar. Thus,
    \begin{equation}
        |I_2| \le C\e \eta^\frac{d-2}{2} \| f\|_{W^{1,p}(\Omega)} \| g\|_{W^{1,p}(\Omega)}.
    \end{equation}
    By the same argument, we can estimate $I_3$ as
    \begin{equation}
        |I_3| \le C\e \eta^{d-2} \| f\|_{W^{1,p}(\Omega)} \| g\|_{W^{1,p}(\Omega)}.
    \end{equation}
    Combining the estimates of $I_1$-$I_3$, we obtain the desired estimate.
\end{proof}

\begin{theorem}\label{thm.optDualrate}
    Suppose that $\Omega$ is a bounded $C^3$ domain and $\Omega_{\eta,\e}$ satisfies the geometric assumption \textbf{A}. Then for any $f, g\in W^{1,p}(\Omega)$,
\begin{equation}
    | \Ag{\mathscr{T}_{\eta,\e} f - \mathscr{T}_{\eta} f, g }_{\phi_{\eta,\e}} | \le C\e \eta^{\frac{d-2}{2}} \| f \|_{W^{1,p}(\Omega)} \| g \|_{W^{1,p}(\Omega)}.
\end{equation}
\end{theorem}
\begin{proof}
    This follows readily from Lemma \ref{lem.OptH1rate} and Lemma \ref{lem.dual.ve}.
\end{proof}

% This estimate in duality form is optimal in terms of the order $O(\e \eta^{\frac{d-2}{2}} )$ and is sufficient for proving the optimal convergence rates of eigenvalues and eigenfunctions.

% If the holes are allowed to intersect with the boundary, then the proof of the main theorem will be a little harder. Actually, we need one more step to bootstrap the convergence rate. By the energy estimate, we have
% \begin{equation}
%     \| \phi_{\eta,\e} \nabla v_{\eta,\e} \|_{L^2(\Omega_{\eta,\e})} \le C\e^\frac12 \eta^\frac{d-2}{2}.
% \end{equation}
% This would give a rate of $O(\e^\frac12 \eta^\frac{d-2}{2})$ including
% \begin{equation}
%     \| \rho_{\eta,\e}^j - \mathcal{P}_\eta \widetilde{\mathcal{P}}_\eta \rho_{\eta,\e}^j \|_{L^2_{\phi_{\eta,\e}}(\Omega_{\eta,\e})} \le C\e^\frac12 \eta^\frac{d-2}{2}.
% \end{equation}
% Now observe that $\mathcal{P}_\eta \widetilde{\mathcal{P}}_\eta \rho_{\eta,\e}^j$ is more regular. Then we can apply the following identity
% \begin{equation}
%     \begin{aligned}
%         [(\mu_{\eta,\e}^j)^{-1} - (\mu_{\eta}^\ell)^{-1}] \Ag{\rho_{\eta,\e}^j, \rho_{\eta}^\ell}_{\phi_{\eta,\e}} & = \Ag{\rho_{\eta,\e}^j - \mathcal{P}_\eta \widetilde{\mathcal{P}}_\eta \rho_{\eta,\e}^j, (\mathscr{T}_{\eta,\e}-\mathscr{T}_{\eta})(\rho_{\eta}^\ell) }_{\phi_{\eta,\e}} \\
%         & \quad + \Ag{ \mathcal{P}_\eta \widetilde{\mathcal{P}}_\eta \rho_{\eta,\e}^j, (\mathscr{T}_{\eta,\e}-\mathscr{T}_{\eta})(\rho_{\eta}^\ell) }_{\phi_{\eta,\e}}.
%     \end{aligned}
% \end{equation}
% This allows us to improve the rate of convergence.

\subsection{Almost orthogonality}
In this subsection, we show that the eigenfunctions (of different eigenvalue problems listed in \eqref{eq.eigen-1}-\eqref{eq.eigen-3}) are almost orthogonal if their corresponding eigenvalues are not close to each other.

\begin{lemma}[Suboptimal almost orthogonality]
For any $1\le j, \ell \le k$,
\begin{equation}\label{est.AlmostOrtho}
    | \Ag{ \rho_{\eta,\e}^j, \rho_{\eta}^\ell }_{\phi_{\eta,\e}} | \le \frac{C_k \e^\frac12 \eta^\frac{d-2}{2} \mu_{\eta,\e}^j \mu_{\eta}^\ell }{|\mu_{\eta}^\ell - \mu_{\eta,\e}^j| }.
\end{equation}
\end{lemma}
\begin{proof}
    Since $\mathscr{T}_{\eta,\e}$ is self-adjoint in $L^2_{\phi_{\eta,\e}}(\Omega_{\eta,\e})$, we have
    \begin{align}\label{eq.AO-1}
    \begin{aligned}
        (\mu_{\eta,\e}^j)^{-1} \Ag{\rho_{\eta,\e}^j, \rho_{\eta}^\ell}_{\phi_{\eta,\e}} & = \Ag{\mathscr{T}_{\eta,\e}(\rho_{\eta,\e}^j), \rho_{\eta}^\ell}_{\phi_{\eta,\e}} \\
        & = \Ag{\rho_{\eta,\e}^j, \mathscr{T}_{\eta,\e}(\rho_{\eta}^\ell) }_{\phi_{\eta,\e}} \\
        & = \Ag{\rho_{\eta,\e}^j, \mathscr{T}_{\eta}(\rho_{\eta}^\ell) }_{\phi_{\eta,\e}} + \Ag{\rho_{\eta,\e}^j, (\mathscr{T}_{\eta,\e}-\mathscr{T}_{\eta})(\rho_{\eta}^\ell) }_{\phi_{\eta,\e}} \\
        & = (\mu_{\eta}^\ell)^{-1} \Ag{\rho_{\eta,\e}^j, \rho_{\eta}^\ell }_{\phi_{\eta,\e}} + \Ag{\rho_{\eta,\e}^j, (\mathscr{T}_{\eta,\e}-\mathscr{T}_{\eta})(\rho_{\eta}^\ell) }_{\phi_{\eta,\e}}.
    \end{aligned}
    \end{align}
    Thus,
    \begin{equation}\label{eq.AO-2}
        ((\mu_{\eta,\e}^j)^{-1} -  (\mu_{\eta}^\ell)^{-1}) \Ag{\rho_{\eta,\e}^j, \rho_{\eta}^\ell}_{\phi_{\eta,\e}} = \Ag{\rho_{\eta,\e}^j, (\mathscr{T}_{\eta,\e}-\mathscr{T}_{\eta})(\rho_{\eta}^\ell) }_{\phi_{\eta,\e}}.
    \end{equation}
    Using the regularity of $\rho_{\eta}^\ell \in W^{1,p}(\Omega)$, and Theorem \ref{thm.SubOptRate}, we obtain
    \begin{equation*}
        |\Ag{\rho_{\eta,\e}^j, (\mathscr{T}_{\eta,\e}-\mathscr{T}_{\eta})(\rho_{\eta}^\ell) }_{\phi_{\eta,\e}}| \le C_k\e^\frac12 \eta^\frac{d-2}{2}. 
    \end{equation*}
    This and \eqref{eq.AO-2} imply the desired estimate.
\end{proof}

To improve the error of almost orthogonality, we need a suboptimal convergence rates for the eigenfunctions. The following lemma is an abstract result taken from \cite{OSY92}.

\begin{lemma}\label{lem.Abst.spec}
    Let $H$ be a separable Hilbert space endowed with a real-valued inner product $\Ag{\cdot,\cdot}_H$ and the norm $\| \cdot \|_H = \Ag{\cdot,\cdot}_H^\frac12$. Let $T:H\to H$ be a bounded linear compact self-adjoint operator. Suppose there exist $u\in H$ with $\| u\|_{H} = 1$ and $\sigma>0$ such that
    \begin{equation}
        \| Tu - \sigma u \|_H \le \delta,
    \end{equation}
    for some $\delta \ge 0$. Then there exists at least one eigenvalue $\sigma^k$ of $T$ such that
    \begin{equation}
        |\sigma^k - \sigma| \le \delta.
    \end{equation}
    Moreover, for any $L>0$, 
    \begin{equation}\label{est.inf.u-v}
        \inf_{v\in S[\sigma-L,\sigma+L]} \| u - v \|_H \le 2\delta L^{-1},
    \end{equation}
    where $S[\sigma-L,\sigma+L]$ is the eigenspace spanned by the eigenvectors corresponding to the eigenvalues in the band $[\sigma-L,\sigma+L]$.
\end{lemma}

For each $k$, let $N_1 = N_1(k,\e,\eta)$ be given as in Lemma \ref{lem.N1gap}. Let $\mathcal{P}_{\eta,\e}^N$ be the spectral projection of $\mathscr{T}_{\eta,\e}$ onto the eigenspace $S_{\eta,\e}^N = \text{span}\{ \rho_{\eta,\e}^1, \rho_{\eta,\e}^2,\dots, \rho_{\eta,\e}^N  \}$, namely, for any $f\in L^2_{\phi_{\eta,\e}}(\Omega_{\eta,\e})$,
\begin{equation*}
    \mathcal{P}_{\eta,\e}^N f = \sum_{j=1}^N \Ag{f, \rho_{\eta,\e}^j}_{\phi_{\eta,\e}} \rho_{\eta,\e}^j.
\end{equation*}
The spectral projection $\mathcal{P}_{\eta,\e}^N$ satisfies $(\mathcal{P}_{\eta,\e}^N)^2 = \mathcal{P}_{\eta,\e}^N$.

Similarly, we can define $\mathcal{P}_\eta^N$ to be the projection of $ \mathscr{T}_\eta^N$ onto the eigenspace $S_\eta^N = \text{span}\{ \rho_\eta^1,\rho_\eta^2,\cdots, \rho_\eta^N \}$.

Let $f = \rho_{\eta}^j$ with $1\le j\le N_1$. Then by Theorem \ref{thm.SubOptRate},
\begin{equation}
    \| \mathscr{T}_{\eta,\e} \rho_{\eta}^j - \mathscr{T}_{\eta} \rho_{\eta}^j \|_{L^2_{\phi_{\eta,\e}}(\Omega_{\eta,\e})} \le C\e^\frac12 \eta^{\frac{d-2}{2}} \| \rho_{\eta}^j \|_{W^{1,p}(\Omega)} \le C_k \e^\frac12 \eta^{\frac{d-2}{2}} .
\end{equation}
Note that $\mathscr{T}_{\eta} \rho_{\eta}^j = (\mu^j_\eta)^{-1} \rho_{\eta}^j$. Then we can apply Lemma \ref{lem.Abst.spec} to $H = L^2_{\phi_{\eta,\e}}(\Omega_{\eta,\e})$ and $T = \mathscr{T}_{\eta,\e}$. Due to the existence of common spectral gap between $(\mu_{\eta,\e}^{N_1})^{-1}$ and $(\mu_{\eta,\e}^{N_1+1})^{-1}$, and between $(\mu_{\eta}^{N_1})^{-1}$ and $(\mu_{\eta}^{N_1+1})^{-1}$, by \eqref{est.inf.u-v} we have
\begin{equation}\label{est.rhoeta.approx}
    \inf_{v\in S_{\eta,\e}^{N_1} } \| \rho_{\eta}^j - v \|_{L^2_{\phi_{\eta,\e}}(\Omega_{\eta,\e})} = \| \rho_{\eta}^j - \mathcal{P}_{\eta,\e}^{N_1} \rho_{\eta}^j \|_{L^2_{\phi_{\eta,\e}}(\Omega_{\eta,\e})} \le C_k  \e^\frac12 \eta^{\frac{d-2}{2}}.
\end{equation}
In other words, in the Hilbert space $L^2_{\phi_{\eta,\e}}(\Omega_{\eta,\e})$, each homogenized eigenfunction $\rho^j_\eta$ with $1\le j\le N_1$ can be well-approximated by a linear combination of $\rho^i_{\eta,\e}$ with $1\le i\le N_1$. Due to the orthogonality, one then can show that each $\rho_{\eta,\e}^i$ can be well approximated by a linear combination of $\rho_{\eta}^j$ with $1\le i, j\le N_1$. Precisely, we have
\begin{lemma}\label{lem.subopt.rhoje}
    For each $1\le j\le N_1$,
    \begin{equation}\label{est.subopt.ef}
    \| \rho_{\eta,\e}^j - \widehat{\mathcal{P}}_{\eta}^{N_1} \rho_{\eta,\e}^j \|_{L^2_{\phi_{\eta,\e}}(\Omega_{\eta,\e})} \le C_k  \e^\frac12 \eta^{\frac{d-2}{2}},
\end{equation}
where
\begin{equation}
    \widehat{\mathcal{P}}_{\eta}^{N_1} f = \sum_{i=1}^{N_1} \Ag{f, \rho_{\eta}^i}_{\phi_{\eta,\e}} \rho_{\eta}^i.
\end{equation}
\end{lemma}
\begin{proof}
    By \eqref{est.rhoeta.approx}, there exist $\alpha_\ell^j = \Ag{\rho_\eta^j, \rho^\ell_{\eta,\e}}_{\phi_{\eta,\e}}$ such that
    \begin{equation}\label{est.rhoN1}
        \| \rho_{\eta}^j - \sum_{\ell=1}^{N_1} \alpha^j_{\ell} \rho_{\eta,\e}^\ell \|_{L^2_{\phi_{\eta,\e}}(\Omega_{\eta,\e})} \le C_k  \e^\frac12 \eta^{\frac{d-2}{2}}.
    \end{equation}
    It follows that for any $1\le i,j\le N_1$,
    \begin{equation}
        \Ag{\rho_\eta^i, \rho_\eta^j}_{\phi_{\eta,\e}} = \sum_{\ell,\tau = 1}^{N_1} \alpha_\ell^i \alpha_\tau^j \Ag{\rho^\ell_{\eta,\e}, \rho^\tau_{\eta,\e}} + O_k(\e^\frac12 \eta^\frac{d-2}{2}).
    \end{equation}
    Let $E = (\alpha_\ell^j)$ and $Q = \Ag{\rho_\eta^i, \rho_\eta^j}_{\phi_{\eta,\e}}$ (both are $N_1\times N_1$ matrices). Then the last equation can be written as
    \begin{equation}\label{eq.QEEt}
        Q = EE^T + O_k(\e^\frac12 \eta^\frac{d-2}{2}).
    \end{equation}
    We would like to show that $E$ is invertible and $|E^{-1}| \le C_k$. To this end, we notice that $\rho^j_\eta$ are almost mutually orthogonal in $L^2_{\phi_{\eta,\e}}(\Omega_{\eta,\e})$. In fact, by \eqref{est.rhoeta.alor}, we have
    \begin{equation}
        \bigg| \int_{\Omega_{\eta,\e}} \phi_{\eta,\e}^2 \rho_\eta^i \rho_\eta^j - \int_{\Omega} \rho_\eta^i \rho_\eta^j \bigg| \le C_k \e^\frac12 \eta^\frac{d-2}{2}.
    \end{equation}
    Since $\int_{\Omega} \rho_\eta^i \rho_\eta^j = \delta_{ij}$, we have
    \begin{equation}
        Q = I + O_k(\e^\frac12 \eta^\frac{d-2}{2}).
    \end{equation}
    Together with \eqref{eq.QEEt}, we have
    \begin{equation}
        EE^T = I + O_k(\e^\frac12 \eta^\frac{d-2}{2}).
    \end{equation}
    This implies that if $\e^\frac12 \eta^\frac{d-2}{2} < c_k$ for some sufficiently small $c_k$, then $E$ is invertible and 
    \begin{equation}
        E^{-1} = E^T (I +  O_k(\e^\frac12 \eta^\frac{d-2}{2}))^{-1} = E^T + O_k(\e^\frac12 \eta^\frac{d-2}{2}).
    \end{equation}
    Hence, by \eqref{est.rhoN1} and the last equation
    \begin{equation}\label{est.rhol.exp}
        \| \rho_{\eta,\e}^\ell - \sum_{j=1}^{N_1} \alpha^j_\ell \rho^j_\eta \|_{L^2_{\phi_{\eta,\e}}(\Omega_{\eta,\e})} \le C_k \e^\frac12 \eta^\frac{d-2}{2}.
    \end{equation}
    In view of the fact $\alpha_\ell^j = \Ag{\rho_\eta^j, \rho^\ell_{\eta,\e}}_{\phi_{\eta,\e}}$, \eqref{est.rhol.exp} is exactly \eqref{est.subopt.ef}.
\end{proof}

\begin{lemma}[Optimal almost orthogonality]\label{lem.OptAlmostOrtho}
For any $1\le j,\ell \le k$,
    \begin{equation}\label{est.AlmostOrtho-2}
    |\Ag{ \rho_{\eta,\e}^j, \rho_{\eta}^\ell }_{\phi_{\eta,\e}}| \le \frac{C_k \e \eta^\frac{d-2}{2} \mu_{\eta,\e}^j \mu_{\eta}^\ell }{|\mu_{\eta}^\ell - \mu_{\eta,\e}^j| }.
\end{equation}
\end{lemma}
\begin{proof}
    In \eqref{eq.AO-2}, we estimate right-hand side by
    \begin{equation}
    \begin{aligned}
        & |\Ag{\rho_{\eta,\e}^j, (\mathscr{T}_{\eta,\e}-\mathscr{T}_{\eta})(\rho_{\eta}^\ell) }_{\phi_{\eta,\e}}|  \\
        & =  |\Ag{\widetilde{\mathcal{P}}_{\eta}^{N_1} \rho_{\eta,\e}^j, (\mathscr{T}_{\eta,\e}-\mathscr{T}_{\eta})(\rho_{\eta}^\ell) }_{\phi_{\eta,\e}} | + |\Ag{\rho_{\eta,\e}^j - \widetilde{\mathcal{P}}_{\eta}^{N_1} \rho_{\eta,\e}^j, (\mathscr{T}_{\eta,\e}-\mathscr{T}_{\eta})(\rho_{\eta}^\ell) }_{\phi_{\eta,\e}}| \\
        & \le C_k \e \eta^\frac{d-2}{2} + C_k \e^\frac12 \eta^\frac{d-2}{2} \e^\frac12 \eta^\frac{d-2}{2} \\
        & \le C_k \e \eta^\frac{d-2}{2},
    \end{aligned}
    \end{equation}
    where we have used Theorem \ref{thm.optDualrate}, Theorem \ref{thm.SubOptRate} and \eqref{est.subopt.ef}.
\end{proof}

\subsection{Optimal lower bound of eigenvalues}
Recall that $S^N_{\eta,\e} = \text{span} \{ \rho_{\eta,\e}^j : j =1,\cdots, N\}$ and $S^N_\eta = \{ \rho_{\eta}^j: j=1,\cdots, N \}$. 
Previously, we have shown in Proposition \ref{prop.mue<mu0+} the optimal upper bound of the eigenvalues $\mu_{\eta,\e}^k$ for all $k \ge 1$, namely, 
\begin{equation}
    \mu^k_{\eta,\e} \le \mu^k_\eta + C_k \e \eta^{\frac{d-2}{2}}.
\end{equation}
In this subsection, we show the optimal lower bounds of the eigenvalues and thus complete the proof of Theorem \ref{thm.main-EV}.

\begin{proof}[Proof of Theorem \ref{thm.main-EV}]
It suffices to show
\begin{equation}\label{est.OptLower}
    \mu^k_\eta \le \mu_{\eta,\e}^k + C_k \e \eta^{\frac{d-2}{2}}.
\end{equation}
One can employ a similar argument as Subsection \ref{sec.5.2}. Here we use a different argument.
Consider $S_{\eta}[t] = \text{span}\{ \rho_{\eta}^j: \mu_{\eta}^j \le t  \}$.
To show \eqref{est.OptLower}, it suffices to show
\begin{equation}\label{est.dimSeta}
    \text{dim} S_{\eta}[\mu_{\eta,\e}^k + M_k \e \eta^{\frac{d-2}{2}}] \ge k,
\end{equation}
for sufficiently large $M_k$ independent of $\e$ and $\eta$. Since $S_{\eta,\e}^k$ has dimension $k$, if we can show that any element in $S_{\eta,\e}^k$ can be well approximated by the elements of $S_{\eta}[\mu_{\eta,\e}^k + M_k \e \eta^{\frac{d-2}{2}}]$, then the latter must have at least dimension $k$. This can be shown by the following two ingredients.

Ingredient 1: By \eqref{est.subopt.ef}, there exists $N_1 \le N_0 k$ such that for any $1\le j\le k$
\begin{equation}
    \| \rho_{\eta,\e}^j - \widehat{\mathcal{P}}_{\eta}^{N_1} \rho_{\eta,\e}^j \|_{L^2_{\phi_{\eta,\e}}(\Omega_{\eta,\e})} \le C_k \e^\frac12 \eta^\frac{d-2}{2}.
\end{equation}

Ingredient 2: By the optimal orthogonality in Lemma \ref{lem.OptAlmostOrtho}, for any $1\le j\le k$ and $\ell \le N_0 k$ with $\mu_{\eta}^i$ satisfying $\mu_{\eta}^\ell \ge \mu_{\eta,\e}^k + M \e \eta^{\frac{d-2}{2}}$, we have
\begin{equation}
     \Ag{ \rho_{\eta,\e}^j, \rho_{\eta}^\ell }_{\phi_{\eta,\e}} \le \frac{C_k'}{M},
\end{equation}
where $C'_k$ is a constant depending only on $k$ and $M$ is a large number to be chosen later.

From Ingredient 1 and Ingredient 2, we have, for any $M>0$,
\begin{equation}\label{est.rhoj-f}
    \max_{1\le j\le k} \inf_{f\in S_{\eta}[\mu_{\eta,\e}^k + M \e \eta^{\frac{d-2}{2}}]} \| \rho_{\eta,\e}^j - f \|_{L^2_{\phi_{\eta,\e}}(\Omega_{\eta,\e})} \le C_k \e^\frac12 \eta^\frac{d-2}{2} + \frac{N_0 k C_k'}{M}.
\end{equation}
Observe that if $\e^\frac12 \eta^\frac{d-2}{2}$ is sufficiently small and $M$ is sufficiently large depending on $k$, then the right-hand side of the above inequality can be small as well. 

Now we claim: there exists $\gamma_k>0$ such that if
\begin{equation}\label{est.approx.rhoe}
    \max_{1\le j\le k} \inf_{f\in S_{\eta}[\mu_{\eta,\e}^k + M_k \e \eta^{\frac{d-2}{2}}]} \| \rho_{\eta,\e}^j - f \|_{L^2_{\phi_{\eta,\e}}(\Omega_{\eta,\e})} \le \gamma_k,
\end{equation}
then
\begin{equation}
    \text{dim} S_{\eta}[\mu_{\eta,\e}^k + M_k \e \eta^{\frac{d-2}{2}}] \ge k.
\end{equation}
To prove the claim, we assume $\text{dim} S_{\eta}[\mu_{\eta,\e}^k + M_k \e \eta^{\frac{d-2}{2}}] = m$ and show $m\ge k$. Note that \eqref{est.approx.rhoe} implies that for each $1\le j\le k$, there exist $\alpha^j_\ell$ ($1\le \ell \le m$),
\begin{equation}
    \| \rho_{\eta,\e}^j - \sum_{\ell=1}^m \alpha^j_\ell \rho_\eta^\ell \|_{L^2_{\phi_{\eta,\e}}(\Omega_{\eta,\e})} \le \gamma_k.
\end{equation}
It follows
\begin{equation}\label{eq.rhoe.ij}
    \Ag{ \rho_{\eta,\e}^i, \rho_{\eta,\e}^j }_{\phi_{\eta,\e}} = \sum_{\ell,\tau=1}^m \alpha_{\ell}^i \alpha_{\tau}^j \Ag{\rho_\eta^\ell, \rho_\eta^\tau}_{\phi_{\eta,\e}} + O(\gamma_k).
\end{equation}
Let $E = (\alpha_\ell^i) \in \R^{k\times m}$ and $Q = (\Ag{\rho_\eta^\ell, \rho_\eta^\tau}_{\phi_{\eta,\e}}) \in \R^{m\times m}$.
Using the orthogonality of $\rho_{\eta,\e}^j$ in $L^2_{\phi_{\eta,\e}}(\Omega_{\eta,\e})$, \eqref{eq.rhoe.ij} gives
\begin{equation}
    I = EQE^T + O(\gamma_k).
\end{equation}
Note that $I$ is the $k\times k$ identity matrix. Thus $\text{Rank}(I - O(\gamma_k)) = k$ if we let $O(\gamma_k) < 1/k$. Hence, in this case $\text{Rank}(EQE^T) = k$ and therefore we must have $m \ge k$. This proves the claim. 

Consequently, for $\e \eta^\frac{d-2}{2}$ sufficiently small and $M = M_k$ sufficiently large (depending on $k$), the right-hand side of \eqref{est.rhoj-f} is bounded by $\gamma_k \le O(1/k)$ and thus the above claim implies \eqref{est.dimSeta}. The proof is complete.
\end{proof}

\subsection{Convergence of eigenfunctions}
In this subsection, we prove Theorem \ref{thm.main-EF}.

\begin{proof}[Proof of Theorem \ref{thm.main-EF}]
    Fix $k\ge 1$. Let $N_1 = N_1(k,\e,\eta)$ be given as in Lemma \ref{lem.N1gap}. It suffices to consider $t\in (0,1]$ since the case $t>1$ has no improvement compared to $t = 1$. We would like to show
    \begin{equation}\label{est.EF.conv}
        \| \rho^k_{\eta,\e} - \sum_{|\mu_\eta^j - \mu_\eta^k| \le t} \Ag{\rho_{\eta,\e}^k, \rho_\eta^j}_{\phi_{\eta,\e}} \rho_\eta^j \|_{L^2_{\phi_{\eta,\e}}(\Omega_{\eta,\e})} \le C_k (\e^\frac12 \eta^{\frac{d-2}{2}} \wedge \e + t^{-1} \e \eta^{\frac{d-2}{2}}).
    \end{equation}
    In view of $\psi_{\eta,\e}^k = \phi_{\eta,\e} \rho_{\eta,\e}^k$ from Proposition \ref{prop.2scale EF}, we see the above estimate implies the desired estimate \eqref{est.main-EF}.

    To show \eqref{est.EF.conv}, we consider two slightly different approaches leading two different estimates.

    Case 1: The first approach is based on Lemma \ref{lem.subopt.rhoje} and Lemma \ref{lem.OptAlmostOrtho}. In fact, by Lemma \ref{lem.subopt.rhoje} and the triangle inequality, we have
    \begin{equation}\label{est.rhoke.expan}
    \begin{aligned}
        &\| \rho^k_{\eta,\e} - \sum_{|\mu_\eta^j - \mu_\eta^k| \le t} \Ag{\rho_{\eta,\e}^k, \rho_\eta^j}_{\phi_{\eta,\e}} \rho_\eta^j \|_{L^2_{\phi_{\eta,\e}}(\Omega_{\eta,\e})} \\
        &\qquad \le \| \rho^k_{\eta,\e} - \sum_{1\le j\le N_1} \Ag{\rho_{\eta,\e}^k, \rho_\eta^j}_{\phi_{\eta,\e}} \rho_\eta^j \|_{L^2_{\phi_{\eta,\e}}(\Omega_{\eta,\e})}  \\
        & \qquad + \|  \sum\limits_{\substack{ |\mu_\eta^j - \mu_\eta^k| \le t\\ j\ge N_1+1}} \Ag{\rho_{\eta,\e}^k, \rho_\eta^j}_{\phi_{\eta,\e}} \rho_\eta^j \|_{L^2_{\phi_{\eta,\e}}(\Omega_{\eta,\e})} + \|  \sum\limits_{\substack{ |\mu_\eta^j - \mu_\eta^k| > t\\ 1\le j\le N_1}} \Ag{\rho_{\eta,\e}^k, \rho_\eta^j}_{\phi_{\eta,\e}} \rho_\eta^j \|_{L^2_{\phi_{\eta,\e}}(\Omega_{\eta,\e})}
    \end{aligned}
    \end{equation}
    The first term on the right-hand side is handled by \eqref{est.subopt.ef} with a bound $C_k \e^\frac12 \eta^\frac{d-2}{2}$. To estimate the second term, we recall that Lemma \ref{lem.N1gap} yields $\mu_\eta^{N_1+1} - \mu_\eta^{N_1} \ge c_k>0 $. Hence, for $j\ge N_1 + 1$, we have $|\mu_\eta^j - \mu_\eta^k| \ge c_k$. Then by Lemma \ref{lem.OptAlmostOrtho}, we obtain
    \begin{equation}\label{est.rhokj.ortho}
        |\Ag{\rho_{\eta,\e}^k, \rho_\eta^j}_{\phi_{\eta,\e}}| \le C_k \e \eta^\frac{d-2}{2}.
    \end{equation}
    Since we have assumed $t\le 1$, in view of the Weyl's law (Lemma \ref{lem.Weyl}), the number of eigenvalues $\mu_\eta^j$ dropping in $|\mu^j_\eta - \mu_\eta^k| \le t$ depends only on $k$. Hence, \eqref{est.rhokj.ortho} and the triangle inequality imply that the second term on the right-hand side of \eqref{est.rhoke.expan} is bounded by $C_k \e \eta^\frac{d-2}{2}$. Next, we estimate the last term in \eqref{est.rhoke.expan}. We use the fact $|\mu_\eta^j - \mu_\eta^k|>t$, Lemma \ref{lem.OptAlmostOrtho} and a similar argument to conclude that the last term of \eqref{est.rhoke.expan} is bounded by $C_k t^{-1} \e \eta^\frac{d-2}{2}$. Combining the above estimates, we obtain
    \begin{equation}\label{est.EF-case1}
        \| \rho^k_{\eta,\e} - \sum_{|\mu_\eta^j - \mu_\eta^k| \le t} \Ag{\rho_{\eta,\e}^k, \rho_\eta^j}_{\phi_{\eta,\e}} \rho_\eta^j \|_{L^2_{\phi_{\eta,\e}}(\Omega_{\eta,\e})} \le C_k (\e^\frac12 \eta^{\frac{d-2}{2}} + t^{-1} \e \eta^{\frac{d-2}{2}}).
    \end{equation}
    The proof of the first case is complete.

    Case 2: We use the maximum principle to estimate the boundary layer term $v_{\eta,\e}$ defined in \eqref{eq.def.bl}. The maximum principle has been proved in \eqref{est.maximum}. It follows that
    \begin{equation}
        \| v_{\eta,\e} \|_{L^\infty(\Omega_{\eta,\e})} \le \e \| \chi_{\eta,\e}\cdot \nabla u_\eta \|_{L^\infty(\Gamma_{\eta,\e})} \le C\e \| \nabla u_\eta \|_{W^{1,p}(\partial \Omega)},
    \end{equation}
    This together with Lemma \ref{lem.OptH1rate} gives
    \begin{equation}
        \| \mathscr{T}_{\eta,\e} f - \mathscr{T}_{\eta} f \|_{L^2_{\phi_{\eta,\e}}(\Omega_{\eta,\e})} \le C\e \| f \|_{W^{1,p}(\Omega)},
    \end{equation}
    for some $p>d$. This is an alternative estimate of Theorem \ref{thm.SubOptRate}, which, by the same argument as Lemma \ref{lem.subopt.rhoje} and Case 1 above, yields 
    \begin{equation}\label{est.EF-case2}
        \| \rho^k_{\eta,\e} - \sum_{|\mu_\eta^j - \mu_\eta^k| \le t} \Ag{\rho_{\eta,\e}^k, \rho_\eta^j}_{\phi_{\eta,\e}} \rho_\eta^j \|_{L^2_{\phi_{\eta,\e}}(\Omega_{\eta,\e})} \le C_k (\e + t^{-1} \e \eta^{\frac{d-2}{2}}).
    \end{equation}

    Finally, combining the estimates \eqref{est.EF-case1} and \eqref{est.EF-case2}, we obtain the desired estimate \eqref{est.EF.conv} and complete the proof.
\end{proof}

\bibliography{Mybib}
\bibliographystyle{plain}
\medskip
\end{document}